\documentclass[11pt]{article}
\usepackage{latexsym}
\usepackage{amsmath}
\usepackage{amssymb}
\usepackage{amsthm}
\usepackage{epsfig}
\usepackage{xcolor}
\usepackage{graphicx}
\usepackage{bm}
\usepackage{enumitem}
\usepackage{mathtools}
\usepackage{graphicx}
\usepackage{tikz}

\usepackage[top=1in, bottom=1in, left=1in, right=1in]{geometry}

\renewcommand{\P}{\mathbb{P}}
\newcommand{\E}{\mathbb{E}}
\newcommand{\Var}{\text{Var}}

\newcommand{\Z}{\mathbb{Z}}
\newcommand{\R}{\mathbb{R}}

\newcommand{\N}{\mathbb{N}}
\renewcommand{\S}{\mathbb{S}}

\newcommand{\eps}{\varepsilon} 

\def\id{{\mathbf I}}

\newcommand{\<}{\langle}
\renewcommand{\>}{\rangle}

\newcommand{\diag}{\text{diag}}
\newcommand{\tr}{\text{tr}}
\newcommand{\op}{{\rm op}}
\newcommand{\ones}{\bm{1}}

\def\sT{{\mathsf T}}

\DeclareMathOperator*{\argmin}{arg\,min}

\newtheorem{theorem}{Theorem}
\newtheorem*{theorem*}{Theorem}
\newtheorem{lemma}{Lemma}

\newtheorem{assumption}{Assumption}

\newtheorem{proposition}{Proposition}

\newtheorem{corollary}{Corollary}

\theoremstyle{definition}

\newtheoremstyle{myremark} % name
    {\topsep}                    % Space above
    {\topsep}                    % Space below
    {\rm}                        % Body font
    {}                           % Indent amount
    {\bf}                        % Theorem head font
    {.}                          % Punctuation after theorem head
    {.5em}                       % Space after theorem head
    {}  % Theorem head spec (can be left empty, meaning normal)

\theoremstyle{myremark}
\newtheorem{remark}{Remark}[section]

\usepackage[toc,page]{appendix}
\usepackage{graphicx}
\usepackage{bbm}

\DeclareSymbolFont{rsfs}{U}{rsfs}{m}{n}
\DeclareSymbolFontAlphabet{\mathscrsfs}{rsfs}

\def\bA{{\boldsymbol A}}
\def\bB{{\boldsymbol B}}
\def\bC{{\boldsymbol C}}
\def\bD{{\boldsymbol D}}
\def\bE{{\boldsymbol E}}
\def\bF{{\boldsymbol F}}

\def\bH{{\boldsymbol H}}

\def\bL{{\boldsymbol L}}
\def\bM{{\boldsymbol M}}

\def\bP{{\boldsymbol P}}
\def\bQ{{\boldsymbol Q}}
\def\bR{{\boldsymbol R}}
\def\bS{{\boldsymbol S}}

\def\bU{{\boldsymbol U}}
\def\bV{{\boldsymbol V}}

\def\bX{{\boldsymbol X}}
\def\bY{{\boldsymbol Y}}
\def\bZ{{\boldsymbol Z}}

\def\ba{{\boldsymbol a}}

\def\be{{\boldsymbol e}}
\def\boldf{{\boldsymbol f}}

\def\bh{{\boldsymbol h}}

\def\bu{{\boldsymbol u}}
\def\bv{{\boldsymbol v}}
\def\bw{{\boldsymbol w}}
\def\bx{{\boldsymbol x}}
\def\by{{\boldsymbol y}}
\def\bz{{\boldsymbol z}}

\def\bmu{{\boldsymbol \mu}}
\def\bbeta{{\boldsymbol \beta}}
\def\bdelta{{\boldsymbol\delta}}
\def\beps{{\boldsymbol \eps}}

\def\bphi{{\boldsymbol \phi}}
\def\btheta{{\boldsymbol \theta}}

\def\bDelta{{\boldsymbol \Delta}}

\def\bPhi{{\boldsymbol \Phi}}
\def\bSigma{{\boldsymbol \Sigma}}

\def\bPi{{\boldsymbol \Pi}}

\def\hba{{\hat {\boldsymbol a}}}
\def\hf{{\hat f}}
\def\ha{{\hat a}}

\def\bbHe{{\rm He}}

\def\de{{\rm d}}

\def\Coeff{{\rm Coeff}}
\def\de{{\rm d}}
\def\Unif{{\rm Unif}}

\def\bbHe{{\rm He}}

\def\cV{{\mathcal V}}

\def\cQ{{\mathcal Q}}

\def\cE{{\mathcal E}}

\def\cV{{\mathcal V}}

\def\cH{{\mathcal H}}
\def\cA{{\mathcal A}}

\def\tbA{\Tilde \bA}

\def\Unif{{\sf Unif}}
\def\normal{{\sf N}}
\def\proj{{\mathsf P}}

\def\reals{{\mathbb R}}
\def\integers{{\mathbb Z}}
\def\naturals{{\mathbb N}}

\def\normal{{\sf N}}
\def\proj{{\mathsf P}}

\def\Unif{{\sf Unif}}
\def\normal{{\sf N}}

\def\proj{{\mathsf P}}

\def\reals{{\mathbb R}}
\def\integers{{\mathbb Z}}
\def\naturals{{\mathbb N}}
\def\proj{{\mathsf P}}

\def\tQ{\tilde{Q}}

\def\hba{{\hat {\boldsymbol a}}}
\def\hf{{\hat f}}

\def\cE{{\mathcal E}}

\def\Coeff{{\rm Coeff}}
\def\de{{\rm d}}
\def\Unif{{\rm Unif}}

\def\cE{{\mathcal E}}

\def\normal{{\sf N}}

\def\bDelta{{\boldsymbol \Delta}}

\def\Cube{{\mathscrsfs Q}}

\def\Coeff{{\rm Coeff}}

\def\bA{{\boldsymbol A}}
\def\btheta{{\boldsymbol \theta}}

\def\Tr{{\rm Tr}}

\def\cV{{\mathcal V}}
\def\bP{{\boldsymbol P}}
\def\diag{{\rm diag}}
\def\bS{{\boldsymbol S}}

\def\bD{{\boldsymbol D}}

\def\bL{{\boldsymbol L}}

\def\hf{\hat f}

\def\bR{{\boldsymbol R}}

\def\cuH{\mathscrsfs{H}}

\def\bC{{\boldsymbol C}}

\def\barsigma{\bar{\sigma}}

\def\oproj{{\overline \proj}}

\def\ind{\mathbbm{1}}

\def\tC{\Tilde C}
\def\tQ{\Tilde Q}
\def\balpha{\boldsymbol{\alpha}}
\def\bgamma{\boldsymbol{\gamma}}
\def\cU{\mathcal{U}}
\def\tbC{\Tilde \bC}
\def\tba{\Tilde \ba}
\def\tbeta{\Tilde \beta}
\def\tbbeta{\Tilde \bbeta}
\def\boldf{\boldsymbol{f}}
\def\bXi{\boldsymbol{\Xi}}
\def\cB{\mathcal{B}}
\def\MP{{\rm MP}}

\usepackage{hyperref}
\hypersetup{
    colorlinks,
    linkcolor={blue!80!black},
    citecolor={green!50!black},
}
\colorlet{linkequation}{blue}

\begin{document}

\title{Spectrum of inner-product kernel matrices in the polynomial regime and multiple descent phenomenon in kernel ridge regression}

\author{Theodor Misiakiewicz\thanks{Department of
    Statistics, Stanford University} \;\;
%   Andrea Montanari\footnotemark[1] \thanks{Department of Electrical Engineering,
%     Stanford University}, \;\;
   % Basil Saeed\thanks{Department of Electrical Engineering,
  %  Stanford University} 
  }

\maketitle

\begin{abstract}
    
    We study the spectrum of inner-product kernel matrices, i.e., $n \times n$ matrices with entries $h (\< \bx_i , \bx_j \>/d)$ where the $(\bx_i)_{i \leq n}$ are i.i.d.~random covariates in $\R^d$. In the linear high-dimensional regime $n \asymp d$, it was shown that these matrices are well approximated by their linearization, which simplifies into the sum of a rescaled Wishart matrix and identity matrix. In this paper, we generalize this decomposition to the polynomial high-dimensional regime $n \asymp d^\ell,\ell \in \naturals$, for data uniformly distributed on the sphere and hypercube. In this regime, the kernel matrix is well approximated by its degree-$\ell$ polynomial approximation and can be decomposed into a low-rank spike matrix, identity and a `Gegenbauer matrix' with entries $Q_\ell (\<\bx_i , \bx_j \>)$, where $Q_\ell$ is the degree-$\ell$ Gegenbauer polynomial. We show that the spectrum of the Gegenbauer matrix converges in distribution to a Marchenko-Pastur law.

    This problem is motivated by the study of the prediction error of kernel ridge regression (KRR) in the polynomial regime $n \asymp d^\kappa, \kappa >0$. Previous work showed that for $\kappa \not\in \naturals$, KRR fits exactly a degree-$\lfloor \kappa \rfloor$ polynomial approximation to the target function. In this paper, we use our characterization of the kernel matrix to complete this picture and compute the precise asymptotics of the test error in the limit $n/d^\kappa \to \psi$ with $\kappa \in \naturals$. In this case, the test error can present a double descent behavior, depending on the effective regularization and signal-to-noise ratio at level $\kappa$. Because this double descent can occur each time $\kappa$ crosses an integer, this explains the multiple descent phenomenon in the KRR risk curve observed in several previous works.

\end{abstract}

\tableofcontents

\section{Introduction}

Kernel methods are among the most popular tools in statistics and machine learning and have been extensively studied in the classical bias-variance trade-off setting \cite{berlinet2011reproducing,wainwright2019high}. Over the past few years, they have attracted a renewed interest because of their connection to neural networks in the `neural tangent kernel' regime \cite{jacot2018neural,li2018learning,du2018gradient,lee2019wide,allen2019convergence,chizat2019lazy}. Moreover, it was argued in \cite{belkin2018understand} that kernel methods share a number of surprising phenomena with deep learning, which are not explained by classical theory. This prompted a number of works to study kernel methods in the `overfitted regime', which brought to light several interesting behavior: near optimality of interpolators and benign overfitting \cite{liang2020just,ghorbani2021linearized,bartlett2020benign}, self-induced regularization \cite{ghorbani2021linearized,liang2020multiple} and double descent of the prediction risk \cite{mei2022generalization,hastie2022surprises}. These phenomena appear in the high-dimensional regime, when both the number of samples $n$ and the dimensionality $d$ of the data are large \cite{rakhlin2019consistency}, and are not captured by previous approaches such as capacity/source conditions \cite{caponnetto2007optimal}. This motivates the development of theory specific to kernel methods in high-dimension.

The seminal work \cite{el2010spectrum} studies the spectrum of inner-product kernel matrices in the linear high-dimensional regime $n \asymp d$ as $d,n \to \infty$. Consider an inner-product kernel function $H_d : \R^d \times \R^d \to \R$ induced by some function $h : \R \to \R$, i.e., $H_d (\bx_1, \bx_2 ) = h(\< \bx_1 , \bx_2 \> / d)$. Given $n$ i.i.d.~covariates $\{ \bx_i \}_{i \in [n]}$ with $\bx_i \in \R^d$, the empirical kernel matrix is given by
\begin{equation}\label{eq:kernel_mat}
\bH = \big(  h ( \< \bx_i , \bx_j \> / d ) \big)_{i,j \in [n]} \in \R^{n \times n}\, . 
\end{equation}
\cite{el2010spectrum} showed that when $n,d \to \infty$ with $n/d = \Theta (1)$, the random matrix $\bH$ can be approximated consistently in operator norm by its linearization $\overline{\bH}_{\text{lin}}$ (i.e., $\|\bH - \overline{\bH}_{\text{lin}} \|_\op \rightarrow 0  $ in probability):
\begin{equation}
    \overline{\bH}_{\text{lin}} = \alpha \ones \ones^\sT + \beta \frac{\bX \bX^\sT }{d} + \mu \id_n \, , 
\end{equation}
where $\bX = [ \bx_1 , \ldots , \bx_n ]^\sT \in \R^{n \times d}$ and, when covariates are isotropic,
\[
\alpha = h(0 ) + h'' (0)/(2d) \, , \qquad \beta = f '(0) \, , \qquad \mu = h(1) - h(0) - h' (0) \, .
\]

The linearization of the kernel matrix in the linear high-dimensional asymptotics was later used to bound the prediction error of kernel ridge regression (KRR) \cite{liang2020just,liu2021kernel,bartlett2021deep}. In particular, it was shown that KRR can learn at most a linear approximation to the target function in that regime. In order to study a more realistic scenario where $n \gg d$ with $n,d$ large, several works have proposed to consider a more general \textit{polynomial high-dimensional regime}, with $n \asymp d^\kappa$ for fixed $\kappa \in \R_{>0}$ as $d,n \to \infty$ \cite{ghorbani2021linearized,liang2020multiple,ghorbani2020neural,canatar2021spectral,mei21generalization}. The papers most relevant to our setting \cite{ghorbani2021linearized,mei21generalization} require the eigenvalues of the kernel operator to have a `spectral gap' in their analysis, and only apply to $\kappa \not\in \naturals$ for isotropic data (see Section \ref{sec:spectrum_Kernel_matrix}). In that case, they show that the kernel matrix \eqref{eq:kernel_mat} can be well approximated by its degree-$\lfloor \kappa \rfloor$ polynomial approximation $\overline{\bH}_{\text{poly-}\lfloor \kappa \rfloor} := \bH_{\leq \lfloor \kappa \rfloor} + \mu_{>\lfloor \kappa \rfloor} \id_n$, where $\bH_{\leq \lfloor \kappa \rfloor}$ is a degree-$\lfloor \kappa \rfloor$ polynomial in the Gram matrix $(\< \bx_i , \bx_j\>)_{ij \in [n]}$. In particular, the matrix $\bH_{\leq \lfloor \kappa \rfloor}$ is low-rank ($\text{rank} = \Theta_d (d^{\lfloor \kappa \rfloor}) \ll d^\kappa \asymp n$) with diverging non-zero eigenvalues (a `spike matrix'). \cite{ghorbani2021linearized,mei21generalization} use this approximation to show that KRR essentially works as a shrinkage operator in this regime, and fits a degree-$\lfloor \kappa \rfloor $ polynomial approximation to the target function. However, for $\kappa \in \naturals$, the matrix $\bH_{\leq  \kappa }$ is not low rank anymore and its analysis and the analysis of KRR remains open. Similarly, \cite{liang2020multiple} provides an upper bound on the variance of KRR with isotropic data, which vanishes when $\kappa \not\in \naturals$ and is vacuous for $\kappa \in \naturals$. They argue from simulation that such a behavior is to be expected as the test error of KRR can display peaks at $\kappa \in \naturals$.

The goal of this paper is to complete this picture and analyze the kernel matrix $\bH$ and the test error of KRR for $n \asymp d^\ell$, for any fixed\footnote{We will denote $\kappa \in \R_{>0}$ the general exponent for $n\asymp d^\kappa$, and prefer the notation $\kappa = \ell$ when $\kappa \in \naturals$.} $\ell \in \naturals$. We consider data uniformly distributed on either the sphere $\S^{d-1} (\sqrt{d}) = \{\bx \in \R^d: \| \bx \|_2 = \sqrt{d}\}$ of radius $\sqrt{d}$, or the hypercube $\Cube^d = \{ +1 , -1 \}^d$. We will write $\cA_d \in \{ \S^{d-1} (\sqrt{d}) , \Cube^d \}$ and $\bx \sim \Unif (\cA_d)$. 
While we present our results for these two simple data distributions, we note that all the results and proofs in this paper can be restated in the abstract setting of \cite{mei21generalization}, which only requires the eigenvalues and eigenfunctions of the kernel operator to follow some decay and concentration properties\footnote{The `spectral gap' condition mentioned above would be relaxed to a convergence condition on the eigenvalues of order $\lambda_i = \Theta (n^{-1})$, with the associated eigenfunctions verifying a condition similar to Proposition \ref{prop:vanishing_quadratic}.}. The drawback of this abstract setting is the difficulty of checking whether these conditions are verified in specific examples, which requires an exact eigendecomposition of the kernel and often a substantial amount of work (see examples in \cite{mei2021learning,misiakiewicz2021learning}).

The rest of the paper is organized as follows. We summarize our main results in Section \ref{sec:summary_results} and discuss related work in Section \ref{sec:related_work}.  In Section \ref{sec:spectrum_Kernel_matrix}, we present the polynomial approximation of the kernel matrix in the polynomial regime and show that its spectrum converges to a shifted and rescaled Marchenko-Pastur law. Finally, we compute in Section \ref{sec:KRR} the precise asymptotics of the test error of KRR with inner-product kernel when $n/d^\ell \to \psi$.

\subsection{Summary of main results}
\label{sec:summary_results}

\subsubsection{Spectrum of inner-product kernel matrices in the polynomial regime}

The spectral analysis of the kernel matrix is based on the explicit eigendecomposition of inner-product kernels on $\cA_d$, in terms of the orthogonal Gegenbauer polynomials $\{Q_k^{(d)} \}_{k \geq 0}$. \cite{ghorbani2021linearized,mei21generalization} showed that the high-degree part of the kernel matrix behaves as an isometry. Hence for $n = \Theta_d (d^\ell)$, the kernel matrix can be approximated consistently in operator norm by its degree-$\ell$ polynomial approximation 
%\begin{equation}
%h ( \< \bx_1 , \bx_2 \>) = \sum_{k =0}^\infty \mu_{d,k} Q^{(d)}_k ( \< \bx_1 , \bx_2 \> ) \, , 
%\end{equation}
%where the $\{Q_k^{(d)} \}_{k \geq 0}$ are the orthogonal Gegenbauer polynomials (on $\cA_d$), i.e., $Q^{(d)}_k : \R \to \R$ is a degree-$k$ polynomial and $\E_{\bx} [ Q^{(d)}_k (\< \bz_1, \bx \>) Q^{(d)}_l (\< \bz_2 , \bx \> ) ] = 0$ for $k \neq l$, $\bx \sim \Unif (\cA_d)$ and $\bz_1 , \bz_2 \in \cA_d$. For any $\ell \in \naturals$, the kernel matrix \eqref{eq:kernel_mat} can therefore be decomposed into
%\begin{equation}
%    \begin{aligned}
%    \bH =&~ \sum_{k =0}^\infty \mu_{d,k} \bQ_k = \bH_{\leq \ell - 1} + \mu_{d,\ell} \bQ_\ell + \bH_{>\ell} \, ,\\
%    \bH_{\leq \ell - 1} = &~ \sum_{k =0}^{\ell - 1} \mu_{d,k} \bQ_k\, , \qquad \bH_{>\ell}=\sum_{k =\ell +1}^{\infty} \mu_{d,k} \bQ_k \, ,
%    \end{aligned}
%\end{equation}
%where we will call $\bQ_k = ( Q^{(d)}_k ( \< \bx_i , \bx_j \> ) )_{ij \in [n]}$ the $k$-th Gegenbauer matrix. In \cite{ghorbani2021linearized,mei21generalization}, it was proved that for $n = \Theta_d (d^\ell)$, the high-degree part of the kernel is approximately an isometry $\E [ \| \bH_{>\ell} - \mu_{d,>\ell} \id_n \|_{\op}^2 ] = o_d(1) $ where $\mu_{d,>\ell} = \sum_{k \geq \ell +1} \mu_{d,k}$. Hence in the polynomial regime $n = \Theta_d ( d^\ell )$, the kernel matrix can be approximated consistently in operator norm by its degree-$\ell$ polynomial approximation 
\begin{equation}\label{eq:poly_approx_kernel_matrix}
    \overline{\bH}_{\text{poly-}\ell} = \bH_{\leq \ell - 1} + \mu_{d,\ell} \bQ_\ell + \mu_{d,>\ell} \id_n \, ,
\end{equation}
where $\bH_{\leq \ell - 1} = ( h_{\leq \ell -1} (\< \bx_i , \bx_j \>))_{i,j \in [n]}$ with $h_{\leq \ell -1}$ the degree-$(\ell - 1)$ polynomial approximation of $h$ in $L^2 (\cA_d)$, and $\bQ_\ell = ( Q^{(d)}_\ell ( \< \bx_i , \bx_j \> ) )_{ij \in [n]}$ with $Q_\ell^{(d)}$ the degree-$\ell$ Gegenbauer polynomial. The matrix $\bH_{\leq \ell - 1}$ has rank $\Theta_d (d^{\ell - 1})\ll n$ (the dimension of the space of degree-$(\ell-1)$ polynomials), with smallest non-zero eigenvalue $\lambda_{\min} (\bH_{\leq \ell - 1}) = \Omega_{d,\P}(d)$ for generic (universal) kernel $h$ \cite{ghorbani2021linearized}. Hence, $\bH_{\leq \ell - 1}$ corresponds to a low-rank spike matrix with diverging eigenvalues and whose eigenspace will align (in some sense) with the subspace $\cV_{\leq \ell -1}$ in $L^2 (\cA_d)$ of all polynomials of degree $\leq \ell - 1$.

On the other hand, $\bQ_\ell$ has rank $\min ( n , B(\cA_d , \ell) )$, where $B(\cA_d , \ell) = \Theta_d ( d^\ell)$ denote the dimension of the subspace $\cV_\ell$ of degree-$\ell$ polynomials orthogonal to $\cV_{\leq \ell - 1}$. We show that its spectrum converges to a Marchenko-Pastur distribution:

\begin{theorem*}[Spectrum of $\bQ_\ell$] Let $\nu_{\MP,\psi}$ be the Marchenko-Pastur distribution with aspect ratio $\psi >0$, and $\hat \nu_n$ the empirical spectral distribution of $\bQ_\ell$. Then $\hat \nu_n$ converges in distribution to $\nu_{\MP,\psi}$, almost surely as $\lim_{d \to \infty} \frac{n}{B(\cA_d,\ell)} = \psi$.
\end{theorem*}

This theorem with the decomposition \eqref{eq:poly_approx_kernel_matrix} shows that the spectrum of the kernel matrix converges to a shifted and rescaled Marchenko-Pastur distribution when $n = \Theta_d (d^\ell), \ell \in \naturals$. Recall that such a result was only known to hold for $\ell = 1$ \cite{el2010spectrum}. The proof of the theorem relies on rewriting the Gegenbauer matrix as a covariance matrix of a certain polynomial mapping of the $\bx_i$'s, and a sufficient condition from \cite{yaskov2016necessary} for matrices with dependent entries to satisfy a Marchenko-Pastur theorem (see Section \ref{sec:spectrum_Kernel_matrix}).

As mentioned in the introduction, the results described in this paper hold in the abstract setting of \cite{mei21generalization}, with the added assumption that the eigenfunctions associated to eigenvalues of order $\Theta (n^{-1})$ obey the condition in \cite{yaskov2016necessary}. For example, this can be proved with little added work for the cyclic invariant kernel considered in \cite{mei2021learning} and the convolutional kernels with patch-size $q$ in \cite{misiakiewicz2021learning} (with the degree-$\ell$ polynomial approximation to the kernel matrix holding for $n = \Theta_d (d^{\ell - 1})$ and $n = \Theta_d (d q^{\ell - 1} )$ respectively).
A more challenging setting would be to show directly such a result in the setting of anisotropic sub-Gaussian data \cite{el2010spectrum,liang2020just}, without an explicit access to the eigendecomposition.

\subsubsection{Precise asymptotics of KRR prediction error in the polynomial regime}

\begin{figure}[t]
\begin{tikzpicture}

\node[inner sep=0pt] (russell) at (0,0)
    {\includegraphics[width=.7\textwidth]{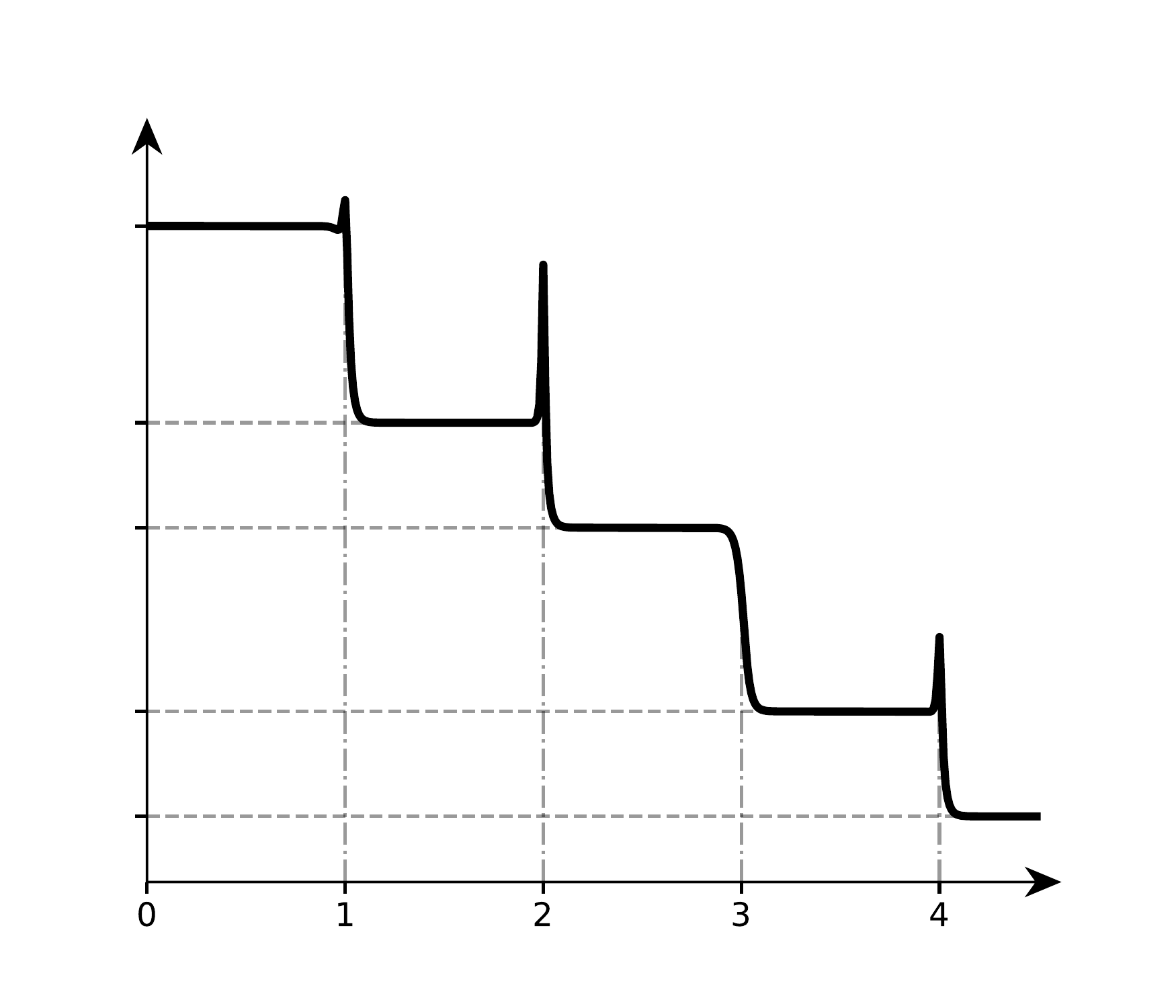}};
\node[inner sep=0pt] (whitehead) at (7,2.5)
    {\includegraphics[width=.48\textwidth]{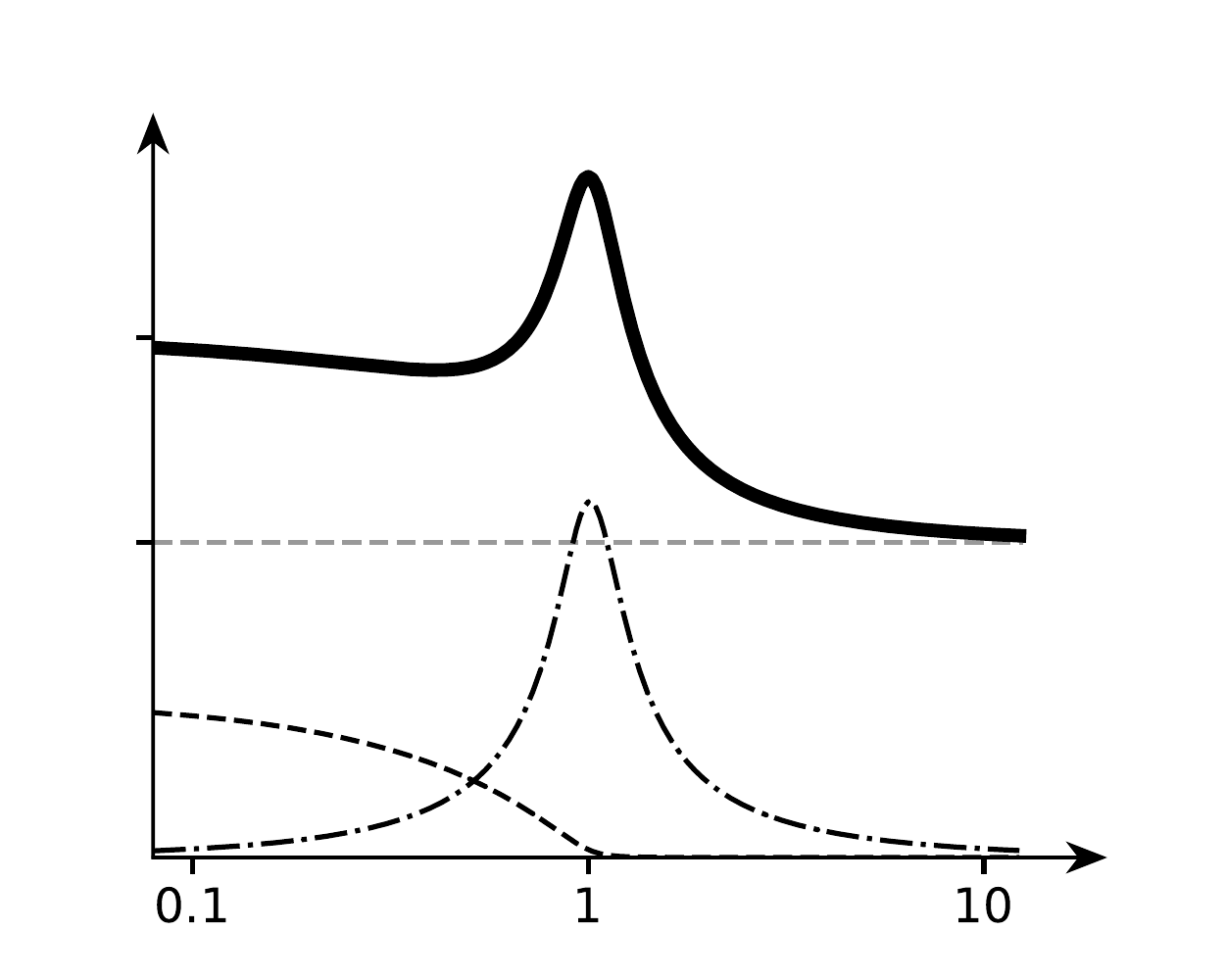}};

\node[rectangle,draw,ultra thick,align=center] (a) at (-1,5) {\large Test error of KRR \\
in the polynomial regime};
 \node[rectangle] (a) at (-5.4,2.7) {$\| \proj_{>0} f_* \|_{L^2}^2$};
 \node[rectangle] (a) at (-5.4,0.8) {$\| \proj_{>1} f_* \|_{L^2}^2$};
 \node[rectangle] (a) at (-5.4,-0.2) {$\| \proj_{>2} f_* \|_{L^2}^2$};
 \node[rectangle] (a) at (-5.4,-2) {$\| \proj_{>3} f_* \|_{L^2}^2$};
 \node[rectangle] (a) at (-5.4,-3.1) {$\| \proj_{>4} f_* \|_{L^2}^2$};
 
 \node[rectangle] (a) at (5.2,-3.7) {$\frac{\log(n)}{\log(d)}$};
 \node[rectangle] (a) at (0,-4.5) {$\kappa$};
 
  \node[rectangle] (a) at (7,-0.5) {$\psi$};
  
\node[rectangle] (a) at (10,0.5) {\footnotesize $\frac{n}{B(\cA_d,\ell)}$};

\node[rectangle] (a) at (2.6,3.5) {\footnotesize $ \| \proj_{>\ell-1} f_* \|_{L^2}^2$};

\node[rectangle] (a) at (4.9,2.45) {\footnotesize $ \| \proj_{>\ell} f_* \|_{L^2}^2$};

\node[rectangle] (a) at (4.9,1.25) {\footnotesize $ \cB(\psi,\zeta_\ell)$};

\node[rectangle] (a) at (7.9,1.25) {\footnotesize $ \cV(\psi,\zeta_\ell)$};
 
\draw (3.5, -1.8) circle (0.8);
\draw (2.74, -1.54) -- (3.9, 4.9);
\draw (3.9, -2.49) -- (10.2, -0.3);

%\node[circle,draw]{} at (3.6,-2)
\end{tikzpicture}
\vspace{-20pt}
\caption{A cartoon illustration of the test error of KRR in the polynomial regime $n/d^\kappa \to \psi$, as $n,d \to \infty$, for any $\kappa,\psi \in \R_{>0}$. The test error follows a staircase, with peaks that can occur at each $\kappa =\ell \in \naturals$, depending on the effective regularization $\zeta_\ell$ and effective signal-to-noise ratio $\text{SNR}_\ell$ at level $\ell$ (a peak occurs if $\zeta_\ell$ and/or $\text{SNR}_\ell$ are small enough). \label{fig:illustation}}

\end{figure}
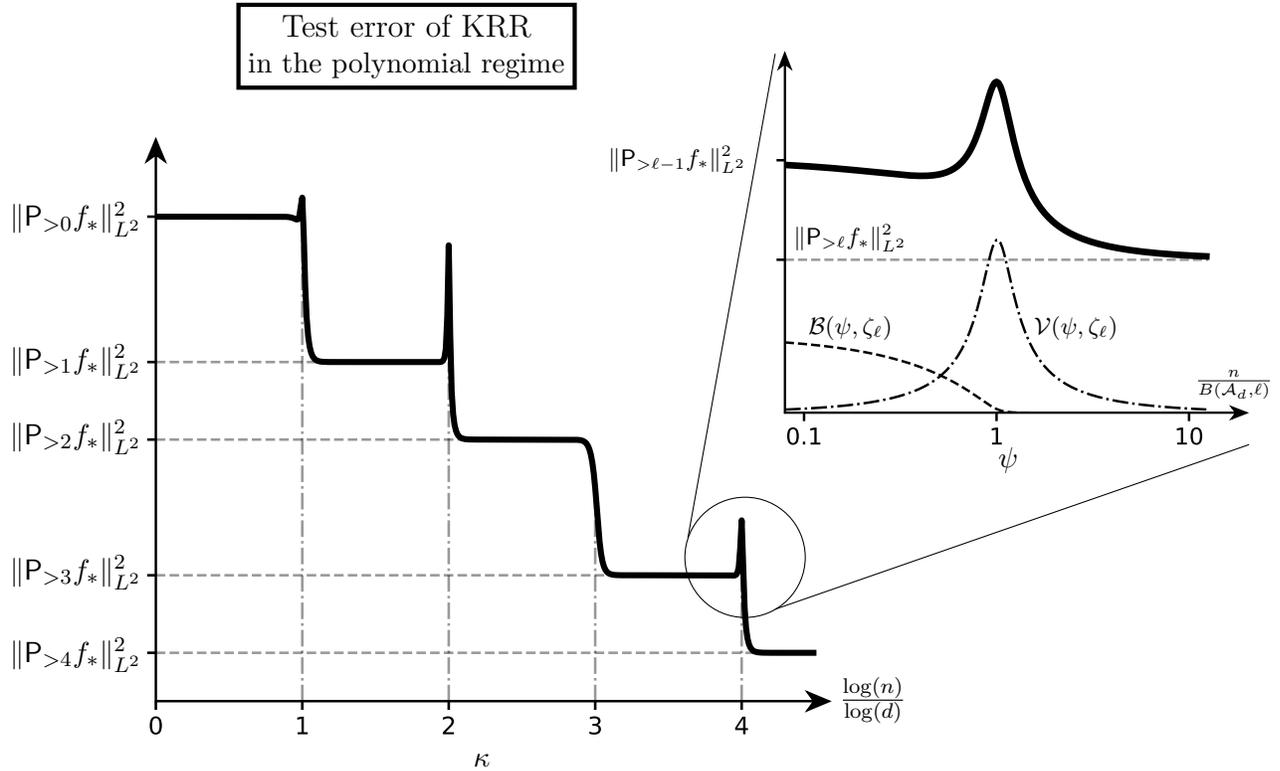

As an application, we consider kernel ridge regression (KRR) in the polynomial regime. We observe $n$ i.i.d. pairs $(y_i , \bx_i)_{i \in [n]} $ with covariates $\bx_i \sim \Unif (\cA_d)$ and responses $y_i = f_* (\bx_i) + \eps_i$, where the target function is chosen $f_* \in L^2 (\cA_d)$ and $\eps_i$ are independent noise with mean $0$ and variance $\E [ \eps_i^2 ] = \sigma_{\eps}^2$. The KRR solution with inner-product kernel $H_d(\bx_1,\bx_2) := h (\< \bx_1 , \bx_2 \>/d)$ and regularization parameter $\lambda >0$ is given by
\[
\hf_\lambda := \argmin_{f } \Big\{ \sum_{i = 1}^n \big( y_i - f ( \bx_i ) \big)^2 + \lambda \| f \|_\cH^2 \Big\} \, ,
\]
where $\| \cdot \|_\cH$ is the RKHS norm associated to kernel $H_d$ in $L^2 (\cA_d)$. The test error (or prediction error) of KRR is given by
\begin{equation*}
R_{\text{test}} ( f_* ; \bX , \beps, \lambda) := \E_{\bx} \Big[ \Big( f_* (\bx) - \hf_\lambda ( \bx) \Big)^2 \Big] \, .
\end{equation*}

Our goal is to characterize the asymptotic test error in the polynomial regime $n = \Theta_d (d^\kappa)$ for any fixed $\kappa \in \naturals$. In \cite{ghorbani2021linearized,mei21generalization}, it was proved that when $\kappa \not\in \naturals$, 
\begin{equation}\label{eq:kappa_notnaturals}
R_{\text{test}} ( f_* ; \bX , \beps, \lambda) = \| \proj_{> \lfloor \kappa \rfloor } f_* \|_{L^2}^2 + o_{d,\P}(1)\, , 
\end{equation}
where $\proj_{>\lfloor \kappa \rfloor}$ is the projection on the subspace orthogonal to polynomials of degree $\leq \lfloor \kappa \rfloor$, and we recall that $o_{d,\P}(\cdot)$ is the little-o in probability notation (i.e., a sequence of random variables $X_d = o_{d,\P}(1)$ if and only if $X_d \to 0$ in probability). In words, KRR fits the best degree-$\lfloor \kappa \rfloor$ polynomial approximation to the target function, and none of the high-degree part $\proj_{>\lfloor \kappa \rfloor} f_*$. For $\kappa = \ell \in \naturals$ and $n = \Theta_d ( d^\ell )$, KRR only partially fits  $\proj_\ell f_*$ (the projection on $\cV_\ell$ the subspace of degree-$\ell$ polynomials orthogonal to degree $\leq \ell - 1$ polynomials), while completely fitting $\proj_{\leq \ell -1}f_*$ (the degree-$(\ell -1 )$ approximation) and none of $\proj_{>\ell} f_*$. 

In the next theorem, we use the limiting spectral distribution of the kernel matrix to compute the precise asymptotics of the prediction error when $n/B(\cA_d,\ell) \to \psi$, where we recall that $B(\cA_d , \ell) = \Theta_d (d^\ell)$ is the  dimension of $\cV_\ell$. We assume that $h$ is a universal kernel with $\mu_{d,\ell} \to \mu_\ell >0$ and $\mu_{d,>\ell} \to \mu_{>\ell} >0$ (where $\mu_{d,\ell},\mu_{d,>\ell}$ are the coefficients in the decomposition~\eqref{eq:poly_approx_kernel_matrix}).

\begin{theorem*}[KRR test error in polynomial regime] Denote $\zeta_\ell = (\mu_{>\ell} + \lambda)/ \mu_\ell$ the effective regularization at level $\ell \in \naturals$. With the assumptions and $\cB(  \psi , \zeta  )$ and $\cV (  \psi , \zeta )$ defined in Theorem \ref{thm:KRR}, we have 
\begin{equation}\label{eq:asymp_test_intro}
R_{\text{test}} ( f_{*} ; \bX , \beps , \lambda) = \| \proj_\ell f_* \|_{L^2}^2 \cdot \cB(  \psi , \zeta_\ell  ) + ( \| \proj_{>\ell} f_* \|_{L^2}^2  + \sigma_\eps^2 ) \cdot \cV (  \psi , \zeta_\ell )  + \| \proj_{>\ell} f_* \|_{L^2}^2  + o_{d,\P}(1)\, ,
\end{equation}
as $ n / B(\cA_d , \ell) \to \psi$, where the convergence in probability is over the randomness in $\bX, \beps, f_*$.
\end{theorem*}

Let us make some comments about this asymptotic formula for the test error. First, it only depends on the kernel through an effective regularization $\zeta_\ell = (\mu_{>\ell} + \lambda)/ \mu_\ell$. We see that the high-degree part of the kernel plays the role of an effective self-induced regularization $\mu_{>\ell}$ which is added to the ridge parameter $\lambda$. In particular, $\zeta_\ell >0$ even when $\lambda \to 0^+$ (KRR solution interpolates the data) which explains the `benign overfitting' phenomenon in this model, i.e., the interpolating solution can still generalize well as noticed previously \cite{liang2020just,ghorbani2021linearized,liang2020multiple}. Secondly, the bias term $\cB (\psi, \zeta_\ell)$ is decreasing with $\psi$, while the variance term $\cV (\psi, \zeta_\ell)$ presents a peak at $\psi = 1$, with the value at the peak increasing as $\zeta_\ell$ decreases (see Figure \ref{fig:test_error_asymp}). Hence, a double descent in the test error can occur if $\zeta_\ell$ or the effective signal-to-noise ratio $\text{SNR}_\ell := \| \proj_\ell f_* \|_{L^2}^2/ ( \| \proj_{>\ell} f_* \|_{L^2}^2  + \sigma_\eps^2 )$ is sufficiently small. Note that this bias-variance decomposition is different than in the classical sense (only taking the label noise $\beps$): the high-degree part $\proj_{>\ell} f_*$ plays the role of an effective additive noise to the target function $\proj_{\leq \ell }f_*$, and we take here the bias variance decomposition over the `effective noise' $\proj_{>\ell} f_* (\bx_i) + \eps_i$. In particular, this means that a double descent can occur even without label noise $\eps_i = 0$ (no variance term in the classical sense).

Finally, the convergence to the asymptotic formula \eqref{eq:asymp_test_intro} is proven in probability over a class of `typical functions' in $L^2(\cA_d)$. The pointwise convergence result (for a fixed function $f_*$) holds if we assume that the random matrix $\bQ_\ell$ satisfies an isotropic local law \cite{alex2014isotropic}. Proving such a result for $\ell \geq 2$ (without independence of the entries of the feature matrix) is a significant challenge and is left for future work. See Section \ref{sec:pointwise} for a discussion.

With this theorem, we finish the task started in \cite{ghorbani2021linearized} and get a complete characterization of the prediction error of KRR in the polynomial regime, i.e., for $n/d^\kappa \to \psi$ for any $\kappa,\psi >0$, in the case of data uniformly distributed on the sphere and hypercube. Figure \ref{fig:illustation} illustrates these theoretical results. In particular, each time $\kappa$ crosses an integer, a peak can occur depending on $\zeta_\ell$ and $\text{SNR}_\ell$ at $\kappa = \ell \in \naturals$, which explains the multiple descent behavior with different sized peaks observed numerically in previous works \cite{liang2020multiple,canatar2021spectral}. KRR with an inner-product kernel and isotropic data offers a first natural example that rigorously shows a complex non-monotonic generalization curve, which has been observed in many machine learning studies \cite{chen2021multiple}.

\subsubsection{Equivalence with a Gaussian covariates model}\label{sec:Gaussian_equivalence}

A recent string of work started showing equivalence between non-linear regression models and simpler Gaussian covariates models in high-dimension \cite{mei2022generalization,goldt2020gaussian,hu2020universality}. These results hint at some general universality phenomena in high-dimensional models, where the test error only depends on the covariance of the features \cite{montanari2022universality} and a few properties of the non-linearity.

Here, we will simply make the following observation: the kernel ridge regression model has the same asymptotic prediction error in the polynomial regime as a simpler linear regression model with Gaussian covariates. Consider a target function:
\begin{equation}\label{eq:target_beta}
\begin{aligned}
&f_* (\bx) = \sum_{k =0 }^\infty \sum_{s \in [B(\cA_d,k)]} \beta_{ks} Y_{ks} (\bx) = \< \bPhi (\bx) , \btheta_* \>\, ,&\\
&\btheta_* = \left( \sqrt{\frac{B(\cA_d, k)}{\mu_{d,k}}}\beta_{ks}\right)_{k \geq 0, s \in [B(\cA_d, k )]} \, , \qquad \bPhi (\bx) = \left( \sqrt{\frac{\mu_{d,k}}{B(\cA_d, k)}}Y_{ks} (\bx) \right)_{k \geq 0, s \in [B(\cA_d, k )]} \, ,&
\end{aligned}
\end{equation}
where $\{Y_{ks}\}$ is the polynomial basis that diagonalizes inner-product kernels on $L^2(\cA_d)$, i.e., $Y_{ks}$ is an eigenfunction of the kernel operator with eigenvalue $\mu_{d,k} /B(\cA_d,k)$. Let us now state the equivalent linear regression model: we are given $n$ i.i.d.~pairs $(\bz_i , y_i )_{i \in [n]}$ with 
\begin{itemize}
    \item[1.] Covariates $\bz = (z_{ks})_{k \geq 0 , s \in [B(\cA_d,k)]} \in \R^\infty$ with $z_0 =\mu_{d,0}$ and $z_{ks} \sim \normal ( 0 , \mu_{d,k} / B(\cA_d , k) )$ independently.
    \item[2.] The linear response $y = \< \theta_* , \bz \> + \Tilde\eps $ with $\theta_*$ given in Eq.~\eqref{eq:target_beta} and independent noise $\Tilde\eps \sim \normal (0,\sigma_\eps^2)$.
\end{itemize}
We fit this model using ridge regression with ridge parameter $\lambda >0$:
\begin{equation}\label{eq:G_model}
\hat \btheta_\lambda = \argmin_{\btheta} \Big\{ \| \by - \bZ \btheta \|_2^2 + \lambda \| \btheta \|_2^2 \Big\} = \bZ^\sT ( \bZ \bZ^\sT + \lambda \id_n )^{-1} \by \, , 
\end{equation}
where $\bZ = [z_1 , \ldots , z_n ]^\sT \in \R^{n \times \infty}$ and $\by = (y_1 , \ldots , y_n)$. We denote the test error:
\[
R_{\text{Gauss}} ( \btheta_* ; \bZ , \Tilde \beps , \lambda) = \E_{\bz} \Big[ \big(\< \bz , \btheta_*\>  - \<\bz, \hat \btheta_\lambda \> \big)^2 \Big] \, .
\]
Such models were studied in the overfitted regime in \cite{bartlett2020benign,tsigler2020benign,richards2021asymptotics}. Here, we show that kernel ridge regression has the same asymptotic test error as the Gaussian covariates model \eqref{eq:G_model}:

\begin{theorem}[Gaussian equivalent model]\label{thm:Gaussian_equivalence} Under the same assumptions as Theorem \ref{thm:KRR}, for any $\kappa,\psi >0$ and $n/d^\kappa \to \psi$ as $d,n \to \infty$, we have
\begin{equation}
    R_{\text{test}} (f_* ; \bX , \beps , \lambda) = R_{\text{Gauss}} ( \btheta_* ; \bZ , \Tilde \beps , \lambda) + o_{d,\P}(1)\, .
\end{equation}
\end{theorem}

We remark that we can rewrite our inner-product kernel $ h(\< \bx_1 , \bx_2 \>/d) = \< \bPhi (\bx_1) , \bPhi (\bx_2) \>$ in terms of the feature map $\bPhi :\cA_d \to \R^\infty$ defined in Eq.~\eqref{eq:target_beta}. The entries of $\bPhi(\bx)$ are uncorrelated (by orthogonality of the $Y_{ks}$'s), $0$ mean and variance $\mu_{d,k}/ B(\cA_d, k)$ (except the first coordinate $Y_{00} = 1$). Theorem \ref{thm:Gaussian_equivalence} shows that in the polynomial high-dimensional regime, we can replace the entries of $\bPhi (\bx)$ by independent Gaussian variables with same covariance structure. This Gaussian covariates model (sometimes called `Gaussian design model') was used as a proxy to study kernel ridge regression with an implicit or explicit equivalence conjecture, or statistical physics heuristics \cite{jacot2020kernel,canatar2021spectral,cui2021generalization,cui2022error}. Of course  rigorously showing such an equivalence is difficult, and the point of this and previous papers: the entries of $\bPhi (\bx)$ are neither independent nor subgaussian, and require  specific tools to analyze, such as hypercontractivity and heavy-tailed matrix concentration \cite{mei21generalization}.

\subsection{Notations}

We denote $\kappa$ the general exponent in the polynomial regime $n = \Theta_d (d^\kappa)$, $\kappa \in \R_{>0}$ and use $\kappa = \ell \in \naturals$ when $\kappa \in \naturals$. For any $k \in \naturals$, $\cV_{\leq k}$ refers to the subspace of polynomials of degree $\leq k$ in $L^2 (\cA_d)$. We will further denote $\cV_{>k}$ the orthogonal complement of $\cV_{\leq k}$ and $\cV_k = \cV_{>k - 1} \cap \cV_{\leq k}$. We denote $\proj_{\leq k}$, $\proj_{>k}$ and $\proj_{k}$ the orthogonal projections onto $\cV_{\leq k}$, $\cV_{>k}$ and $\cV_k$ respectively. We introduce $B(\cA_d , k) = \dim (\cV_k)$ the dimension of $\cV_k$ and $\{Y_{k s} \}_{s\in [B(\cA_d, k)]}$ an orthonormal basis of $\cV_k$.

For a positive integer, we denote by $[n]$ the set $\{1 ,2 , \ldots , n \}$. For vectors $\bu,\bv \in \R^d$, we denote $\< \bu, \bv \> = u_1 v_1 + \ldots + u_d v_d$ their scalar product, and $\| \bu \|_2 = \< \bu , \bu\>^{1/2}$ the $\ell_2$ norm. Given a matrix $\bA \in \R^{n \times m}$, we denote $\| \bA \|_{\op} = \max_{\| \bu \|_2 = 1} \| \bA \bu \|_2$ its operator norm and by $\| \bA \|_{F} = \big( \sum_{i,j} A_{ij}^2 \big)^{1/2}$ its Frobenius norm. If $\bA \in \R^{n \times n}$ is a square matrix, the trace of $\bA$ is denoted by $\Tr (\bA) = \sum_{i \in [n]} A_{ii}$.

We use $O_d(\, \cdot \, )$  (resp. $o_d (\, \cdot \,)$) for the standard big-O (resp. little-o) relations, where the subscript $d$ emphasizes the asymptotic variable. Furthermore, we write $f = \Omega_d (g)$ if $g(d) = O_d (f(d) )$, and $f = \omega_d (g )$ if $g (d) = o_d (f (d))$. Finally, $f =\Theta_d (g)$ if we have both $f = O_d (g)$ and $f = \Omega_d (g)$. We will sometimes write $f \asymp g$, instead, which will just mean that $f = \Theta_d (g)$.

We use $O_{d,\P} (\, \cdot \,)$ (resp. $o_{d,\P} (\, \cdot \,)$) the big-O (resp. little-o) in probability relations. Namely, for $h_1(d)$ and $h_2 (d)$ two sequences of random variables, $h_1 (d) = O_{d,\P} ( h_2(d) )$ if for any $\eps > 0$, there exists $C_\eps > 0 $ and $d_\eps \in \Z_{>0}$, such that
\[
\begin{aligned}
\P ( |h_1 (d) / h_2 (d) | > C_{\eps}  ) \le \eps, \qquad \forall d \ge d_{\eps},
\end{aligned}
\]
and respectively: $h_1 (d) = o_{d,\P} ( h_2(d) )$, if $h_1 (d) / h_2 (d)$ converges to $0$ in probability.  Similarly, we will denote $h_1 (d) = \Omega_{d,\P} (h_2 (d))$ if $h_2 (d) = O_{d,\P} (h_1 (d))$, and $h_1 (d) = \omega_{d,\P} (h_2 (d))$ if $h_2 (d) = o_{d,\P} (h_1 (d))$. Finally, $h_1(d) =\Theta_{d,\P} (h_2(d))$ if we have both $h_1(d) =O_{d,\P} (h_2(d))$ and $h_1(d) =\Omega_{d,\P} (h_2(d))$.

\section{Related work}
\label{sec:related_work}

The spectrum of kernel random matrices in high-dimension $n \asymp d$ was first studied in \cite{el2010spectrum}, which extends the analysis of covariance matrices $(\<\bx_i , \bx_j \>/d)_{ij \in [n]}$ started in \cite{marchenko1967distribution} to matrices with entries $f ( \<\bx_i , \bx_j \>/d)$ with $f$ twice differentiable on a neighborhood of $0$. This result was later generalized to non-smooth functions and other scalings of the kernel entries \cite{cheng2013spectrum,do2013spectrum}. In this paper, we consider instead general inner-product kernels in the polynomial high-dimensional regime with $n \asymp d^\ell$, which was not considered before. Note that \cite{el2010spectrum} allows for general anisotropic covariates with i.i.d. entries and bounded moments, while we restrict ourselves here to covariates uniformly distributed on the sphere and the hypercube (recall again that our analysis applies to a more abstract setting which requires having access to the kernel operator eigendecomposition). We consider that extending our results to the distributional assumptions of \cite{el2010spectrum} is an interesting and important problem. From a technical point, our proof reduces to showing a Marchenko-Pastur law for a matrix without independent entries. This problem has been studied in \cite{bai2008large,adamczak2011marchenko,pastur2011eigenvalue,o2012note,yaskov2016necessary} under different assumptions. In particular, \cite{yaskov2016necessary} considers a feature matrix with i.i.d.~isotropic rows and use a leave-one-out argument to show a necessary and sufficient condition for the Marchenko-Pastur theorem to hold. This condition states that a certain quadratic vanishes with high-probability, and is verified in our setting by Proposition \ref{prop:vanishing_quadratic}.

The prediction error of kernel ridge regression in the linear high-dimensional regime was studied in \cite{liang2020just,liu2021kernel,bartlett2021deep} using the linearization of the kernel in this regime \cite{el2010spectrum}. In particular, \cite{liang2020just} points out that the minimum RKHS norm interpolating solution (KRR with ridge penalty $\lambda \to 0^+$) can still generalize well. 
Another line of work considers linear regression models with Gaussian or sub-Gaussian covariates as in Eq.~\eqref{eq:G_model}, which are technically easier to study and allows to focus on the interaction between eigenvalue decay and target function in the prediction error \cite{bartlett2020benign,tsigler2020benign,richards2021asymptotics,cui2021generalization}. 

The polynomial high-dimensional regime was first considered in \cite{ghorbani2021linearized}. They take data uniformly distributed on the $d$-dimensional sphere and show a staircase decay phenomena on the prediction error of inner-product kernels: for $n \asymp d^\kappa, \kappa \not\in \naturals$, KRR fits the best degree-$\lfloor \kappa \rfloor$ polynomial approximation to the target function. This computation was later extended to a general abstract framework in \cite{mei21generalization} which apply to kernel operators with top eigenfunctions verifying a hypercontractivity inequality and eigenvalues a certain spectral decay property (the number of eigenvalues such that $d^\delta/n \geq \lambda_i \geq d^{-\delta}/n$ is smaller than $n^{1-\delta}$, for some $\delta >0$). In this regime, \cite{mei21generalization} shows that KRR effectively acts as a shrinkage operator with some effective regularization $\lambda^{\text{eff}} > \lambda $: denoting $(\psi_{d,j})_{j \geq 1}$ and $\{ \lambda_{d,j} )_{j \geq 1}$ the eigenfunctions and eigenvalues of the kernel operator, then
\begin{equation}\label{eq:shrinkage}
f_* (\bx) = \sum_{j = 1}^\infty c_j \psi_{d,j}(\bx) \,\,\, \mapsto \,\,\, \hf_\lambda (\bx) \approx \sum_{j = 1}^\infty \frac{\lambda_{d,j}}{\lambda_{d,j} + \frac{\lambda^{\text{eff}}}{n}} c_j \psi_{d,j}(\bx)\, .
\end{equation}
Applied to inner-product kernels on the sphere or hypercube, the spectral decay property only hold for $\kappa \not\in \naturals$. In that case, $\lambda^{\text{eff}} = \Theta_d(1)$ and $\lambda_i \gg n^{-1}$ (polynomials of degree $\leq \lfloor \kappa \rfloor$, exactly learned) or $\lambda_i \ll n^{-1}$ (polynomials of degree $> \lfloor \kappa \rfloor$, not learned at all). This framework was applied to data distributed on anisotropic spheres in \cite{ghorbani2020neural}, invariant kernels in \cite{mei2021learning} and convolutional kernels in \cite{misiakiewicz2021learning}. Our present paper relaxes the spectral decay assumption to the case when the number of eigenvalues $\lambda_i = \Theta (n^{-1})$ is of order $n$, with additional conditions on the concentration of eigenfunctions and eigenvalues for a Marchenko-Pastur type theorem to hold. In that case, KRR still acts as a shrinkage operator for eigenspaces with $\lambda_i \ll n^{-1}$ or $\lambda_i \gg n^{-1}$. However, it presents a more complex behavior on the eigenspace associated to the $\lambda_i = \Theta (n^{-1})$.

\cite{liang2020multiple} also considers the polynomial asymptotic regime. The authors upper bound the prediction error of KRR for inner-product kernels and isotropic data with i.i.d.~entries and a tail condition. However, they require a strong condition on the target function to bound the bias term which converges to $0$ as soon as $\kappa \not\in \naturals$, and they do not recover the staircase decay from \cite{ghorbani2021linearized}. Finally, \cite{jacot2020kernel,canatar2021spectral} provides precise asymptotic predictions for the test error of general KRR which applies to the polynomial regime, using a Gaussian equivalence conjecture (Section \ref{sec:Gaussian_equivalence}) or statistical physics heuristics. 

The double descent phenomenon \cite{belkin2019reconciling} has now been well studied in regression settings: random feature models \cite{mei2022generalization,liao2020random,gerace2020generalisation}, linear models \cite{hastie2022surprises,kobak2020optimal,wu2020optimal,richards2021asymptotics} and KRR in the linear high-dimensional asymptotics \cite{bartlett2021deep,liu2021kernel}. The multiple-descent phenomenon in the KRR prediction error was observed empirically in \cite{liang2020multiple,canatar2021spectral}. In particular, \cite{liang2020multiple} shows an upper bound on the variance term of the prediction error that vanishes when $\kappa \not\in \naturals$ and is vacuous when $\kappa \in \naturals$. To the best of our knowledge, our work is the first to prove and precisely describe the multiple-descent phenomenon in KRR. Note that \cite{chen2021multiple} showed that linear regression models can be explicitly constructed to present several peaks in their prediction error. However, they do not prove this phenomenon for a natural learning model.

\section{Inner-product kernel matrices in the polynomial regime}
\label{sec:spectrum_Kernel_matrix}

Recall that the covariates are taken $\{ \bx_i \}_{i \in [n]}  \sim_{iid} \Unif ( \cA_d) $, where $\cA_d$ is either the hypercube $\Cube^d = \{+1,-1\}^d$ or the hypersphere $\S^{d-1} (\sqrt{d}) = \{ \bx \in \R^d: \| \bx \|_2 = \sqrt{d}\}$ in dimension $d$. We consider the inner-product kernel $H_d$ associated to the kernel function $h:[-1,1] \to \R$, defined by $H_d ( \bx_1 , \bx_2 ) = h ( \< \bx_1 , \bx_2 \>/d)$ for any $\bx_1,\bx_2 \in \cA_d$. Denote $\bH$ the empirical kernel matrix, that is
\[
\bH = ( H_d ( \bx_i , \bx_j ) )_{ij \in [n]} \in \R^{n \times n}\, .
\]
The goal of this section is to study the limiting spectral distribution of $\bH$ in the polynomial regime, i.e., when $n,d \to \infty$, with $\lim_{n,d \to \infty} \frac{n}{d^\kappa} = \psi$ for some $\kappa,\psi \in \R_{>0}$.

\subsection{Definitions and spectral decomposition of inner-product kernels}
\label{sec:defs}

We start by recalling some basic properties of functional spaces over $\cA_d$ (see Appendix \ref{sec:technical_background} for a complete exposition). Let $L^2 ( \cA_d) := L^2 (\cA_d, \Unif)$ be the space of square-integrable functions on $\cA_d$ with scalar product and norm denoted by $\< \cdot , \cdot \>_{L^2}$ and $\| \cdot \|_{L^2}$ given by
\[
\< f, g \>_{L^2} = \E_{\bx \sim \Unif ( \cA_d)} \big[ f(\bx) g ( \bx) \big]\, .
\]
For $\cA_d \in \{ \Cube^d, \S^{d-1} \}$, $L^2 ( \cA_d)$ admits the following orthogonal decomposition\footnote{For example, on the hypercube $\cA_d = \Cube^d$, the sum is over $0 \leq k \leq d$ where $V_{d,k}$ is the span of degree-$k$ Fourier basis, i.e., $V_{d,k} = \text{span} \{ \bx \mapsto \prod_{i \in S} x_i | S \subseteq [d], |S| =k \}$ and $\dim (V_{d,k}) = {{d}\choose{k}}$.}
\[
L^2 ( \cA_d) = \bigoplus_{k = 0}^\infty V_{d,k}\, , 
\]
where $ V_{d,k}$ is the subspace of polynomials of degree $k$ orthogonal to polynomials of degree $k-1$. Denote $\cV_{\leq \ell} = \bigoplus_{k=0}^\ell V_{d,k}$ the subspace of polynomials of degree $\leq \ell$, its orthogonal complement $\cV_{>\ell} = \bigoplus_{k = \ell+1}^\infty \cV_{d,k}$ and $B(\cA_d , k) = \dim ( V_{d,k} )$ (note that $B(\cA_d , k) = \Theta_d (d^k)$ for $k$ fixed). We will define $\proj_\ell$, $\proj_{\leq \ell}$ and $\proj_{>\ell} = \id  -\proj_{\leq \ell}$ the orthogonal projections onto $\cV_{d,\ell}$, $\cV_{d,\leq \ell}$ and $\cV_{d, >\ell}$ respectively.

Consider $\{ Y_{ks} \}_{s \in [B(\cA_d, k)]}$ an orthonormal basis of $V_{d,k}$. The $Y_{ks}$'s are degree-$k$ polynomials: they correspond to degree-$k$ spherical harmonics for $\cA_d = \S^{d-1}$ and Fourier (parity) functions for $\cA_d = \Cube^d$. We will call $\{ Y_{ks} \}_{k \geq 0 , s \in [B(\cA_d ,k )]}$ the polynomial basis of $\cA_d$. We further introduce the orthogonal basis $\{ Q^{(d)}_k \}_{k \geq 0}$ of Gegenbauer polynomials on $\cA_d$ (orthogonal with respect to the marginal measure $\< \ones, \bx\>$, with $\bx \sim \Unif(\cA_d)$), where $Q^{(d)}_k : [ - d, d ] \to \R$ is a degree-$k$ polynomial defined by
\[
Q_k^{(d)} ( \< \bx_1 , \bx_2 \>) = \frac{1}{B(\cA_d , k)} \sum_{k \in [B(\cA_d , k) ] } Y_{ks} (\bx_1 ) Y_{ks} (\bx_2 ) \, .
\]
(Note in particular, $\E [ Q_k^{(d)} ( \< \ones , \bx \>) Q_l^{(d)} ( \< \ones , \bx \>) ] = \delta_{kl} B (\cA_d , k)^{-1} $.)

For any inner-product kernel $H_d$ on $L^2 (\cA_d)$, there exists $h_d : [-1,+1] \to \R$ such that $H_d ( \bx_1 , \bx_2 ) = h_d ( \< \bx_1 , \bx_2 \>/d)$ (we allow $h_d$ here to depend on $d$). These kernels have the following simple eigendecomposition in the polynomial basis (see Appendix \ref{sec:technical_background}):
\begin{equation}\label{eq:diagonalization_h}
\begin{aligned}
 h_d (\< \bx_1 , \bx_2 \>/d) =&~ \sum_{k = 0}^\infty \xi_{d,k} (h) \sum_{s \in [B(\cA_d , k ) ]} Y_{ks} (\bx_1 ) Y_{ks} (\bx_2) =\sum_{k = 0}^\infty \mu_{d,k} (h) \cdot Q_{k}^{(d)} ( \< \bx_1 , \bx_2 \>)\, , \\
 \xi_{d,k} (h) =&~ \E_{\bx \sim \Unif (\cA_d)} \left[ h_d (x_1 / \sqrt{d}) Q_k^{(d)} ( \sqrt{d} x_1 ) \right] \, ,
 \end{aligned}
\end{equation}
where we denoted $\mu_{d,k} (h) = \xi_{d,k} (h) B(\cA_d , k)$. For ease of notation, we will write $\xi_{d,k} := \xi_{d,k} (h)$
and $\mu_{d,k} := \mu_{d,k}(h)$. Note that $\xi_{d,k} \geq 0$ for any $k$, by assumption of $H_d$ being positive semi-definite.

Let us introduce some further notation. For any $ \ell \in \naturals$, denote $B_{\leq \ell} = \sum_{k=0}^\ell B(\cA_d,k)$ and
\[
\begin{aligned}
\bY_{k} = &~ ( Y_{ks} (\bx_i) )_{i \in [n], s \in [B(\cA_d,k)]} \in \R^{n \times B(\cA_d,k)},\\
\bY_{\leq \ell} = &~ [\bY_{0}^\sT , \ldots , \bY_{\ell}^\sT ]^\sT \in \R^{n \times B_{\leq \ell}}, \\
\bD_{\leq \ell} = &~ \diag ( \xi_{d,0}\id_{B(\cA_d , 0)} , \ldots , \xi_{d,\ell} \id_{B(\cA_d , \ell )}) \in \R^{ B_{\leq \ell} \times  B_{\leq \ell}}  .
\end{aligned}
\]
We define the empirical Gegenbauer matrices $\bQ_k = ( Q_k^{(d)} ( \< \bx_i, \bx_j \>) )_{ij \in [n]}$. Note that with the above notations $\bQ_k = B(\cA_d , k)^{-1} \bY_k \bY_k^\sT$. We can therefore decompose the empirical kernel matrix in terms of the Gegenbauer matrices:
\[
\bH = \sum_{k \geq 0} \mu_{d,k} \bQ_k \, .
\]
For any $\ell \in \naturals$, it will be useful to introduce $\mu_{d,>\ell} = \sum_{k >\ell} \mu_{d,k} $ and 
\[
\begin{aligned}
\bH_{\leq \ell} =&~ \sum_{k=0}^{\ell}  \mu_{d,k} \bQ_k = \bY_{\leq \ell } \bD_{\leq \ell} \bY_{\leq \ell}^\sT\, , \qquad\qquad
\bH_{> \ell} = \sum_{k>\ell} \mu_{d,k} \bQ_k \, .
\end{aligned}
\]

\subsection{Limiting spectral distribution of the empirical kernel matrix}\label{sec:spectral}

We will make the following genericity assumption on the kernel functions $\{ h_d \}_{d \geq 1}$:

\begin{assumption}[Genericity condition on $\{ h_d \}_{d \geq 1}$ at level $\ell \in \naturals$]\label{ass:genericity_kernel}
We assume that the sequence of kernel functions $h_d: [-1,+1] \to \R$ obeys the following:
\begin{enumerate}
    \item[(a)] Low-degree $\xi_{d,k}$: there exists $\delta>0$ such that $\min_{k = 0 , \ldots, \ell -1} \xi_{d,k} = \Omega_d (d^{-\ell+\delta} )$.
    
    \item[(b)] High-degree $\xi_{d,k}$: there exists $\mu_{\ell}, \mu_{>\ell}>0$ such that
    \[
    \lim_{d \to \infty} \mu_{d,\ell} = \mu_{\ell}\, , \qquad \lim_{d \to \infty} \mu_{d, >\ell} = \mu_{>\ell}\, .
    \]
    
    \item[(c)] For $\cA_d = \Cube^d$, there exists $\delta >0$ such that  $\min_{k = 0 , \ldots, \ell} \mu_{d,k} = O_d (d^{-\ell-\delta} )$.
\end{enumerate}
\end{assumption}

Note that these assumptions are mild. Assumption \ref{ass:genericity_kernel}.$(a)$ is a quantitative version of a universality condition and is verified by generic kernels (e.g., if $h_d = h$ is smooth and universal, then $\lim_{d \to \infty} \mu_{d,k} = h^{(k)} (0)>0$ where $h^{(k)}$ is the $k$-th derivative of $h$). Assumption \ref{ass:genericity_kernel}.$(b)$ requires $h_d$ to have a non-vanishing high-degree part ($h_d$ is not a degree-$\ell$ polynomial). Assumption \ref{ass:genericity_kernel}.$(c)$ is simply added for ease of presentation and can be removed (see Remark \ref{rmk:hf_hypercube}). Note that Assumption \ref{ass:genericity_kernel}.$(c)$ is verified if $h_d = h$ is $(\ell+1)$-times differentiable (see Appendix D.2 in \cite{mei21generalization}).

\cite{ghorbani2021linearized,mei21generalization} analyzed the kernel matrix $\bH$ under a `spectral gap' condition and characterized its behavior for $n = \Theta_d (d^\kappa),\kappa \not\in \naturals$. More precisely, assume $\{h_d \}_{d\geq 1}$ verifies Assumption \ref{ass:genericity_kernel} at level $\lfloor \kappa \rfloor$. They showed the following:
\begin{enumerate}
    \item[1.] If $n = \omega_d (d^\ell \log (d) )$, then 
    \begin{equation}\label{eq:spiked_matrix}
        \E_{\bX} \big[ \| \bY_{\leq \ell}^\sT \bY_{\leq \ell} / n - \id_{B_{\leq \ell}} \|_{\op}^2 \big] = o_d(1) \, .
    \end{equation}
    
    \item[2.] If $n = O_d (d^{\ell+1} e^{-a_d \sqrt{\log d}}) $ for a sequence $a_d \to \infty$, then
    \begin{equation}
        \E_{\bX} \big[ \| \bH_{>\ell} - \mu_{>\ell} \id_n \|_\op^2 \big] = o_d (1) \, .
    \end{equation}
\end{enumerate}
From Assumption \ref{ass:genericity_kernel}.$(a)$, the matrix $\bH_{\leq \lfloor \kappa \rfloor} = \bY_{\leq \lfloor \kappa \rfloor} \bD_{\leq \lfloor \kappa \rfloor} \bY^\sT_{\leq \lfloor \kappa \rfloor}$ is a low dimensional matrix of rank $B_{\lfloor \kappa \rfloor} =  \Theta_d( d^{\lfloor \kappa \rfloor}) \ll n \asymp d^\kappa$ with minimum non-zero eigenvalue $\lambda_{\min} (\bH_{\leq \lfloor \kappa \rfloor}) = O_{d,\P}(1) \cdot \min_{k \leq \lfloor \kappa \rfloor} n \xi_{d,k} = \omega_{d,\P}(1)$. Hence we can decompose the kernel matrix (with $\E [\| \bDelta \|_\op^2] = o_{d}(1)$)
\begin{equation}\label{eq:Hn_decompo}
\bH = \bH_{\leq \lfloor \kappa \rfloor} + \bH_{> \lfloor \kappa \rfloor} = \bY_{\leq \lfloor \kappa \rfloor} \bD_{\leq \lfloor \kappa \rfloor} \bY^\sT_{\leq \lfloor \kappa \rfloor} + \mu_{>\lfloor \kappa \rfloor} \id_n + \bDelta \, ,
\end{equation}
as the sum of a low-rank spiked matrix with diverging eigenvalues plus a multiple of the identity matrix (the high-degree part of the kernel plays the role of a `self-induced ridge regularization'). In particular, \cite{ghorbani2021linearized,mei21generalization} used this decomposition to show that kernel ridge regression with any target function $f_* \in L^2 (\cA_d)$ learns exactly the projection $\proj_{\leq\lfloor \kappa \rfloor} f_*$ on degree-$\lfloor \kappa \rfloor$ polynomials and none of the high-degree part $\proj_{>\lfloor \kappa \rfloor} f_*$.% (i.e., $\proj_{\leq\lfloor \kappa \rfloor}$ and $\proj_{>\lfloor \kappa \rfloor} = \id - \proj_{\leq\lfloor \kappa \rfloor}$ are the orthogonal projections onto $\cV_{\leq \lfloor \kappa \rfloor}$ and $\cV_{> \lfloor \kappa \rfloor}$ respectively).

For $\kappa = \ell \in \naturals$, the degree-$\ell$ Gegenbauer matrix is neither low-dimensional nor concentrates on identity. In this case, the kernel matrix has the following decomposition (with $\E [\| \bDelta \|_{\op}^2] = o_{d} (1)$)
\[
\bH = \bH_{\leq \ell - 1} + \mu_{d,\ell} \bQ_\ell + \bH_{>\ell} = \bY_{\leq \ell - 1} \bD_{\leq \ell -1 } \bY_{\leq \ell - 1}^\sT  + \mu_{d,\ell} \bQ_\ell + \mu_{>\ell} \id_n + \bDelta\, ,
\]
where $\bH_{\leq \ell - 1}$ is a low-dimensional spiked matrix with diverging eigenvalues and $\bH_{>\ell}$ plays the role of a self-induced regularization. Hence, it only remains to study the spectrum of $\bQ_\ell$. 

Assume that $\lim_{d \to \infty} \frac{n}{B(\cA_d,\ell)} = \psi$ and  recall that we can write $\bQ_\ell= \bY_\ell \bY_\ell^\sT / B(\cA_d, \ell)$ where the random matrix $\bY_\ell = [ \bY_{\ell} (\bx_1) , \ldots , \bY_{\ell} (\bx_n) ]^\sT \in \R^{n \times B(\cA_d, \ell)}$ has rows $\bY_\ell (\bx_i ) = ( Y_{\ell s} (\bx_i) )_{s \in [B(\cA_d , \ell )]}$ which are independent and isotropic $\E [\bY_\ell (\bx_i)  \bY_\ell (\bx_i)^\sT ] = \id_{B(\cA_d , \ell)}$. The case where $\bZ\in \R^{n \times p}$ has iid entries is well understood and $\bZ\bZ^\sT/p$ is a Wishart-type matrix with limiting Marchenko-Pastur (MP) distribution. However, in our setting, each row $\bY_\ell (\bx_i)$ does not have independent entries. In fact, $\bY_\ell (\bx) \in \R^{B(\cA_d , \ell)}$ is a mapping from a $d$-dimensional space to a $\Theta_d (d^\ell)$-dimensional space. In general, even if a mapping is isotropic, it might present long-range dependencies and the empirical spectral distribution of the associated covariance matrix might not converge to a MP distribution\footnote{Consider for example $\cA_d = \Cube^d$ and $\bphi(\bx)= (Y_S (\bx) )_{S \subseteq [d-1],|S| = \ell}$ (all Fourier basis of degree $\ell$ not containing the last coordinate) and take $\bY (\bx) = \frac{1}{\sqrt{2}} (x_d - 1) \bphi(\bx)$. In that case, $\bY (\bx)$ is indeed an isotropic random vector but ${{d-1}\choose{k}}^{-1} \| \bY (\bx)\|_2^2$ does not converge in probability to $1$, which contradicts a necessary condition for convergence to the MP law \cite[Theorem 2.1]{yaskov2016necessary}.}. In our case, however, we will use that the entries of $\bY_\ell (\bx)$ are low-degree polynomials: each entry approximately only depends on a low-number of other coordinates, so that no such long-range dependencies occur and the spectrum of $\bQ_\ell$ indeed converges to a MP law. The proof of the next theorem formalizes this intuition.

Recall the definition of the Marchenko-Pastur distribution: given a parameter $\psi>0$,
    \begin{equation}\label{eq:MarchenkoPastur}
    \nu_{\MP,\psi} (\de x) = (1 - 1/\psi)_+ \delta_0 + \frac{1}{2\pi } \frac{\sqrt{(\lambda_+ - x) (x - \lambda_-)}}{\psi x} 1_{[\lambda_-, \lambda_+]} \de x,
    \end{equation}
    where $\lambda_{\pm} =(1 \pm \sqrt{\psi} )^2$.

    \begin{theorem}\label{thm:spectrum_Gegenbauer}
    Fix $\ell \in \naturals$ and $\psi >0$. Let $(\bx_i )_{i \in [n]} \sim_{iid} \Unif ( \cA_d)$ and consider
    \[
    \bQ_{\ell} = ( Q^{(d)}_\ell ( \< \bx_i , \bx_j \> ) )_{ij \in [n]} \in \R^{n \times n} \, ,
    \]
    the $\ell$-th Gegenbauer kernel matrix. Denote $\hat \nu_n$ the empirical spectral distribution of $\bQ_\ell$. Then $\hat \nu_n$ converges in distribution to $\nu_{\MP,\psi}$, almost surely as $\lim_{d \to \infty} \frac{n}{B(\cA_d,\ell)} = \psi$.
    \end{theorem}

\begin{figure}[t]
\begin{center}
    \includegraphics[width=16.5cm]{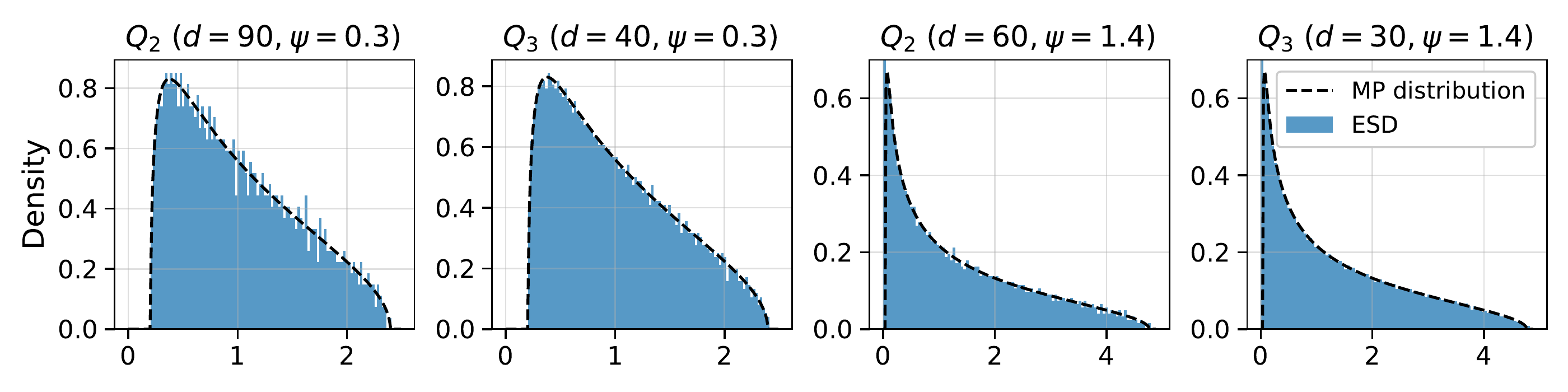}
\end{center}
\vspace{-15pt}

\caption{Empirical spectral density (ESD) for different Gegenbauer matrices and $\psi$ versus the Marchenko-Pastur density (black dashed curve). \label{fig:GegenMP}}
\end{figure}

See Figure \ref{fig:GegenMP} for a numerical illustration of this theorem. In Section \ref{thm:spectrum_Gegenbauer}, we will use Theorem \ref{thm:spectrum_Gegenbauer} and the decomposition \eqref{eq:Hn_decompo} of the kernel matrix to show that kernel ridge regression with $n = \Theta_d(d^\ell)$ samples and target function $f_* \in L^2 (\cA_d)$ learns exactly the projection $\proj_{\leq \ell -1} f_*$ on degree-$(\ell - 1)$ polynomials, partially the degree-$\ell$ component $\proj_\ell f_*$ and none of the high-degree part $\proj_{>\ell} f_*$.

\begin{remark}[High-frequency kernels on the hypercube]\label{rmk:hf_hypercube} As mentioned above, in the case of the hypercube, Assumption \ref{ass:genericity_kernel}.$(c)$ is verified for sufficiently smooth kernels. However, some natural non-smooth kernels can be constructed\footnote{For example, $h_d ( \<\bx_1 , \bx_2 \>/d) = \E_{\bw \sim \Unif (\Cube^d)} [(\<\bw,\bx_1 \>)_+ (\<\bw,\bx_2 \>)_+ ]$, the infinitely-wide linearized neural network associated to the ReLu activation function.} that do not satisfy Assumption \ref{ass:genericity_kernel}.$(c)$. In fact, we could have $\mu_{d,d-k} = \Omega_d(1)$ for all $k = 0 , \ldots, \ell$. In that case, the decomposition \eqref{eq:Hn_decompo} remains valid with the additional terms $\bH_{\geq d - \ell +1} + \mu_{d,d-\ell} \bQ_{d - \ell}$. The result \eqref{eq:spiked_matrix} still applies for $\bY_k$ replaced by $\bY_{d-k}$ and ridge regression would additionally learn exactly $\proj_{\geq d - \ell +1} f_*$. Furthermore, $\bQ_{d-\ell} = \bS \bQ_\ell \bS$ with $\bS = \diag (( Y_{[d]} (\bx_i) )_{i \in [n]})$ and $Y_{[d]} (\bx)=\prod_{i \in [d]} x_i$ is the full parity function. Hence $\bQ_{d - \ell}$ also verifies Theorem \ref{thm:spectrum_Gegenbauer} (however, the sum $\bQ_\ell+\bQ_{d - \ell}$ might be harder to analyze).
\end{remark}

\subsection{Proof of Theorem \ref{thm:spectrum_Gegenbauer}}

Several sufficient conditions for the Marchenko-Pastur law have been proved in the literature for matrices without independent entries \cite{bai2008large,adamczak2011marchenko,pastur2011eigenvalue,o2012note,yaskov2016necessary}. Here, we use a simple condition presented for example in \cite[Remark 2.2]{yaskov2016necessary} (see also \cite{bai2008large,pastur2011eigenvalue}), which applies to random matrices with iid isotropic rows, and which will directly imply Theorem \ref{thm:spectrum_Gegenbauer}.

\begin{proposition}\label{prop:vanishing_quadratic} Recall $\bx \sim \Unif(\cA_d)$ and $\bY_{\ell} ( \bx) = ( Y_{ks} (\bx) )_{s \in [B(\cA_d , \ell)]}$. Then, for all sequences $\{ \bA_d \}_{d \geq 1}$ of real symmetric positive semi-definite matrices $\bA_d \in \R^{B(\cA_d , \ell) \times B(\cA_d , \ell)}$ with $\| \bA_d \|_\op \leq 1$, we have 
\begin{equation}\label{eq:quad_form}
\cQ_d(\bx) = \frac{1}{B(\cA_d,\ell)} \bY_{\ell} (\bx)^\sT \bA_d \bY_{\ell} (\bx) - \frac{1}{B(\cA_d,\ell)} \Tr \left[  \bA_d \right] \stackrel{\P}{\longrightarrow} 0\, , \qquad \text{as }d \to \infty\, .
\end{equation}
\end{proposition}

Note that if the rows had independent entries, the variance of the quadratic form \eqref{eq:quad_form} would be simply bounded by $ B(\cA_d, \ell)^{-2} \Tr( \bA_d^2) \leq B(\cA_d, \ell)^{-1} \| \bA_d \|_\op^2 = o_d(1)$. In our setting, we will use that each entry of $\bY_\ell (\bx)$ is a low-degree polynomial of $\bx$ (hence only depends on a small number of coordinates of $\bx$): the vector $\bY_\ell (\bx)$ only presents a local dependency structure and the variance still vanishes. Bellow, we present the proof on the hypercube, which is particularly simple: we have access to a simple basis (parity functions) and the entries exactly depend on a low-number of coordinates.  On the sphere, we will use an explicit representation of the spherical harmonics in terms of generalized spherical coordinates and we defer the proof to Appendix \ref{app:proof_sphere}. In this case, each entry only approximately depends on a low-number of coordinates, and the analysis is much more involved.

\begin{proof}[Proof of Proposition \ref{prop:vanishing_quadratic}]
We consider the case $\bx \sim \Unif ( \Cube^d) $ and $\bY_\ell ( \bx )  = ( Y_S (\bx) )_{S \subset [d], |S| = \ell}$, where $Y_{S} (\bx) = \prod_{i\in S} x_i$ corresponds to the orthogonal Fourier basis. We will denote $B = B(\Cube^d,\ell)$ and $\bA = \bA_d$ for convenience.

Consider the variance of the quadratic form \eqref{eq:quad_form} (the sum is over $S_i \subseteq [d], |S_i | = \ell$):
\[
\begin{aligned}
\Var_\bx (\cQ_d (\bx) ) = &~\frac{1}{B^2}\Var_{\bx} \Big( \bY_{\ell} (\bx)^\sT \bA \bY_{\ell} (\bx) -  \Tr \left[  \bA \right] \Big) \\
=&~ \frac{1}{B^2} \sum_{S_1 \neq S_2 } \sum_{S_3 \neq S_4} \E_\bx [ Y_{S_1} (\bx) Y_{S_2} (\bx) Y_{S_3} (\bx) Y_{S_4} (\bx) ] A_{S_1 S_2} A_{S_3 S_4} = \frac{1}{B^2} \ba^\sT \bB \ba \, ,
\end{aligned}
\]
where we denoted $\ba \in \R^{B(B-1)}$ the vector containing all off-diagonal elements of $\bA$, and $\bB \in \R^{B(B-1) \times B(B-1)}$ is the matrix with entries
\[
B_{S_1S_2,S_3S_4 } = \E_\bx [ Y_{S_1} (\bx) Y_{S_2} (\bx) Y_{S_3} (\bx) Y_{S_4} (\bx) ] = \begin{cases}
1 & \text{if } S_1 \Delta S_2 = S_3 \Delta S_4 \, , \\
0 & \text{o.w.}\, .
\end{cases}
\]
Note that $\bB$ is a symmetric matrix and we have
\[
\| \bB \|_{\op} \leq \| \bB \|_{\infty,1} = \max_{S_1 \neq S_2} \sum_{S_3 \neq S_4} \ind_{S_1 \Delta S_2 = S_3 \Delta S_4} = O_d ( d^{\ell -1} )\, ,
\]
where we used that if $|S_1 \Delta S_2 | = 2k$ with $1 \leq k \leq \ell$, then there are at most ${{2k}\choose{k}} {{d}\choose{\ell -k}}$ ways of choosing $S_3,S_4 \subseteq [d], |S_3|=|S_4|= \ell$ such that $|S_3 \cap S_4 | = \ell - k$ and $S_3 \Delta S_4 = S_1 \Delta S_2$. We deduce that
\[
\Var_\bx (\cQ_d (\bx) ) \leq \frac{1}{B^2} \| \bB \|_\op \|\ba \|_2^2 = O_d(1) \cdot \frac{d^{\ell -1}}{B^2} \Tr ( \bA^2 ) \leq O_d(1) \cdot \frac{d^{\ell -1}}{B}\| \bA \|_\op^2 = o_d (1) \, , 
\]
where we used that $B = \Theta_d(d^\ell)$.
\end{proof}

\section{Application: kernel ridge regression in the polynomial regime}
\label{sec:KRR}

As an application of the results in Section \ref{sec:spectral}, we consider kernel ridge regression in the polynomial regime. We observe $n$ i.i.d. pairs $(y_i , \bx_i)_{i \in [n]} $ with covariates $\bx_i \sim \Unif (\cA_d)$ and responses $y_i = f_* (\bx_i) + \eps_i$, where the target function $f_* \in L^2 (\cA_d)$ and $\eps_i$ are i.i.d. noise with mean $0$ and variance $\E [ \eps_i^2 ] = \sigma_{\eps}^2$. We consider doing kernel ridge regression (KRR) with inner-product kernel $H_d(\bx_1,\bx_2) := h_d (\< \bx_1 , \bx_2 \>/d)$ and regularization parameter $\lambda >0$. KRR outputs a model $\hf ( \bx ; \hba_\lambda ) = \sum_{i \in [n]} \ha_{\lambda,i} h_d ( \< \bx , \bx_i \> / d)$ with
\[
\hba_\lambda := \argmin_{\ba \in \R^n} \Big\{ \sum_{i = 1}^n \big( y_i - \hf ( \bx_i ; \ba ) \big)^2 + \lambda \| \hf ( \cdot ; \ba ) \|_\cH^2 \Big\} = ( \bH + \lambda \id_n)^{-1} \by \, ,
\]
where $\| \cdot \|_\cH$ is the RKHS norm associated to kernel $H_d$ and we denoted $\by = (y_1 , \ldots , y_n)$. The test error (or prediction error) of KRR is given by
\begin{equation}
R_{\text{test}} ( f_\star ; \bX , \beps, \lambda) := \E_{\bx} \Big[ \Big( f_\star (\bx) - \by^\sT ( \bH + \lambda \id_n)^{-1} \bh ( \bx) \Big)^2 \Big] \, ,
\end{equation}
where $\bh (\bx) = ( h_d (\<\bx , \bx_i \> /d) )_{i \in [n]}$. We will also consider the training error and the RKHS norm of the KRR solution, which are given by
\begin{align}
    R_{\text{train}} ( f_\star ; \bX , \beps, \lambda) := &~ \frac{1}{n} \sum_{i =1}^n \big( y_i - \hf ( \bx_i ; \hba_\lambda ) \big)^2 = \frac{1}{n} \big\| \by - \bH ( \bH + \lambda \id_n)^{-1} \by \big\|^2\, ,\\
    \| \hf ( \cdot ; \hba_\lambda ) \|_\cH = &~ \hba_\lambda^\sT \bH \hba_\lambda \, .
\end{align}

In \cite{ghorbani2021linearized,mei21generalization}, the asymptotic test error was characterized for $n =\Theta_d (d^\kappa),\kappa \not\in \naturals$, and it was shown that $R_{\text{test}} ( f_\star ; \bX , \beps, \lambda) = \| \proj_{> \lfloor \kappa \rfloor} f_* \|_{L^2}^2 + o_{d,\P}(1)$. In this section, the goal is to compute the asymptotic test error for any fixed integer $\ell \in \naturals$, when $d,n \to \infty$ with $n / B(\cA_d ; \ell) \to \psi$.

\subsection{Asymptotic test error for typical target functions}

In this section, we characterize the asymptotic test error with high probability over a class of random target functions. We will discuss in Section \ref{sec:pointwise} how to extend these results to fixed target functions.

We assume the following distribution over the sequence of target functions $\{ f_{*,d} \in L^2 ( \cA_d ) \}_{d \geq 1}$:

\begin{assumption}[Function distribution at level $\ell \in \naturals$]\label{ass:dist_f}
Let $\{ f_{*,d} \in L^2 ( \cA_d ) \}_{d \geq 1}$ be a sequence of target functions with decomposition in the polynomial basis
\begin{equation*}
f_{*,d} (\bx) = \sum_{k =0 }^{\ell -1} \sum_{s \in [B(\cA_d,k)]}  \beta^*_{d,ks} Y_{ks} (\bx) +  \sum_{k \geq \ell } \sum_{s \in [B(\cA_d,k)]}  \tbeta_{d,ks} Y_{ks} (\bx)\, .
\end{equation*}
such that $\bbeta_{d}^* = (\beta^*_{d,ks})_{k \in [\ell-1], s \in [B(\cA_d, k)]}$ is a sequence of deterministic vectors with $\| \bbeta_{d}^* \|_2^2 = O_d(1)$, and $\tbbeta_d = (\tbeta_{d,ks})_{k \geq \ell, s \in [B(\cA_d,k)]}$ are sequences of zero-mean independent random variables with
\begin{equation*}
\E [ \tbeta_{d,ks}^2 ] = \frac{F_{d,k}^2}{B(\cA_d, k)}\, , \qquad \E [ \tbeta_{d,ks}^4 ] \leq C \frac{F_{d,k}^4}{B(\cA_d, k)^2}\, , 
\end{equation*}
where the sequences $(F_{d,k})_{k \geq \ell}$ verify
\[
\lim_{d \to \infty} F_{d,\ell}^2 = F_\ell^2 \, , \qquad \lim_{d \to \infty} \sum_{k \geq \ell+1} F_{d,k}^2 = F_{>\ell}^2 \, .
\]
for some positive constants $F_\ell,  F_{>\ell}> 0$.
\end{assumption}

We will denote $\E_{f_*}$ the expectation over the random sequence $\tbbeta_d = (\tbeta_{d,ks})_{k \geq \ell, s \in [B(\cA_d,k)]}$. We will further make the following technical assumption about the tail of the coefficients $(F_{d,k})_{k \geq \ell}$ which will simplify the proofs.

\begin{assumption}[Decay of $(F_{d,k})_{k \geq \ell}$ at level $\ell \in \naturals$]\label{ass:technical_ass}
Let $(F_{d,k})_{k \geq \ell}$ be the sequences of coefficients from Assumption \ref{ass:dist_f}. We will assume that there exists an integer $m \geq \ell$ and $\delta <1$ such that
\[
\sum_{k > m} \frac{F_{d,k}^2}{\mu_{d,k}} = O_d(d^\delta)\, ,
\]
where we recall that $\mu_{d,k} = \xi_{d,k} B(\cA_d,k)$ with $\xi_{d,k}$ the $k$-th Gegenbauer coefficient of $h_d$ (see Eq.~\eqref{eq:diagonalization_h}).
\end{assumption}

Assumption \ref{ass:technical_ass} implies that the RKHS norm of the high-degree part of the target functions does not diverge too quickly, i.e., $\E_{f_*} [ \| \proj_{> m} f_{*,d} \|_{\cH}^2] = \sum_{k > m} F_{d,k}^2 / \mu_{d,k} = O_d(d^\delta)$. This assumption is used to simplify the proof of the concentration of $R_{\text{test}} ( f_{*,d} ; \bX , \beps, \lambda)$ on $\E_{f_*}  [R_{\text{test}} ( f_{*,d} ; \bX , \beps, \lambda) ] $.

\begin{figure}[t]
\begin{center}
    \includegraphics[width=17cm]{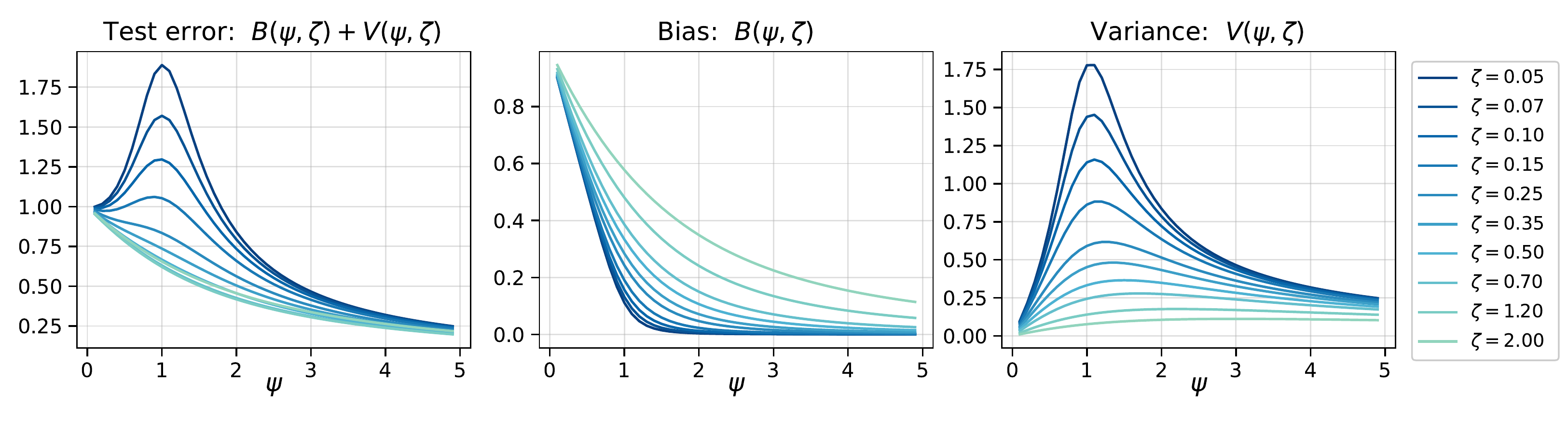}
\end{center}

\vspace{-19pt}

\caption{Asymptotic curves for the test error $R_{\text{test}} = \cB(\psi, \zeta)+\cV(\psi, \zeta)$, the bias part $\cB(\psi, \zeta)$ and the variance part $\cV(\psi, \zeta)$ for different values of the effective regularization $\zeta$.} \label{fig:test_error_asymp}
\end{figure}

We are now in position to state our main theorem:

\begin{theorem}[Asymptotic characterization of KRR in the polynomial regime]\label{thm:KRR}
Fix an integer $\ell \geq 1$ and $\psi >0$ and assume $\lim_{d \to \infty} n / B(\cA_d , \ell) = \psi$. Let $\{ h_d \}_{d \geq 1}$ be a sequence of inner-product kernel functions satisfying Assumption \ref{ass:genericity_kernel} at level $\ell$. Let $\{ f_{*,d} \in L^2 (\cA_d) \}_{d\geq 1}$ be a sequence of target functions satisfying Assumptions \ref{ass:dist_f} and \ref{ass:technical_ass} at level $\ell$. Denote 
\[
\begin{aligned}
\cB(  \psi ,\zeta) = &~ 1 - \psi + \psi \zeta^2 r' (- \zeta) \, , \\
\cV (  \psi , \zeta) =&~ \psi \big[ r_\psi (- \zeta) - \zeta r_\psi' (-\zeta)\big] \, ,
\end{aligned}
\]
where $r_\psi (-\zeta) = \int (x + \zeta)^{-1} \nu_{\MP,\psi} (\de x)$ is the Stieljes transform of the Marchenko-Pastur distribution (see Eq.~\eqref{eq:MarchenkoPastur}).

Let $\bX = (\bx_i)_{i \in [n]} \sim_{iid} \Unif(\cA_d)$ and $y_i = f_{*,d} (\bx_i) + \eps_i$ where $\eps_i$ independent zero-mean noise with$\E[ \eps_i^2] = \sigma_{\eps}^2$ and $\E[ \eps_i^4] < \infty$. Fix $\lambda >0$ and denote $\zeta_* = (\lambda + \mu_{>\ell} )/\mu_\ell$. Then we have
\begin{equation}\label{eq:asymp_test}
R_{\text{test}} ( f_{*,d} ; \bX , \beps , \lambda) \stackrel{\P}{\longrightarrow} F_\ell^2 \cdot \cB(  \psi , \zeta_*  ) + ( F_{>\ell}^2 + \sigma_\eps^2 ) \cdot \cV (  \psi , \zeta_* )  + F_{>\ell}^2 \, ,
\end{equation}
as $d \to \infty$, where the convergence in probability is over the randomness in $\bX, \beps , f_*$. Furthermore,
\begin{align}
    R_{\text{train}} ( f_{*,d} ; \bX , \beps , \lambda) \stackrel{\P}{\longrightarrow}&~  \frac{\lambda^2}{\mu_\ell^2} \Big\{ F_\ell^2 \big[ r_\psi (- \zeta_*) - \zeta_* r_\psi' (-\zeta_*) \big] + (F_{>\ell}^2 + \sigma_\eps^2) \cdot r_\psi' (-\zeta_*) \Big\} \, , \\
    \frac{1}{n} \| \hf ( \cdot ; \hba_\lambda ) \|_{\cH}^2 \stackrel{\P}{\longrightarrow}&~ \frac{F_{\ell}^2}{\mu_\ell} \Big\{ 1 - \zeta_* r_\psi ( - \zeta_* ) - \frac{\lambda}{\mu_\ell} \big[ r_\psi ( - \zeta_* )  - \zeta_* r_\psi ' (- \zeta_*) \big]  \Big\} \\
    &~+  \frac{F_{>\ell}^2 + \sigma_\eps^2}{\mu_\ell} \Big\{ r_\psi (- \zeta_*) - \frac{\lambda}{\mu_\ell} r_\psi ' ( - \zeta_* ) \Big\}\, .
\end{align}
%Denote $\zeta = \mu_{\ell} / (\ell! \mu_{>\ell})$, $r =   \xi^2 / ( \bar \lambda + 1 / \psi )$, $v = r (1 - 1/\psi) - 1$, $\omega = \frac{1}{2} \sqrt{v^2 + 4 r } + \frac{v}{2} $, and \notate{Change with the final formula}
%\[
%\begin{aligned}
%B( \xi , \psi , \bar \lambda ) = &~ \frac{1}{(\omega +1)^2 - \omega^2 / \psi}\, , \\
%V ( \xi , \psi , \bar \lambda ) =&~ \frac{\omega^2 / \psi}{(\omega +1)^2 - \omega^2 / \psi}\, .
%\end{aligned}
%\]
%Then for any regularization parameter $\lambda > 0$, the asymptotic prediction error of KRR when $n/ B(\cA_d,\ell) \to \psi$ is given by 
%\[
% \E_{f_\star,\beps} \big[R_{\text{test}} ( f_\star ; \bX , \beps , \lambda)  \big] = F_\ell^2 \cB(  \psi , \xi_* ) + ( F_{>\ell}^2 + \sigma_\eps^2 ) \cV (  \psi , \xi_* )  + F_{>\ell}^2 + o_{d,\P}(1)\, .
%\]
%\notate{Do also train error and RKHS norm.}
\end{theorem}

\begin{table}[t]
\centering
\renewcommand{\arraystretch}{1.5}
\begin{tabular}{ |c|c|c|  }
\hline
Subspace & Bias & Variance \\
\hline
$\cV_{\leq \ell - 1}$ & $o_{d,\P}(1)$ & $o_{d,\P}(1)$ \\
$\cV_{\ell }$ & $F_\ell^2 \cdot \cB ( \psi , \zeta_*) +F_{>\ell}^2 \cdot \cV(\psi, \zeta_*) + o_{d,\P}(1)$   & $\sigma_\eps^2 \cdot \cV(\psi, \zeta_*)+ o_{d,\P}(1)$ \\
$\cV_{> \ell }$ & $F_{>\ell}^2 + o_{d,\P}(1)$ &  $o_{d,\P}(1)$ \\
\hline

\end{tabular}
\caption{Bias-variance decomposition of KRR over each subspace. } \label{tab:bias-variance}
\end{table}

A detailed proof of Theorem \ref{thm:KRR} can be found in Appendix \ref{app:proof_main_KRR}. Let us give some intuition on the asymptotic formula for the test error \eqref{eq:asymp_test}. It will be instructive to consider the contribution to the test error of the three subspaces
\[
L^2(\cA_d) = \cV_{\leq \ell - 1} \oplus \cV_{\ell} \oplus \cV_{>\ell}\, .
\]
We compute the classical bias-variance decomposition of the test error on each of these subspaces $\cV \in \{\cV_{\leq \ell-1} , \cV_\ell , \cV_{>\ell} \}$:
\[
\E_{\beps} \big[ \| \proj_{\cV}( f_* - \hf(\cdot , \hba_\lambda )) \|_{L^2}^2 \big] =  \| \proj_{\cV}( f_* - \E_\beps \hf(\cdot , \hba_\lambda )) \|_{L^2}^2 + E_{\bx} \big[ \Var_\beps ( \proj_{\cV} \hf(\bx , \hba_\lambda ) | \bx ) \big] \,.
\]
We gathered the results in Table \ref{tab:bias-variance}:
\begin{description}
\item[Subspace $\cV_{\leq \ell -1}$:] asymptotically, KRR learns exactly $\proj_{\leq \ell - 1} f_*$. Intuitively, we can effectively replace $\bH$ by $\bH_{\leq \ell -1}$ on this subspace. Hence, KRR reduces to the easy low-dimensional problem of fitting a degree-$(\ell -1)$ polynomial ($\Theta_d (d^{\ell - 1})$ parameters) with a degree-$(\ell -1)$ polynomial kernel with $n = \Theta_d (d^\ell)$ samples.

\item[Subspace $\cV_{>\ell}$:] KRR does not learn $\proj_{>\ell} f_*$ at all. This subspace contributes to the phenomenon of `benign overfitting' \cite{bartlett2020benign,bartlett2021deep}. Indeed consider $\lambda = 0^+$, i.e., the KRR solution interpolates the training data $\hf (\bx_i ; \hba_{\lambda}) = y_i, i \leq n$. The KRR solution can be decomposed as the sum of a regular solution $\proj_{\leq \ell} \hf (\cdot ; \hba_{\lambda}) $ (degree-$\ell$ polynomial) useful for prediction and a spiky component $ \proj_{> \ell} \hf (\cdot ; \hba_{\lambda})$. This second component helps to interpolate with spikes $\proj_{> \ell} \hf (\bx_i; \hba_{\lambda}) = y_i -  \proj_{\leq \ell} \hf(\bx_i , \hba_\lambda ) $ but do not contribute overall to the test error $\| \proj_{>\ell} \hf(\cdot , \hba_\lambda ) \|_{L^2}^2 = o_{d,\P}(1)$.

\item[Subspace $\cV_\ell$:] from the decomposition \eqref{eq:Hn_decompo} and rescaling the kernel, we can effectively replace $\bH + \lambda \id_n$ by $ \bQ_\ell + [( \mu_{>\ell} +  \lambda)/\mu_\ell]\id_n$ on $\cV_\ell$. We therefore obtain a high-dimensional regression problem with $B(\cA_d , n)$ parameters, effective regularization $\zeta_* = (\mu_{>\ell} +\lambda)/\mu_\ell$ and effective noise $\proj_{>\ell} f_* (\bx_i) +\eps_i$.
\end{description}

We plotted in Figure \ref{fig:test_error_asymp} the asymptotic test error \eqref{eq:asymp_test} with $F_\ell^2 = F_{>\ell}^2 + \sigma_\eps^2 = 1$ for different value of the effective regularization parameter $\zeta_\ell$. For sufficiently small effective regularization $\zeta_\ell$ or effective signal-to-noise ratio $\text{SNR}_\ell = F_\ell^2/(F_{>\ell}^2 + \sigma_\eps^2) $, a double descent behavior can appear.

\subsection{Pointwise asymptotic test error}
\label{sec:pointwise}

The convergence in Theorem \ref{thm:KRR} holds \textit{in probability} over the class of target functions described in Assumption \ref{ass:dist_f}, while the results in \cite{ghorbani2021linearized,mei21generalization} for $n \asymp d^\kappa, \kappa \not\in \naturals$ hold \textit{pointwise}, i.e., for any deterministic sequence $\{ f_{*,d} \in L^2(\cA_d) \}_{d \geq 1}$. In this section, we briefly describe what would be needed to strengthen Theorem \ref{thm:KRR} to hold pointwise.

Recall $\bY_\ell (\bx) = (Y_{\ell s} (\bx) )_{s \in [B(\cA_d , \ell)]} \in \R^{B(\cA_d,\ell)}$ and $\bY_\ell = [\bY_\ell (\bx_1) , \ldots , \bY_\ell (\bx_n) ]^\sT \in \R^{n \times B(\cA_d,\ell)}$. For simplicity, consider a target function $f_* (\bx) = \< \bbeta_* , \bY_\ell ( \bx) \>$, no noise $\eps_i = 0$, and $\bH = \mu_\ell \bQ_\ell + \mu_{>\ell} \id_n$ (i.e., $\mu_{d,k} = 0$ for $k\leq \ell -1$). We focus on the contribution of $\cV_\ell$ to the test error (noting $B = B(\cA_d , \ell)$ for convenience)
\begin{equation}
\begin{aligned}
\| \proj_{\cV_\ell} (f_* - \hf(\cdot , \hat \ba_\lambda )) \|_{L^2}^2 = &~ \| \bbeta_* - \xi_\ell \bY_\ell^\sT (\bH + \lambda \id_n )^{-1} \bY_{\ell} \bbeta_* \|_2^2  \\
=&~ \zeta_\ell^2 \cdot \bbeta_*^\sT \big( \bY_\ell^\sT \bY_\ell / B+ \zeta_\ell \id_n  \big)^{-2} \bbeta_* \, .
\end{aligned}
\end{equation}
In order to prove the asymptotic test error formula, we would need to show that 
\[
\bbeta_*^\sT \big( \bY_\ell^\sT \bY_\ell / B + \zeta_\ell \id_n  \big)^{-2} \bbeta_* - \frac{1}{B}\Tr \big[( \bY_\ell^\sT \bY_\ell / B + \zeta_\ell \id_n  )^{-2}\big] \cdot \|  \bbeta_* \|_2^2 \stackrel{\P}{\longrightarrow} 0 \, .
\]
As shown in \cite{hastie2022surprises}, this can be reduced to showing that 
\begin{equation}\label{eq:iso_law}
\bbeta_*^\sT \big( \bY_\ell^\sT \bY_\ell / B + \zeta_\ell \id_n  \big)^{-1} \bbeta_* - \frac{1}{B}\Tr \big[( \bY_\ell^\sT \bY_\ell / B + \zeta_\ell \id_n  )^{-1}\big] \cdot \|  \bbeta_* \|_2^2 \stackrel{\P}{\longrightarrow} 0 \, ,
\end{equation}
which is known as an `isotropic local law' \cite{alex2014isotropic} (in fact we would need to show that Eq.~\eqref{eq:iso_law} is $O_{d,\P} (B^{-1/2 +\delta})$ to bound the cross terms in the test error, as in \cite{bartlett2021deep}). \cite{alex2014isotropic} shows Eq.~\eqref{eq:iso_law} for matrices $\bY_\ell$ with i.i.d.~entries, which implies the pointwise version of Theorem \ref{thm:KRR} for $\ell = 1$ (this was already proved in \cite{bartlett2021deep} for more general covariates distributions). It is a significant challenge to extend \cite{alex2014isotropic} to our setting, with non-independent entries in the rows of $\bY_\ell$, and we leave it for future work.

\section*{Acknowledgements}

This work was supported by NSF through award DMS-2031883 and the Simons Foundation
through Award 814639 for the Collaboration on the Theoretical Foundations of Deep Learning. We
also acknowledge the NSF grant CCF-2006489 and the ONR grant N00014-18-1-2729.

\addcontentsline{toc}{section}{Bibliography}
\bibliographystyle{amsalpha}
\bibliography{KRR_polyRegime.bbl}

\clearpage
\appendix

\section{Proof of Proposition \ref{prop:vanishing_quadratic}: the case of spherical harmonics}\label{app:proof_sphere}

In this section, it will be convenient to consider $\S^{d-1}:= \S^{d-1}(1)$ the unit sphere instead of $\S^{d-1} (\sqrt{d})$ the sphere of radius $\sqrt{d}$ as in the main text. In particular, $\| \cdot \|_{L^2}$ and $\< \cdot , \cdot \>_{L^2}$ will correspond here to the standard scalar product and norm associated to the $L^2 $ space with $\bx \sim \Unif ( \S^{d-1})$.

\subsection{Explicit representation of spherical harmonics}

The proof of Proposition \ref{prop:vanishing_quadratic} will rely on an explicit representation of spherical harmonics in terms of the generalized spherical coordinate system in dimension $d$. See for example \cite{avery2012hyperspherical,dai2013spherical}.

Recall the definition of the spherical coordinate system: for $\bx \in \S^{d-1}$, 
\begin{equation}\label{eq:spherical_coordinates}
\begin{cases}
x_1 = \sin ( \theta_{d-1} ) \cdots \sin (\theta_{2}) \sin (\theta_{1})\, , \\
x_2 = \sin ( \theta_{d-1} ) \cdots \sin (\theta_{2}) \cos(\theta_{1})\, , \\
\cdots \\
x_{d-1} = \sin ( \theta_{d-1} )  \cos (\theta_{d-2})\, , \\
x_d = \cos ( \theta_{d-1} ) , \\
\end{cases}
\end{equation}
where $0 \leq \theta_{1} \leq 2\pi$ and $0 \leq \theta_i \leq \pi$ for $i = 2, \ldots, d-1$. The uniform probability measure on the unit sphere is given by
\[
\mu (\de \sigma) = \frac{\Gamma (d/2)}{2 \pi^{d/2}} \prod_{j=1}^{d-2} \sin ( \theta_{d-j})^{d-j - 1} \de \theta_{d-1} \cdots \de \theta_{2} \de \theta_{1} \, .
\]
For convenience, we introduce the normalized Gegenbauer polynomials $\tQ_{k}^{(d)} : [-1,1] \to \R$ on the sphere of radius $1$, such that for $\bx \sim \Unif ( \S^{d-1})$, 
\[
\begin{aligned}
\E_{\bx} [ \tQ_{k}^{(d)} (x_1)\tQ_{k'}^{(d)} (x_1) ] =&~ \frac{\Gamma (d/2) }{\sqrt{\pi} \Gamma ((d-1)/2)}\int_{-1}^1 \tQ_k^{(d)}(x_1)  \tQ_{k'}^{(d)}(x_1)  (1 - x_1^2)^{(d-3)/2} \de x_1 \\
=&~ \frac{\Gamma (d/2) }{\sqrt{\pi} \Gamma ((d-1)/2)} \int_{0}^\pi \tQ_k^{(d)}(\cos(\theta))  \tQ_{k'}^{(d)}(\cos(\theta) )  \sin(\theta)^{d-2} \de \theta \\
=&~ \delta_{k,k'}\, .
\end{aligned}
\]

Introduce the set of $d$ indices
\[
\cA_{d,\ell} = \big\{ \balpha = ( \alpha_1 , \ldots , \alpha_d) \in \naturals^d \big\vert \alpha_1 + \ldots + \alpha_{d-1} = \ell , \text{if $\alpha_{d-1} > 0$, then $\alpha_d \in \{0,1\}$, o.w. $\alpha_d = 0$} \big\}\, .
\]
Notice that 
\begin{equation}
| \cA_{d,\ell} | = 2 {{d+\ell - 2}\choose{d-2}} - {{d+\ell-3}\choose{d-3}} = \frac{2\ell+d-2}{d-2}  {{d+\ell-3}\choose{\ell}} = B(d,\ell)\, .
\end{equation}

\begin{proposition}\label{prop:spherical_representation}
For integers $d > 2$ and $\ell \geq 0$,  and $\balpha \in \cA_{d,\ell}$, define
\begin{equation}\label{eq:spherical_harmonics_explicit}
Y_{\balpha} (\bx) = C_{\balpha}^{1/2} g_{\balpha} (\theta_{1} ) \prod_{j = 1}^{d-2} \Big\{ \tQ^{(d_j )}_{\alpha_j} ( \cos (\theta_{d-j})) \sin (\theta_{d-j})^{|\alpha^{j+1}|} \Big\} \, ,
\end{equation}
where $| \alpha^{j+1} | = \alpha_{j+1} + \ldots + \alpha_{d-1}$, $d_j = 2 | \alpha^{j+1} | +  d-j + 1$, 
\[
g_{\balpha } (\theta_1) = \begin{cases}
\cos (\alpha_{d-1} \theta_1) & \text{if $\alpha_d = 0$ and $\alpha_{d-1} > 0$,} \\
1 / \sqrt{2} & \text{if $\alpha_d = 0$ and $\alpha_{d-1} = 0$,} \\
\sin (\alpha_{d-1} \theta_1) & \text{if $\alpha_d = 1$ and $\alpha_{d-1} > 0$,}
\end{cases}
\]
and
\[
\begin{aligned}
C_{\balpha} =&~ \frac{2 \pi^{d/2}}{\Gamma(d/2)} \cdot \frac{1}{\pi} \cdot \prod_{j = 1}^{d-2} \frac{\Gamma (d_j/2)}{\sqrt{\pi} \Gamma((d_j - 1)/2)}  \\
=&~ \frac{2 \Gamma (d_{1}/2)}{\Gamma(d/2) \Gamma((d_{d-2} - 1)/2)} \prod_{2\leq j \leq d-2: \alpha_j>0} \frac{\Gamma (d_{j}/2)}{\Gamma ( (d_{j-1} - 1)/2)}\, .
\end{aligned}
\]
Then $Y_{\balpha}$ is an homogeneous polynomial of degree $\ell$. Furthermore, $\{ Y_{\balpha} \}_{\balpha \in \cA_{d,\ell}}$ is an orthonormal basis of $V_{d,\ell}$, the space of degree $\ell$ spherical harmonics. 
\end{proposition}

A proof of this proposition can be found for example in \cite{dai2013spherical}. For completeness, we include here the proof with our notations and normalization choice.

\begin{proof}[Proof of Proposition \ref{prop:spherical_representation}]
We have for $\bx \sim \Unif (\S^{d-1})$,
\[
\begin{aligned}
&~\E_{\bx} [Y_{\balpha} (\bx)^2]\\
= &~ \frac{\Gamma (d/2)}{2 \pi^{d/2}} \int Y_{\balpha} (\bx)^2 \prod_{j=1}^{d-2} \sin ( \theta_{d-j})^{d-j - 1} \de \theta_{d-1} \cdots \de \theta_{2} \de \theta_{1} \\
=&~ \Big( \int_0^{2\pi} g_{\balpha} (\theta_1 )^2 \frac{\de \theta_1}{\pi} \Big) \cdot \prod_{j=1}^{d-2}  \Big\{ \frac{\Gamma (d_j/2)}{\sqrt{\pi} \Gamma((d_j - 1)/2)}  \int_0^\pi \tQ_{\alpha_j}^{(d_j)} (\cos (\theta_{d-j}))^2 \sin (\theta_{d-j} )^{d_j - 2} \de \theta_{d-j} \Big\} \\
=&~ 1 \, .
\end{aligned} 
\]
Furthermore, for $\bbeta \neq \balpha$, we have $\< Y_{\balpha} , Y_{\bbeta} \>_{L^2} = \E_{\bx} [Y_{\balpha} (\bx) Y_{\bbeta} (\bx)] = 0$. Indeed, take the largest $j \in [d]$ such that $\beta_j \neq \alpha_j$, then either $j = d, (d-1)$ and in that case $\< Y_{\balpha} , Y_{\bbeta } \>_{L^2} = 0$ because of the orthogonality of $g_{\balpha}$ and $g_{\bbeta}$, or $j < d-1$, and $| \alpha^{j+1}| = | \beta^{j+1}|$ and we have $\< Y_{\balpha} , Y_{\bbeta } \>_{L^2} = 0$ by orthogonality of the Gegenbauer polynomials $\tQ^{(d_j)}_{\alpha_j}$ and $\tQ^{(d_j)}_{\beta_j}$. 

In order to check that Eq.~\eqref{eq:spherical_harmonics_explicit} is a homogeneous polynomial, recall that in the spherical coordinates \eqref{eq:spherical_coordinates}, we have 
\[
\cos \theta_k = \frac{x_{k+1} }{ \sqrt{x_1^2 + \ldots + x_{k+1}^2}}\, \qquad \sin \theta_k = \sqrt{\frac{x_1^2 + \ldots + x_{k}^2}{ x_1^2 + \ldots + x_{k+1}^2}}\, ,
\]
and therefore
\[
\prod_{j = 1}^{d-2} \sin (\theta_{d-j} )^{|\alpha^{j+1}|} = (x_1^2 + x_2^2)^{\alpha_{d-1}/2} \prod_{j = 1}^{d-2} \big( x_1^2 + \ldots + x_{d-j+1}^2 \big)^{\alpha_j / 2} \, .
\]
Furthermore, notice that the degree-$\alpha_j$ polynomial $\tQ^{(d_j)}_{\alpha_j}$ is even when $\alpha_j$ is even and odd when $\alpha_j$ is odd. Therefore 
\[
\big( x_1^2 + \ldots + x_{d-j+1}^2 \big)^{\alpha_j / 2} \tQ^{(d_j)}_{\alpha_j} \left( \frac{x_{d-j+1} }{ \sqrt{x_1^2 + \ldots + x_{d-j+1}^2}} \right)\, ,
\]
is a polynomial of degree $\alpha_j$. We can further write $h_{\balpha} (x_1,x_2) := (x_1^2 + x_2^2)^{\alpha_{d-1}/2} g_{\balpha} (\theta_1)$ as the real part or the imaginary part of the polynomial $(x_2 + i x_1 )^{\alpha_{d-1}}$ (or a constant if $\alpha_{d-1} = \alpha_d =0$), depending on $\alpha_{d-1}, \alpha_d$. We deduce that
\begin{equation}\label{eq:expression_spherical_x}
Y_{\balpha} (\bx) = C_{\balpha}^{1/2} h_{\balpha} (x_1,x_2) \prod_{j = 1}^{d-2} \left\{  \big( x_1^2 + \ldots + x_{d-j+1}^2 \big)^{\alpha_j / 2} \tQ^{(d_j)}_{\alpha_j} \left( \frac{x_{d-j+1} }{ \sqrt{x_1^2 + \ldots + x_{d-j+1}^2}} \right) \right\} \, ,
\end{equation}
is a polynomial of degree $\alpha_1 +  \ldots + \alpha_{d-1} = \ell $. Furthermore, from the expression \eqref{eq:expression_spherical_x}, we can directly check that $Y_{\balpha} (\bx)$ is a homogeneous polynomial.
\end{proof}

\subsection{Proof of Proposition \ref{prop:vanishing_quadratic}}

Similarly to the case on the hypercube, we see from the expression \eqref{eq:spherical_harmonics_explicit} that the spherical harmonics $Y_{\balpha} ( \bx) $ are approximately represented as a product of at most $\ell$ independent zero-mean variables $Q^{(d_j)}_{\alpha_j} (\cos ( \theta_{d-j}))$, corresponding to the $\alpha_j >0$ (at most $\ell$ of them). We expect product of spherical harmonics that do not have overlap of their support $\{ j : \alpha_j >0 \}$ to be approximately uncorrelated, and a proof similar to the hypercube case outlined in the main text to extend to the sphere. The main difficulty comes from the fact that $Y_{\balpha} (\bx)$ depends on every spherical coordinates through $\sin ( \theta_{d-j} )^{|\alpha^{j+1}|}$. However carefully bounding the contribution of each coordinates allows to show that the variance still vanishes (see Section \ref{sec:technical_lemmas} for technical bounds).

Consider $\bx \sim \Unif ( \S^{d-1} )$ and let $\bY_\ell ( \bx )  = ( Y_{\balpha} (\bx) )_{\balpha \in \cA_{d,\ell}}$ be the spherical harmonics basis given in Proposition \ref{prop:spherical_representation}. Denote $B = B(\S^{d-1},\ell)$ and $\bA  = \bA_d$ for convenience. Again we consider bounding the variance of the quadratic form, which we decompose in two terms
\[
\begin{aligned}
 \Var_\bx ( \cQ (\bx) ) =  \E_{\bx}  \left[ \left( \frac{1}{B} \bY_{\ell} (\bx)^\sT\bA \bY_{\ell} (\bx)    - \frac{1}{B }  \tr (\bA) \right)^2 \right]  \leq ({\rm I}) + ({\rm II})\, ,
\end{aligned}
\]
where 
\[
\begin{aligned}
({\rm I}) = &~ \frac{1}{B^2} \Big\vert  \sum_{\balpha \neq \bbeta , \bgamma \neq \bdelta } \E_{\bx} \big[ Y_{\balpha} ( \bx) Y_{\bbeta} ( \bx) Y_{\bgamma} ( \bx) Y_{\bdelta} ( \bx) \big] A_{\balpha, \bbeta} A_{\bgamma, \bdelta }\Big\vert \, , \\
({\rm II}) = &~ \frac{1}{B^2}\Big\vert  \sum_{\balpha , \bbeta  } \left( \E_{\bx} \big[ Y_{\balpha} ( \bx)^2 Y_{\bbeta} ( \bx)^2  \big] - 1 \right) A_{\balpha, \balpha} A_{\bbeta, \bbeta } \Big\vert \, .
\end{aligned}
\]
We will show that both these terms are $o_d (1)$, which implies the concentration in probability of the quadratic form.

\noindent
{\bf Step 1. Bounding term $({\rm I})$.}

We proceed similarly than in the main text. Consider $\bC$ the square matrix of size $B(B - 1)$ such that for any $\balpha \neq \bbeta$ and $\bgamma \neq \bdelta$,
\[
\bC_{(\balpha, \bbeta) , ( \bgamma , \bdelta)} =  \E_{\bx} \big[ Y_{\balpha} ( \bx) Y_{\bbeta} ( \bx) Y_{\bgamma} ( \bx) Y_{\bdelta} ( \bx) \big] \, .
\]
Denote $\ba \in \R^{B (B - 1)}$ the vector that contains all off-diagonal entries of $\bA$. Then the first term can be bounded by
\[
({\rm I}) = \frac{1}{B^2} \Big\vert  \ba^\sT \bC \ba \Big\vert  \leq \frac{1}{B^2}  \| \bC \|_{\op} \| \ba \|_2^2 \leq \frac{\| \bC \|_{\op}}{B}   \| \bA \|_{\op}^2 \, .
\]
By assumption $\| \bA \|_{\op} \leq 1$. Hence, it is sufficient to show
\begin{equation}\label{eq:bound_op_C_I}
\| \bC \|_{\op} \leq \| \bC \|_{1, \infty} =  \max_{\balpha \neq \bbeta } \sum_{\bgamma \neq \bdelta} \Big\vert \E_{\bx} \big[ Y_{\balpha} ( \bx) Y_{\bbeta} ( \bx) Y_{\bgamma} ( \bx) Y_{\bdelta} ( \bx) \big]  \Big\vert = o_d ( B) \, .
\end{equation}

Denote $S_{\balpha} = \{ j \in [d] : \alpha_j > 0 \}$ (similarly $S_{\bbeta}, S_{\bgamma} , S_{\bdelta}$) and $r ( \balpha, \bbeta , \bgamma , \bdelta ) \subseteq [d-1]$ the subset of indices $j \in [d-1]$ such that $j$ belongs to exactly one of the sets $S_{\balpha}, S_{\bbeta}, S_{\bgamma} , S_{\bdelta}$ (e.g., $\alpha_j >0$ and $\beta_j = \gamma_j = \delta_j =0$). In the rest of this step, we fix $\balpha \neq \bbeta$ arbitrary and consider $v ( \bgamma , \bdelta) = r ( \balpha, \bbeta , \bgamma , \bdelta ) \cap (S_{\bgamma} \cup S_{\bdelta})$, i.e., the subset of indices where only $\gamma_j$ or $\delta_j$ are non-zero. For convenience denote $T(\bgamma, \bdelta) := \Big\vert \E_{\bx} \big[ Y_{\balpha} ( \bx) Y_{\bbeta} ( \bx) Y_{\bgamma} ( \bx) Y_{\bdelta} ( \bx) \big]  \Big\vert$. By Lemma \ref{lem:technical_termI_bound},
\begin{equation*}
    \begin{aligned}
     \sum_{\bgamma \neq \bdelta}  T(\bgamma, \bdelta) =&~ \sum_{ u = 0}^{2\ell}  \sum_{S \subseteq [d-1], |S| =u } \sum_{\bgamma \neq \bdelta, v ( \bgamma , \bdelta) = S} T(\bgamma, \bdelta)  \\
     \leq &~ C_\ell \sum_{ u = 0}^{2\ell} \sum_{S \subseteq [d-1], |S| =u } \sum_{\bgamma \neq \bdelta, v ( \bgamma , \bdelta) = S} \prod_{j \in S} \frac{1}{d-j} \, .
    \end{aligned}
\end{equation*}
For fixed $S$, let us bound the number of $\bgamma \neq \bdelta$ such that $v(\bgamma, \bdelta) = S$. If $u >0$, it means that there are at most $\ell -u$ other coordinates $j \in [d-1]$ where $\gamma_j >0$, and $\ell-u$ coordinates where $\delta_j >0$, and either both $\delta_j, \gamma_j >0$ or $j \in S_{\balpha} \cup S_{\bbeta}$: we deduce that there is at most $d^{\ell - u}$ ways of choosing coordinates for $S_{\bgamma} \Delta S_{\bdelta} \setminus (S_{\balpha} \cup S_{\bbeta})$, and then at most $(4\ell)^{2\ell}$ ways of choosing $\bgamma,\bdelta$ compatible with this support. For $u = 0$, because $\bgamma \neq \bdelta$, we can't have $S_{\bgamma} \cup S_{\bdelta} \setminus (S_{\balpha} \cup S_{\bbeta}) = \ell$, hence again there is at most $O(d^{\ell-1})$ such $\bgamma, \bdelta$. We deduce that there exists a constant $C_\ell '$ such that 
\begin{equation}\label{eq:termI_bbd}
    \begin{aligned}
     \sum_{\bgamma \neq \bdelta}  T(\bgamma, \bdelta) \leq &~ C_\ell' d^{\ell-1} \sum_{ u = 0}^{2\ell} \sum_{S \subseteq [d-1], |S| =u }  \prod_{j \in S} \frac{1}{d-j} \\ 
     \leq &~ C_\ell' d^{\ell-1} \sum_{ u = 0}^{2\ell} \left( \sum_{j \in [d-1]} \frac{1}{d-j} \right)^u \\
     \leq &~ \tC_{\ell} d^{\ell-1} \log (d)^{2 \ell} \, .
    \end{aligned}
\end{equation}
Using that $B = \Theta (d^{\ell})$ and that the inequality \eqref{eq:termI_bbd} is uniform over $\balpha,\bbeta$, the bound $\| \bC \|_{\op} = o_d (B)$ result from Eq.~\eqref{eq:bound_op_C_I}.

\noindent
{\bf Step 2. Bounding term $({\rm II})$.}

We introduce a new symmetric matrix $\tbC \in \R^{B \times B}$ such that for any $\balpha, \bgamma \in \cA_{d,\ell}$,
\[
\tbC_{\balpha, \bbeta} = \E_{\bx} \big[ Y_{\balpha} (\bx)^2 Y_{\bbeta} (\bx)^2 \big] - 1 \, .
\]
Denote $\tba \in \R^B$ the vector that contains all the diagonal elements of $\bA$, i.e., $\tba = ( \bA_{\balpha,\balpha} )_{\balpha \in \cA_{d,\ell}}$. Similarly to the previous step, we bound the second term by
\[
({\rm II}) = \frac{1}{B^2} \big\vert \tba^\sT \tbC \tba \big\vert \leq \frac{1}{B^2}  \| \tbC \|_{\op} \| \tba \|_2^2 \leq \frac{ \| \tbC \|_{\op}}{B} \| \bA \|_{\op}^2 \, ,
\]
where we used that $\| \tba \|_2^2  \leq \tr ( \bA^\sT \bA ) \leq B \| \bA \|_{\op}^2$. Again, it is sufficient to show
\begin{equation}\label{eq:bound_op_C_II}
\| \tbC \|_{\op} \leq \| \tbC \|_{1, \infty} =  \max_{\balpha } \sum_{\bbeta } \Big\vert \E_{\bx} \big[ Y_{\balpha} ( \bx)^2 Y_{\bbeta} ( \bx)^2 \big]  \Big\vert = o_d ( B) \, .
\end{equation}
Let us fix $\balpha$ arbitrary, and denote $\cU_{\balpha} \subseteq \cA_{d,\ell}$ the set
\[
\cU_{\balpha} = \Big\{ \bbeta \in \cA_{d,\ell}: \text{ either }S_{\bbeta}\cup S_{\balpha} \neq \emptyset \text{ or } \max\{ j : j \in S_{\bbeta} \} \geq d - \sqrt{d}  \Big\} \, .
\]
One can easily check that $| \cU_{\balpha} | = O_d (d^{\ell - 1/2} )$. Using Lemma \ref{lem:technical_termI_bound} (which shows that terms in $\cU_{\balpha}$ are bounded by a constant), and Lemma \ref{lem:technical_termI_bound} (to bound the terms $\bbeta \not\in \cU_{\balpha}$), we get
\begin{equation}
\begin{aligned}
    \sum_{\bbeta } \Big\vert \E_{\bx} \big[ Y_{\balpha} ( \bx)^2 Y_{\bbeta} ( \bx)^2 \big]  = &~ \sum_{\bbeta \in \cU_{\balpha} } \Big\vert \E_{\bx} \big[ Y_{\balpha} ( \bx)^2 Y_{\bbeta} ( \bx)^2 \big] + \sum_{\bbeta \not\in \cU_{\balpha} } \Big\vert \E_{\bx} \big[ Y_{\balpha} ( \bx)^2 Y_{\bbeta} ( \bx)^2 \big] \\
    =&~ O_d ( d^{\ell - 1/2}) \, .
    \end{aligned}
\end{equation}
Noting that this bound is uniform over $\balpha$ and that $B = \Theta (d^{\ell})$ concludes the proof.

\subsection{Technical lemmas}
\label{sec:technical_lemmas}

We first prove the following useful bound on the expectation of Gegenbauer polynomials $Q_{\alpha}^{(k)}$ over input $\bx \sim \Unif (\S^{d-1} )$ with mismatched dimension $k \neq d$.

\begin{lemma}\label{lem:mismatched_dimension}
Fix $\ell \in \naturals$. Consider integers $d,k_1,k_2, \alpha$ such that $k_1 , k_2 \leq 4 \ell$ and $\alpha>0$, and $\bx \sim \Unif (\S^{d-1} )$. There exists a universal constant $C_{\ell}>0$ such that 
\begin{align}
    \Big\vert \E_{\bx} \Big[ \tQ^{(d+k_1)}_{\alpha} (x_1 ) \big\{ 1 - x_1^2 \big\}^{k_2/2} \Big] \Big\vert \leq&~ \frac{C_{\ell}}{d}\, , \\
    \Big\vert \E_{\bx} \Big[ \tQ^{(d)}_{\alpha} (x_1 )^2 \big\{ 1 - x_1^2 \big\}^{k_2/2} \Big]  - 1\Big\vert \leq&~ C_\ell \frac{k_2}{d}\, .
\end{align}
\end{lemma}

\begin{proof}[Proof of Lemma \ref{lem:mismatched_dimension}]
Let us write explicitly the expectation:
\[
\begin{aligned}
&~ \E_{\bx} \Big[ \tQ^{(d+k_1)}_{\alpha} (x_1 ) \big\{ 1 - x_1^2 \big\}^{k_2/2} \Big]\\
=&~ \frac{\Gamma(d/2)}{\sqrt{\pi} \Gamma ((d-1)/2)} \int_{-1}^{+1} \tQ_\alpha^{(d+k_1)} (x_1) \big\{ 1 - x_1^2 \big\}^{(k_2 + d - 3)/2} \de x_1 \\
= &~\frac{\Gamma(d/2) \Gamma((k_1 + d - 1)/2)}{ \Gamma((k_1 + d)/2) \Gamma ((d-1)/2)} \E_{\bx \sim \Unif ( \S^{k_1 + d -1} )} \Big[ \tQ^{(d+k_1)}_{\alpha} (x_1 )  \Big\{ (1 - x_1^2 )^{(k_2 - k_1)/2} - 1 \Big\}  \Big] \, ,
\end{aligned}
\]
where we used in the last line that $Q_{\alpha}^{(d+k_1)}$ is orthogonal to the constant polynomial. By assumption, $k_1 \leq 4\ell$ and the multiplicative factor can be bounded by a constant independent of $d$. 

Let us bound the expectation. First, note that we can assume $d $ sufficiently large (at the price of taking a larger but still constant $C_{\ell}$). By Cauchy-Schwarz inequality, (here $\bx \sim \Unif ( \S^{k_1 + d -1} )$)
\[
\begin{aligned}
&~ \Big\vert \E_{\bx} \Big[ \tQ^{(d+k_1)}_{\alpha} (x_1 )  \Big\{ (1 - x_1^2 )^{(k_2 - k_1)/2} - 1 \Big\}  \Big]^2 \Big\vert \\
\leq&~ \E_{\bx} \Big[ \tQ^{(d+k_1)}_{\alpha} (x_1 )^2    \Big] \E_{\bx} \Big[  \Big\{ (1 - x_1^2 )^{(k_2 - k_1)/2} - 1 \Big\}^2  \Big] \\
= &~ \E_{\bx} \Big[  \Big\{ (1 - x_1^2 )^{(k_2 - k_1)/2} - 1 \Big\}^2  \ind_{|x_1| \leq 1/2} \Big] + \E_{\bx} \Big[  \Big\{ (1 - x_1^2 )^{(k_2 - k_1)/2} - 1 \Big\}^2  \ind_{|x_1| >1/ 2} \Big] \\
\leq&~ C \E_{\bx} [ x_1^4] + C (3/4)^{d/2} \leq C/d^2 \, .
\end{aligned}
\]
The second inequality can be obtained following similar argument. First notice that for $k_2 = 0$, we have equality by normalization of Gegenbauer polynomials. For $k_2 >0$, we can absorb the dependency on $k_2$ of the right-hand-side by multiplying $C_\ell$ by $4\ell$. Therefore we have
\begin{equation}
\begin{aligned}
&~ \Big\vert \E_{\bx} \Big[ \tQ^{(d)}_{\alpha} (x_1 )^2 \big\{ 1 - x_1^2 \big\}^{k_2/2} \Big]  - 1\Big\vert^2 \\
=&~  \Big\vert \E_{\bx} \Big[ \tQ^{(d)}_{\alpha} (x_1 )^2 \Big\{ (1 - x_1^2 )^{k_2/2} - 1 \Big\}\Big]  \Big\vert^2 \\
 \leq&~ \|  \tQ^{(d)}_{\alpha} \|_{L^4}^2  \E_{\bx} \Big[  \Big\{ (1 - x_1^2 )^{k_2/2} - 1 \Big\}^2\Big] \leq C \E_{\bx} [x_1^4] \leq \frac{C}{d^2}\, ,
\end{aligned}
\end{equation}
where we used hypercontractivity on the sphere to bound $\|  \tQ^{(d)}_{\alpha} \|_{L^4}^2$ and that $\big\{ (1 - x_1^2 )^{k_2/2} - 1 \big\}^2  \leq C x_1^4$ for some constant $C$.
\end{proof}

\begin{lemma}\label{lem:technical_termI_bound}
For $ \balpha, \bbeta , \bgamma , \bdelta \in \cA_{d,\ell}$, denote $r (\balpha, \bbeta , \bgamma , \bdelta ) \subseteq [d-1]$ the subset of indices $j \in [d-1]$ such that only one of the $\alpha_j, \beta_j, \gamma_j , \delta_j$ is non-zero. There exists a universal constant $C_{\ell} >0$ such that for any $ \balpha, \bbeta , \bgamma , \bdelta \in \cA_{d,\ell}$, 
\begin{equation}\label{eq:bound_prod_Y}
    \Big\vert \E_{\bx} \big[ Y_{\balpha} ( \bx) Y_{\bbeta} ( \bx) Y_{\bgamma} ( \bx) Y_{\bdelta} ( \bx) \big]  \Big\vert \leq C_{\ell} \prod_{j \in  r (\balpha, \bbeta , \bgamma , \bdelta )} \frac{1}{d-j }\, .
\end{equation}
\end{lemma}

\begin{proof}[Proof of Lemma \ref{lem:technical_termI_bound}]
First note that by H\"older inequality followed by hypercontractivity on the sphere (Lemma \ref{lem:hypercontractivity_sphere}),
\begin{equation}\label{eq:bound_4Y_hyper}
\begin{aligned}
\Big\vert \E_{\bx} \big[ Y_{\balpha} ( \bx) Y_{\bbeta} ( \bx) Y_{\bgamma} ( \bx) Y_{\bdelta} ( \bx) \big]  \Big\vert \leq&~ \| Y_{\balpha} \|_{L^4} \| Y_{\bbeta} \|_{L^4}  \| Y_{\bgamma} \|_{L^4}  \| Y_{\bdelta} \|_{L^4}  \\
\leq&~ 3^{2\ell} \| Y_{\balpha} \|_{L^2} \| Y_{\bbeta} \|_{L^2}  \| Y_{\bgamma} \|_{L^2}  \| Y_{\bdelta} \|_{L^2} = 3^{2\ell}  =: c_{\ell} \,.
\end{aligned}
\end{equation}
Consider the representation \eqref{eq:spherical_harmonics_explicit}. If $d-1 \in r (\balpha, \bbeta , \bgamma , \bdelta )$, then the expectation \eqref{eq:bound_prod_Y} is simply $0$. Assume that $d-1 \not\in r (\balpha, \bbeta , \bgamma , \bdelta )$, then from the bound \eqref{eq:bound_4Y_hyper}, we can decompose 
\begin{equation}\label{eq:prod4Y_Rj}
\Big\vert \E_{\bx} \big[ Y_{\balpha} ( \bx) Y_{\bbeta} ( \bx) Y_{\bgamma} ( \bx) Y_{\bdelta} ( \bx) \big]  \Big\vert  \leq c_{\ell} \cdot \prod_{ j \in r (\balpha, \bbeta , \bgamma , \bdelta )} R_j  (\balpha, \bbeta , \bgamma , \bdelta )\, ,
\end{equation}
where 
\begin{equation}\label{eq:Rj_def}
\begin{aligned}
R_j  (\balpha, \bbeta , \bgamma , \bdelta ) = \frac{\Big\vert \E_{\bz^{(j)}} \Big[ \prod_{\btheta \in \{ \balpha, \bbeta , \bgamma , \bdelta \} }  \tQ_{\theta_j}^{(d^\theta_j)} (z^{(j)}_1)  \big\{ 1 - (z^{(j)}_1 )^2\big\}^{|\theta^{j+1}|/2 } \Big] \Big\vert }{ \prod_{\btheta \in \{ \balpha, \bbeta , \bgamma , \bdelta \} } \E_{\bz^{(j)}} \Big[ \tQ_{\theta_j}^{(d^\theta_j)} (z^{(j)}_1)^4  \big\{ 1 - (z^{(j)}_1 )^2\big\}^{2|\theta^{j+1}| } \Big]^{1/4} } \, ,
\end{aligned}
\end{equation}
with $\bz^{(j)} \sim \Unif ( \S^{d-j} )$. Without loss of generality, assume that $\alpha_j > 0 $ (hence, $\beta_j = \gamma_j = \delta_j = 0$). For $\btheta \neq \balpha$, we have
\begin{equation}\label{eq:Rj_bound_1}
\begin{aligned}
\E_{\bz^{(j)}} \Big[  \big\{ 1 - (z^{(j)}_1 )^2\big\}^{2|\theta^{j+1}| } \Big] =&~ \frac{\Gamma ((d-j+1)/2) }{\sqrt{\pi}\Gamma ((d-j)/2)} \int_{-1}^1 (1 - x_1^2 )^{(4 | \theta^{j+1} | + d - j -2 )/2 } \de x_1 \\
=&~ \frac{\Gamma ((d-j+1)/2) }{\Gamma ((d-j)/2)} \cdot \frac{\Gamma ((4 | \theta^{j+1} | + d-j)/2) }{\Gamma ((4 | \theta^{j+1} | + d-j+1)/2)}\\
=&~ \prod_{0 \leq k \leq 2 | \btheta^{j+1} | -1} \frac{d- j + k}{d- j +k +1/2}  \geq (2/3)^{2\ell}\, ,
\end{aligned}
\end{equation}
where we used that $ | \btheta^{j+1} | \leq \ell$.
Consider now when $\btheta = \balpha$. First, the denominator is lower bounded by
\begin{equation}\label{eq:Rj_bound_2}
\begin{aligned}
    &~ \E_{\bz^{(j)}} \Big[ \tQ_{\alpha_j}^{(d^\alpha_j)} (z^{(j)}_1)^4  \big\{ 1 - (z^{(j)}_1 )^2\big\}^{2|\alpha^{j+1}| } \Big]^{1/2} \\
    \geq&~ \E_{\bz^{(j)}} \Big[ \tQ_{\alpha_j}^{(d^\alpha_j)} (z^{(j)}_1)^2  \big\{ 1 - (z^{(j)}_1 )^2\big\}^{|\alpha^{j+1}| } \Big] \\
     =&~ \frac{\Gamma ((d-j+1)/2) }{\sqrt{\pi}\Gamma ((d-j)/2)} \int_{-1}^1 \tQ_{\alpha_j}^{(d^\alpha_j)} (x_1)^2 (1 - x_1^2 )^{(d_j^\alpha - 3 )/2 } \de x_1 \\
     =&~  \frac{\Gamma ((d-j+1)/2) }{\Gamma ((d-j)/2)} \cdot \frac{\Gamma ((d_j^\alpha - 1)/2) }{\Gamma (d_j^\alpha/2)} \cdot \E_{\bx \sim \Unif (\S^{d_j^\alpha - 1})} \big[ \tQ_{\alpha_j}^{(d^\alpha_j)} (x_1)^2 \big] \\
     =&~ \prod_{0 \leq k \leq 2 | \alpha^{j+1} | -1} \frac{d- j + k}{d- j +k +1/2}  \geq (2/3)^{2\ell} \, ,
     \end{aligned}
\end{equation}
while the numerator is upper bounded using Lemma \ref{lem:mismatched_dimension} and that $| \alpha^{j+1} |, \ldots , | \delta^{j+1} | \leq \ell$,
\begin{equation}\label{eq:Rj_bound_3}
\begin{aligned}
\Big\vert \E_{\bz^{(j)}} \Big[ \tQ_{\theta_j}^{(d^\alpha_j)} (z^{(j)}_1)  \big\{ 1 - (z^{(j)}_1 )^2\big\}^{(|\alpha^{j+1}|+m)/2 } \Big] \Big\vert  \leq \frac{\tC_{\ell}}{d-j+1}\, .
\end{aligned}
\end{equation}
Combining bounds \eqref{eq:Rj_bound_1}, \eqref{eq:Rj_bound_2} and \eqref{eq:Rj_bound_3} in the definition \eqref{eq:Rj_def} of $R_j$ yields $R_j  (\balpha, \bbeta , \bgamma , \bdelta )  \leq C_{\ell}' / (d-j+1)$. Noting that $| r (\balpha, \bbeta , \bgamma , \bdelta ) | \leq 4\ell $, we conclude by using this inequality in \eqref{eq:prod4Y_Rj}.
\end{proof}

\begin{lemma}\label{lem:technical_termII_bound}
Recall that for $ \balpha\in \cA_{d,\ell}$, we denote $S_{\balpha} = \{ j \in [d-1]: \alpha_j >0\}$. There exists a universal constant $C_{\ell} >0$ such that for any $ \balpha, \bbeta  \in \cA_{d,\ell}$ with $S_{\balpha} \cap S_{\bbeta} = \emptyset$ and $\max \{ j : j \in S_{\bbeta} \} \leq d - \sqrt{d}$, 
\begin{equation}\label{eq:bound_prod_YII}
    \Big\vert \E_{\bx} \big[ Y_{\balpha} ( \bx)^2 Y_{\bbeta} ( \bx)^2 \big]  - 1 \Big\vert \leq \frac{C_{\ell}}{\sqrt{d}} \, . 
\end{equation}
\end{lemma}

\begin{proof}[Proof of Lemma \ref{lem:technical_termII_bound}]
Let us decompose the expectation using the representation \eqref{eq:spherical_harmonics_explicit}:
\begin{equation}\label{eq:Y2Y2}
\begin{aligned}
&~ \E_{\bx} \big[ Y_{\balpha} ( \bx)^2 Y_{\bbeta} ( \bx)^2 \big] \\
= &~ \frac{\Gamma (d/2)}{2 \pi^{d/2}} C_{\balpha} C_{\bbeta} \int_0^{2\pi} g_{\balpha} (\theta_1)^2 g_{\bbeta} (\theta_1)^2 \de \theta_1 \\
&~ \times  \prod_{j = 1}^{d-2} \left\{ \int_0^\pi \tQ^{(d^\alpha_j)}_{\alpha_j} (\cos ( \theta_{d-j}))^2  \tQ^{(d^\beta_j)}_{\beta_j} (\cos ( \theta_{d-j}))^2 \sin (\theta_{d-j})^{2| \alpha^{j+1} | + 2 | \beta^{j+1} | + d - j - 1 } \de \theta_{d-j} \right\} \, .
\end{aligned}
\end{equation}
The different terms contribute as follows in the above product. First, we can't have $\alpha_{d-1} >0$ and $\beta_{d-1} >0$ at the same time, hence
\[
\int_0^{2\pi} g_{\balpha} (\theta_1)^2 g_{\bbeta} (\theta_1)^2 \de \theta_1  = \frac{\pi}{2}\, .
\]
For $j \in [d-2]$, we have
\begin{itemize}
    \item If $j \not\in S_{\balpha} \cup S_{\bbeta}$,
\[
\int_0^{\pi}\sin (\theta_{d-j})^{2| \alpha^{j+1} | + 2 | \beta^{j+1} | + d - j - 1 } \de \theta_{d-j} = \sqrt{\pi} \frac{\Gamma( | \alpha^{j+1} | + | \beta^{j+1} | + (d-j)/2 )}{\Gamma( | \alpha^{j+1} | + | \beta^{j+1} | + (d-j+1)/2 )} \, .
\]
\item If $j \in S_{\balpha}$, from Lemma \ref{lem:mismatched_dimension}, there exists a universal constant $C_{\ell}$ such that 
\[
\begin{aligned}
&~ \int_0^{\pi} \tQ^{(d^\alpha_j)}_{\alpha_j}  (\cos ( \theta_{d-j}))^2 \sin (\theta_{d-j})^{2| \alpha^{j+1} | + 2 | \beta^{j+1} | + d - j - 1 } \de \theta_{d-j}\\
= &~ \sqrt{\pi} \frac{\Gamma( | \alpha^{j+1} | + (d-j)/2 )}{\Gamma( | \alpha^{j+1} | + (d-j+1)/2 )} \left( 1 + \frac{| \beta^{j+1}|}{d-j+1} \cdot K_j \right) \, ,
\end{aligned}
\]
with $|K_j| \leq C_{\ell}$. 

\item If $j \in S_{\bbeta}$, similarly
\[
\begin{aligned}
&~\int_0^{\pi} \tQ^{(d^\beta_j)}_{\beta_j} (\cos ( \theta_{d-j}))^2  \sin (\theta_{d-j})^{2| \alpha^{j+1} | + 2 | \beta^{j+1} | + d - j - 1 } \de \theta_{d-j} \\
=&~ \sqrt{\pi} \frac{\Gamma( | \beta^{j+1} | + (d-j)/2 )}{\Gamma( | \beta^{j+1} | + (d-j+1)/2 )} \left( 1 + \frac{| \alpha^{j+1}|}{d-j+1} \cdot K_j\right) \, ,
\end{aligned}
\]
with $|K_j| \leq C_{\ell}$. 
\end{itemize}

Combining these contributions in Eq.~\eqref{eq:Y2Y2} yields
\begin{equation*}
    \E_{\bx} \big[ Y_{\balpha} ( \bx)^2 Y_{\bbeta} ( \bx)^2 \big] = M_{\balpha, \bbeta} \cdot  \prod_{j \in S_{\balpha}} \left( 1 + \frac{| \beta^{j+1}|}{d-j+1} \cdot K_j \right)  \cdot   \prod_{j \in S_{\bbeta}} \left( 1 + \frac{| \alpha^{j+1}|}{d-j+1} \cdot K_j \right)  \, ,
\end{equation*}
where $M_{\balpha,\bbeta}$ is an explicit multiplicative factor. By assumption, $d-j \geq \sqrt{d}$ for any $j \in S_{\beta}$ and $| \beta^{j+1} | =0$ for $j \geq d - \sqrt{d}$. We deduce that there exists a constant $C'_{\ell}$ such that
\begin{equation*}
     \Big\vert \E_{\bx} \big[ Y_{\balpha} ( \bx)^2 Y_{\bbeta} ( \bx)^2 \big] - M_{\balpha, \bbeta} \Big\vert \leq C'_{\ell} M_{\balpha, \bbeta} \frac{1}{\sqrt{d}}\, .
\end{equation*}
Hence to prove the lemma, it is sufficient to show that $|M_{\balpha, \bbeta} - 1 | = O(d^{-1/2})$.

Expanding $M_{\balpha, \bbeta}$ yields
\begin{equation}\label{eq:decompo_Malphabeta}
    \begin{aligned}
    &~ M_{\balpha, \bbeta} \\
    =&~ \frac{1}{\Gamma (d/2)} \cdot \prod_{j \in [d-2]} \frac{\Gamma(d^{\alpha}_j/2) \Gamma( d^{\beta}_j /2 ) \Gamma( | \alpha^{j+1} | + | \beta^{j+1} | + (d-j)/2 ) }{\Gamma((d^{\alpha}_j -1 )/2) \Gamma( (d^{\beta}_j-1) /2) \Gamma( | \alpha^{j+1} | + | \beta^{j+1} | + (d-j+1)/2 )}  \\
    &~ \times \prod_{ j \in S_{\balpha} \cup S_{\beta} }\left( \frac{\Gamma( d_j^{\alpha} /2 )}{\Gamma( (d_j^\alpha - 1)/2 )} \right)^{c_j^{\balpha}} \left( \frac{\Gamma( d_j^{\beta} /2 )}{\Gamma( (d_j^\beta - 1)/2 )} \right)^{c_j^{\bbeta}} \frac{\Gamma( | \alpha^{j+1} | + | \beta^{j+1} | + (d-j+1)/2 )}{\Gamma( | \alpha^{j+1} | + | \beta^{j+1} | + (d-j)/2 )} \, ,
    \end{aligned}
\end{equation}
where $c_j^{\balpha} = 1 $ if $\alpha_j > 0 $ and $=0$ is $\alpha_j = 0$ (similarly for $c_j^{\bbeta}$). Note that on the first line, the product can be simplified by telescoping the terms and we obtain
\begin{equation*}
\begin{aligned}
&~ \frac{\Gamma(d^{\alpha}_1/2) \Gamma(d^{\beta}_1/2) \Gamma (|\alpha^{d-1}| + |\beta^{d-1}| + 1) }{\Gamma (d/2) \Gamma ((d^\alpha_{d-2} -1 )/2) \Gamma ((d^\beta_{d-2} -1 )/2) \Gamma ( | \alpha^{2}| + | \beta^{2}| + d/2) } \\
&~\times  \prod_{j \in [d-2]: \alpha_j>0 \text{ or } \beta_j >0}  \frac{\Gamma(d^{\alpha}_j/2) \Gamma( d^{\beta}_j /2 ) \Gamma( | \alpha^{j} | + | \beta^{j} | + (d-j -1)/2 ) }{\Gamma((d^{\alpha}_{j-1} -1 )/2) \Gamma( (d^{\beta}_{j-1}-1) /2) \Gamma( | \alpha^{j+1} | + | \beta^{j+1} | + (d-j+1)/2 )}
\\
=&~
    \frac{\Gamma(| \alpha^2 | + d/2) \Gamma(| \beta^2 | + d/2) \Gamma (\alpha_{d-1} + \beta_{d-1} +1)}{\Gamma (d/2) \Gamma (\alpha_{d-1} +1 ) \Gamma (\beta_{d-1}+1) \Gamma ( | \alpha^{2}| + | \beta^{2}| + d/2)} \\
    &~\times  \prod_{j \in [d-2]: \alpha_j>0 \text{ or } \beta_j >0}  \frac{\Gamma(|\alpha^{j+1} | + (d-j+1)/2) \Gamma( |\beta^{j+1} | + (d-j+1) /2 )  }{\Gamma(| \alpha^j| + (d-j +1  )/2) \Gamma( | \beta^j| + (d-j+1) /2) }\\
    &~ \phantom{AAAAAAAAAAAAAAAAAAAAAAAAA} \cdot \frac{\Gamma( | \alpha^{j} | + | \beta^{j} | + (d-j -1)/2 )}{\Gamma( | \alpha^{j+1} | + | \beta^{j+1} | + (d-j+1)/2 )}\, .
    \end{aligned}
\end{equation*}
We see that because at most one of the $\alpha_{d-1},\beta_{d-1}$ is non-zero, and $| \alpha^2 |, | \beta^2 | \leq \ell$, the first term is $1 + O_d(d^{-1})$. There are at most $2 \ell$ terms in the rest of the product and each is of order $1 + O_d((d-j)^{-1})$. Noting that by assumption $| \beta^{j}| = 0$ for $ j > d - \sqrt{d}$, and the corresponding term cancel out. We deduce that this product is of order $1 + O_d (d^{-1/2})$. 

Similarly, on the second line of Eq.~\eqref{eq:decompo_Malphabeta}, there are at most $2\ell$ terms that are of order $1 + O_d((d-j)^{-1})$, with terms cancelling each other as soon as $j > d-\sqrt{d}$. We deduce that this second term is also of order $1 + O_d (d^{-1/2})$. We conclude that $M_{\alpha,\bbeta} = 1 + O_d (d^{-1/2})$, which finishes the proof.
\end{proof}

\clearpage

\section{Proof of Theorem \ref{thm:KRR}: asymptotic characterization of KRR}\label{app:proof_main_KRR}

In this Appendix, we prove the asymptotic characterization of kernel ridge regression in the polynomial regime, described in Theorem \ref{thm:KRR}. Throughout the proof, we will denote $\bDelta$ any matrix with $\| \bDelta \|_{\op} = o_{d,\P}(1)$. In particular, $\bDelta$ can change from line to line. 

In Section \ref{sec:outline}, we outline the proof for the asymptotic prediction risk which is split into two parts: convergence in probability over $\bX$ of $\E_{\beps, f_*} [ R_{\text{test}} (f_{*,d}; \bX , \beps , \lambda )]$ to the asymptotic risk (proved in Section \ref{sec:proof_risk_decompo}) and convergence in $L^1$ over $\bX,\beps,f_*$ of $R_{\text{test}} (f_{*,d}; \bX , \beps , \lambda )$ to $\E_{\beps, f_*} [ R_{\text{test}} (f_{*,d}; \bX , \beps , \lambda )]$ (proved in Section \ref{sec:proof_risk_concentration}). The proofs for the training error and RKHS norm are very similar and we outline the main steps in Section \ref{sec:training_error}. Finally, the proof of some of the more technical claims are deferred to Section \ref{sec:technical_KRR}.

\subsection{Outline of the proof}\label{sec:outline}

In this section, we focus on the test error (we will write $f_* = f_{*,d}$ for simplicity):
\[
\begin{aligned}
 R_{\text{test}} (f_\star ; \bX , \beps,\lambda ) =&~ \E_{\bx} \big[ \big( f_* (\bx) - \hf ( \bx ; \hba_\lambda) \big)^2 \big] \, ,
\end{aligned}
\]
where we recall that the kernel ridge regression solution is given by
\[
\hf ( \bx ; \hba_\lambda) = \by^\sT ( \bH + \lambda \id_n)^{-1} \bh ( \bx) \, ,
\]
with $\bH = ( h ( \< \bx_i , \bx_j \>/d ))_{ij \in [n]}$, $\by = (y_1 , \ldots , y_n)$ and $\bh (\bx) = ( h (\<\bx , \bx_i \> /d) )_{i \in [n]}$.

We will decompose $\cA_d$ into three orthogonal subspaces and bound the risk along each of them: 
\begin{equation}\label{eq:Ad_decompo}
L^2(\cA_d) = \cV_{\leq \ell - 1} \oplus \cV_{\ell} \oplus \cV_{>\ell}\, ,
\end{equation}
where $\cV_{\leq \ell -1}$ is the subspace spanned by polynomials of degree $\leq \ell -1$, $\cV_{\ell}$ is the subspace spanned by polynomials of degree $\ell$ orthogonal to $\cV_{\leq \ell -1}$, and $\cV_{>\ell}$ is the subspace of all functions in $L^2(\cA_d)$ orthogonal to polynomials of degree $\leq \ell$. Recall that we denote $\proj_{\leq\ell}:L^2(\cA_d) \to L^2(\cA_d)$ the orthogonal projection onto $\cV_{\leq \ell -1}$. Define $\proj_{> \ell} = \id - \proj_{\leq \ell}$, and $\proj_{\ell} = \proj_{>\ell - 1} \proj_{\leq \ell}$ the orthogonal projections onto $\cV_{>\ell}$ and $\cV_{\ell}$ respectively. By the orthogonal decomposition \eqref{eq:Ad_decompo} of $\cA_d$ (in $L^2 (\cA_d)$ sense), we have directly:
% We will characterize the contribution to the prediction error of each of the three subspaces $\proj_{\leq \ell-1}, \proj_{\ell}$ and $\proj_{>\ell}$:
\[
\begin{aligned}
 R_{\text{test}} (f_\star ; \bX , \beps,\lambda ) =&~ 
 \| \proj_{\leq \ell-1}(f_* - \hf ( \cdot ; \hba_\lambda))\|_{L^2}^2 + \| \proj_{\ell}(f_* - \hf ( \cdot ; \hba_\lambda))\|_{L^2}^2 + \| \proj_{> \ell}(f_* - \hf ( \cdot ; \hba_\lambda))\|_{L^2}^2\, .
\end{aligned}
\]

Let us recall and introduce some new notations. Define $B_{\ell - 1} = \sum_{k =0}^{\ell-1} B(\cA_d, k) $ (note that $B_{\leq \ell - 1} = \Theta_d (d^{\ell - 1})$) and
\[
\begin{aligned}
\bD_k =&~ \xi_{d,k} ( h) \id_{B(\cA_d , k)}\, , \\
\bY_k ( \bx) = &~ ( Y_{ks} (\bx) )_{s \in [B(\cA_d,k)]} \in \R^{B(\cA_d,k)} \, , \\
\bY_{k} = &~ ( Y_{ks} (\bx_i) )_{i \in [n], s \in [B(\cA_d,k)]} \in \R^{n \times B(\cA_d,k)}\, ,\\
\bY_{\leq \ell-1} (\bx) = &~ ( \bY_k (\bx) )_{k \leq \ell -1} \in \R^{B_{\ell-1}} \, ,\\
\bY_{\leq \ell - 1} = &~ [\bY_{0}^\sT , \ldots , \bY_{\ell-1}^\sT ]^\sT \in \R^{n \times B_{\ell-1}}, \\
\bD_{\leq \ell - 1} = &~ \diag ( \bD_0 , \ldots , \bD_{\ell-1} ) \in \R^{B_{\ell-1} \times B_{\ell-1}} \, .
\end{aligned}
\]
Furthermore, introduce $\bbeta_* = ( \beta^*_{ks} )_{k \leq \ell -1, s \in [B(\cA_d , k) ]} \in \R^{B_{\ell -1}} $ and $\tbbeta_k = ( \tbeta_{ks} )_{s \in [B(\cA_d , k )]} \in \R^{B(\cA_d , k)}$ for $k \geq \ell$. Denote $\boldf_* = ( f_* (\bx_1 ) , \ldots , f_* ( \bx_n) )$ and $\beps = (\eps_1 , \ldots, \eps_n)$ so that $\by = \boldf_* + \beps$. We can rewrite $\boldf_* = \bY_{\leq \ell - 1} \bbeta_* +  \bY_{\ell} \tbbeta_{\ell} + \boldf_{ > \ell}$ where  
\[
\boldf_{> \ell} = \sum_{ k = \ell+1}^\infty  \bY_k \tbbeta_k \, .
\]

Recall that we can decompose the inner-product kernel in terms of Gegenbauer polynomials associated to $\cA_d$:
\[
h ( \< \bx_1 , \bx_2 \> / d) = \sum_{k = 0}^\infty \xi_{d,k} (h) B(\cA_d, k) Q_k^{(d)} ( \< \bx , \by \>)  = \sum_{k = 0}^\infty \mu_{d,k} (h) Q_k^{(d)} ( \< \bx , \by \>) \, ,
\]
where we recall we denoted $\mu_{d,k} = \xi_{d,k} (h) B(\cA_d, k)$.
We introduce the matrix $\bM := \E_{\bx} [ \bh (\bx) \bh (\bx)^{\sT}]$, and we denote below $\xi_k := \xi_{d,k}(h)$ and $\mu_{d,k} = \mu_{d,k} (h)$ for simplicity. The vector and matrices can be decomposed in the polynomial basis as
\[
\begin{aligned}
\bh ( \bx) = &~ \bY_{\leq \ell -1 } \bD_{\leq \ell -1 }\bY_{\leq \ell -1 } (\bx) +\xi_{\ell} \bY_{\ell } \bY_{\ell } (\bx) + \bh_{>\ell} ( \bx) \, , \\
\bH =&~ \bY_{\leq \ell -1 } \bD_{\leq \ell -1 }\bY_{\leq \ell -1 }^\sT +\xi_{\ell} \bY_{\ell } \bY_{\ell }^\sT + \bH_{>\ell} \, , \\
\bM =&~ \bY_{\leq \ell -1 } \bD_{\leq \ell -1 }^2\bY_{\leq \ell -1 }^\sT +\xi_{\ell}^2 \bY_{\ell } \bY_{\ell }^\sT + \bM_{>\ell} \, ,
\end{aligned}
\]
where
\[
\bh_{>\ell} ( \bx) = \sum_{k = \ell+1}^\infty \xi_k  \bY_{k } \bY_{k} (\bx)\, , \qquad \bH_{>\ell} = \sum_{k = \ell+1}^\infty \xi_k  \bY_{k } \bY_{k}^\sT\, ,  \qquad \bM_{>\ell} = \sum_{k = \ell+1}^\infty \xi_k^2  \bY_{k } \bY_{k}^\sT \, .
\]
Recall that we denote $\bQ_k = B(\cA_d , k)^{-1} \bY_k \bY_k^\sT$ the matrix of the $k$-th Gegenbauer polynomial evaluated on the inner-product of the inputs. We will further denote:
\[
\bXi = ( \bH + \lambda \id_n )^{-1}\, .
\]

By Theorem 6 in \cite{mei21generalization} (see also Proposition \ref{prop:GMMM_prop3} and Corollary \ref{cor:bound_exp_op} in Section \ref{sec:technical_bound_expectation}), the high-degree component of the kernel matrices satisfy
\begin{equation}
    \begin{aligned}
    \bH_{>\ell} = &~ \kappa_H (\id + \Delta_H ) \, , \\
    \bM_{>\ell}  = &~ \kappa_M (\id + \Delta_M ) \, , \\
    \end{aligned}
\end{equation}
where $\max ( \| \Delta_H \|_{\op} , \| \Delta_M \|_{\op} ) = o_{d,\P}(1) $, and
\[
\kappa_H = \sum_{k \geq \ell+1} \xi_k B(\cA_d , k) =  \sum_{k \geq \ell+1} \mu_{d,k} \, , \qquad \kappa_M = \sum_{k \geq \ell+1} \xi_k^2 B(\cA_d , k)\, .
\]

We first compute $\E_{\beps, f_*} [ R_{\text{test}} (f_* ; \bX , \beps, \lambda )]$, the expected test error with the expectation taken over $f_*$ and $\beps$. The following proposition characterizes the convergence of the test error on each of the three subspaces \eqref{eq:Ad_decompo} as $d \to \infty$, where the convergence is in probability with respect to $\bX$: 

\begin{proposition}\label{prop:risk_decompo}
Follow the assumptions and notations of Theorem \ref{thm:KRR}. We get the following expressions for the test error on the different subspaces: (where $o_{d,\P}(\cdot)$ is with respect to the randomness on $\bX$)
\begin{enumerate}
    \item[1.] On $\cV_{<\ell}$:
    \[
    \E_{f_*,\eps} \big[ \| \proj_{\leq \ell - 1} (f_* - \hf (\cdot ; \hba_{\lambda} ) )\|_{L^2}^2 \big] = o_{d,\P} (1) \, .
    \]
    \item[2.] On $\cV_{>\ell}$:
    \[
    \E_{f_*,\eps} \big[ \| \proj_{> \ell} (f_* - \hf (\cdot ; \hba_{\lambda} ) )\|_{L^2}^2 \big] = F_{>\ell}^2 + o_{d,\P} (1) \, .
    \]
    \item[3.] On $\cV_{\ell}$: denoting $\zeta_* = \frac{\mu_{>\ell} + \lambda}{\mu_\ell}$, we have
    \[
    \E_{f_*,\eps} \big[ \| \proj_{ \ell} (f_* - \hf (\cdot ; \hba_{\lambda} ) )\|_{L^2}^2 \big] = F_\ell^2 \cdot \cB ( \psi , \zeta_*)  +  (F_{>\ell}^2 + \sigma_\eps^2) \cdot \cV(\psi, \zeta_*)+o_{d,\P} (1) \, .
    \]
\end{enumerate}
\end{proposition}

In order to get the convergence in probability with respect to $f_*,\beps$, we show that the test error $R_{\text{test}} (f_* ; \bX , \beps, \lambda )$ converges to $\E_{\beps, f_*} [ R_{\text{test}} (f_* ; \bX , \beps, \lambda )]$ in $L^1$ over $f_*,\beps,\bX$:

\begin{proposition}[Convergence to expectation]\label{prop:risk_concentration}
Under the assumptions of Theorem \ref{thm:KRR}, we have:
\[
\E_{\bX, f_*, \beps} \Big[ \Big\vert R_{\text{test}} (f_* ; \bX , \beps , \lambda ) - \E_{f_*,\beps} \big[R_{\text{test}} (f_* ; \bX , \beps , \lambda )\big] \Big\vert \Big] =  o_{d}(1)\, .
\]
\end{proposition}

The proofs of Propositions \ref{prop:risk_decompo} and \ref{prop:risk_concentration} can be found in Sections \ref{sec:proof_risk_decompo} and \ref{sec:proof_risk_concentration} respectively. The characterization of the test error in Theorem \ref{thm:KRR} follows directly from these two propositions.

\subsection{Proof of Proposition \ref{prop:risk_decompo}}\label{sec:proof_risk_decompo}

\noindent
{\bf Step 1. Bounding the contribution of $\proj_{\leq \ell -1} ( f_* - \hf ( \cdot ; \hba_\lambda))$.}

We decompose the contribution along $\proj_{\leq \ell -1}$ as follows:
\begin{equation}\label{eq:first_term_KRR}
\begin{aligned}
\E_{f_* , \beps} \Big[  \big\| \proj_{\leq \ell-1}(f_* - \hf ( \cdot ; \hba_\lambda))\big\|_{L^2}^2\Big]  = &~\E_{f_* , \beps} \Big[  \big\|\bbeta_* - \by^\sT \bXi \bY_{\leq \ell -1} \bD_{\leq \ell -1} \big\|_2^2\Big] \\
=&~ B_{11} + B_{12} + V_{1} \, ,
\end{aligned}
\end{equation}
where
\[
\begin{aligned}
B_{11} =&~ \big\| ( \id -  \bY_{\leq \ell-1}  \bXi \bY_{\leq \ell -1} \bD_{\leq \ell -1}  ) \bbeta_* \big\|_2^2 \, ,\\
B_{12} =&~ \sum_{k = \ell}^\infty  F_{d,k}^2 \Tr \big[ \bQ_{k} \bXi \bY_{\leq \ell -1} \bD_{\leq \ell -1}^2  \bY_{\leq \ell -1} \bXi \big] \, , \\
V_{1} = &~ \sigma_\eps^2 \Tr \big[ \bXi^2 \bY_{\leq \ell -1} \bD_{\leq \ell -1}^2  \bY_{\leq \ell -1} \big] \, .
\end{aligned}
\]
The terms $B_{11}+ B_{12}$ and $V_1$ correspond respectively to the bias and variance of the kernel estimator along the subspace $\proj_{\leq \ell -1}$. 

First from Lemma \ref{lem:YXiYD}, we get
\begin{equation}\label{eq:B11}
B_{11} = \| \bbeta_* \|_2^2 \cdot o_{d,\P} (1) = o_{d,\P} (1) \, .
\end{equation}
For the second term, notice that by Theorem 6 in \cite{mei21generalization} (see also Proposition \ref{prop:GMMM_prop3} and Corollary \ref{cor:bound_exp_op} in Section \ref{sec:technical_bound_expectation}), we have
\[
\bF_{\geq \ell} := \sum_{k = \ell}^\infty  F_{d,k}^2  \bQ_{k} = F_{d,\ell}^2 \bQ_{\ell} + F_{>\ell}^2 \cdot ( \id_n + \Delta) \, .
\]
Using $\| \bQ_\ell \|_\op = O_{d,\P}(d^\delta)$ for any $\delta >0$ by Lemma \ref{lem:bound_Ql}, we can apply Eq.~\eqref{eq:XiMXi_3} of Lemma \ref{lem:XiMXi} and obtain
\begin{equation}\label{eq:B12}
B_{12} = \Tr \big[ \bA \bXi \bY_{\leq \ell -1} \bD_{\leq \ell -1}^2  \bY_{\leq \ell -1} \bXi \big] = o_{d,\P} (1) \, .
\end{equation}
Similarly, by Eq.~\eqref{eq:XiMXi_3} of Lemma \ref{lem:XiMXi} with $\bA = \id_n$, we get
\begin{equation}\label{eq:V1}
\begin{aligned}
V_1 = \sigma_\eps^2 \Tr \big[ \bXi \bY_{\leq \ell -1} \bD_{\leq \ell -1}^2  \bY_{\leq \ell -1} \bXi \big] = o_{d,\P} (1) \, .
\end{aligned}
\end{equation}

Combining Eqs.~\eqref{eq:B11}, \eqref{eq:B12} and \eqref{eq:V1} in Eq.~\eqref{eq:first_term_KRR} yields the first of the three contributions:
\begin{equation}\label{eq:bound_leqell}
\E_{f_* , \beps} \Big[  \big\| \proj_{\leq \ell-1}(f_* - \hf ( \cdot ; \hba_\lambda))\big\|_{L^2}^2\Big] = o_{d,\P}(1) \, .
\end{equation}

% \notate{Find out what to do with below}
% \[
% B_{12} = \sum_{k = \ell}^\infty  F_{d,k}^2 \Tr \big[ \bQ_{k}  \bY_{\leq \ell -1}  \bY_{\leq \ell -1}^\sT \big]/n^2 + o_{d,\P}(1) \cdot \left( F_\ell^2 + F_{>\ell}^2 \right) \, ,
% \]
% where we used that $\Tr [ \bQ_{k} ] = n$. Notice further that 
% \[
% \E_{\bX} \left[ \sum_{k = \ell}^\infty  \frac{F_{d,k}^2}{B(\cA_d,k)} \Tr \big[ \bY_k \bY_k^\sT  \bY_{\leq \ell -1}  \bY_{\leq \ell -1}^\sT \big]/n^2 \right] = \sum_{s = 0 }^{\ell-1} \sum_{k = \ell}^\infty F_{d,k}^2  \frac{B(\cA_d, s)}{n} =  o_{d,\P}(1)\, ,
% \]
% where we used that $B(\cA_d , s)/n = O_d( d^{-1})$ for $s\leq \ell -1$, and that for $k \neq s$ (see for example Lemma E.2 in \cite{ghorbani2021linearized})
% \begin{equation}\label{eq:exp_YYYY}
% \E_{\bX} \big[  \bY_k^{\sT} \bY_{s} \bY_s^\sT \bY_k \big] = n B(\cA_d,s) \cdot \id_{B(\cA_d , k)} \, .
% \end{equation}
% \notae{above}

\noindent
{\bf Step 3. Bounding the contribution of $\proj_{>\ell} ( f_* - \hf ( \cdot ; \hba_\lambda))$.}

We will in fact show that $\E_{f_*, \beps} \big[ \| \proj_{>\ell} \hf ( \cdot ; \hba_\lambda)) \|_{L^2}^2 \big] = o_{d,\P}(1)$. This implies that
\begin{equation}\label{eq:bound_geqell}
\begin{aligned}
\E_{f_* , \beps} \Big[  \big\| \proj_{>\ell}(f_* - \hf ( \cdot ; \hba_\lambda))\big\|_{L^2}^2\Big] =&~ \E_{f_* , \beps} \Big[  \big\| \proj_{>\ell}f_* \big\|_{L^2}^2\Big]  + \E_{f_*, \beps} \big[ \| \proj_{>\ell} \hf ( \cdot ; \hba_\lambda)) \|_{L^2}^2 \big] \\
&~ + 2 \cdot \E_{f_* , \beps} \Big[  \big\| \proj_{>\ell}f_* \big\|_{L^2}^2\Big]^{1/2} \E_{f_*, \beps} \big[ \| \proj_{>\ell} \hf ( \cdot ; \hba_\lambda)) \|_{L^2}^2 \big]^{1/2} \\
=&~ \E_{f_* , \beps} \Big[  \big\| \proj_{>\ell}f_* \big\|_{L^2}^2\Big]  + o_{d,\P} (1) \\
=&~ F_{>\ell}^2 + o_{d,\P} (1) \, .
\end{aligned}
\end{equation}
We have simply 
\[
\begin{aligned}
\E_{f_* , \beps} \Big[  \big\| \proj_{>\ell} \hf ( \cdot ; \hba_\lambda)\big\|_{L^2}^2\Big]  = &~ \sum_{k = \ell+1}^\infty \E_{f_* , \beps}  \Big[  \big\| \xi_k \by^\sT \bXi \bY_{k}  \big\|_2^2\Big] \\
=&~ \E_{f_* , \beps}  \Big[  \by^\sT  \bXi \bM_{>\ell} \bXi \by  \Big] \leq \|  \bXi \bM_{>\ell} \bXi \|_{\op} \E_{f_* , \beps}  \cdot \big[  \| \by \|_2^2 \big] \, .
\end{aligned}
\]
Note that
\[
\E_{f_* , \beps, \bX}  \big[  \| \by \|_2^2 \big] = n \left( \sigma_\eps^2 + \E_{f_*} \big[ \| f_* \|_{L^2}^2 \big] \right) = O_d ( n ) \, ,
\]
and therefore $\E_{f_* , \beps}  \big[  \| \by \|_2^2 \big] = n \cdot O_{d,\P}(1)$ by Markov's inequality. We deduce using bound \eqref{eq:XiMXi_2} from Lemma \ref{lem:XiMXi},
\[
\E_{f_* , \beps} \Big[  \big\| \proj_{>\ell} \hf ( \cdot ; \hba_\lambda)\big\|_{L^2}^2\Big] = n\|  \bXi \bM_{>\ell} \bXi \|_{\op} \cdot O_{d,\P} (1) = o_{d,\P} (1) \, .
\]

\noindent
{\bf Step 4. Bounding the contribution of $\proj_{\ell } ( f_* - \hf ( \cdot ; \hba_\lambda))$.}

We decompose the contribution along $\proj_{\ell}$ as follows:
\[
\begin{aligned}
\E_{f_* , \beps} \Big[  \big\| \proj_{ \ell}(f_* - \hf ( \cdot ; \hba_\lambda))\big\|_{L^2}^2\Big]  = &~\E_{f_* , \beps} \Big[  \big\|\tbbeta_\ell - \xi_{\ell} \by^\sT \bXi \bY_{\ell} \big\|_2^2\Big]
= B_{21} + B_{22} + B_{23} + V_{2} \, ,
\end{aligned}
\]
where
\[
\begin{aligned}
B_{21} =&~ F_{d,\ell}^2 \Big\{ 1 - 2 \xi_\ell \cdot \Tr ( \bXi \bQ_{\ell} ) + \xi_{\ell}^2 B(\cA_d, \ell) \cdot \Tr ( \bXi \bQ_{\ell} \bXi \bQ_\ell ) \Big\}   \, ,\\
B_{22} =&~  \sum_{k =\ell+1}^\infty F_{d,k}^2 \xi_{\ell}^2 B(\cA_d, \ell) \cdot \Tr ( \bXi \bQ_{\ell} \bXi \bQ_k ) \, , \\
B_{23} =&~ \xi_{\ell}^2 \| \bY^\sT_{\ell} \bXi \bY_{\leq \ell -1 } \bbeta_* \|_2^2  \, , \\
V_{2} = &~ \sigma_\eps^2 \xi_{\ell}^2 B (\cA_d, \ell) \cdot \Tr \big[ \bXi^2 \bQ_{\ell}  \big] \, .
\end{aligned}
\]

Let us first show that $B_{23}$ goes to zero in probability. Using $\| \bQ_\ell \|_{\op} = O_{d,\P}(d^\delta)$ for any $\delta >0$,  
\[
    \begin{aligned}
    B_{23} =&~ \xi_\ell^2 \bbeta_*^\sT \bY_{\leq \ell-1}^\sT \bXi \bY_{\ell} \bY_{\ell}^\sT\bXi \bY_{\leq \ell - 1} \bbeta_* \\
    \leq &~ \frac{\mu_{d,\ell}^2}{B(\cA_d, \ell)} \| \bQ_\ell \|_{\op} \| \bbeta_* \|_2^2  \| \bY_{\leq \ell-1}^\sT \bXi^2 \bY_{\leq \ell - 1} \|_{\op} \\
    =&~ O_{d,\P} (n^{-1}d^\delta) \cdot \| \bY_{\leq \ell-1}^\sT \bXi^2 \bY_{\leq \ell - 1} \|_{\op} \, .
    \end{aligned}
\]
From Eq.~\eqref{eq:SMW_useful}, we have
\[
 \bY_{\leq \ell-1}^\sT \bXi^2 \bY_{\leq \ell - 1} = \bD_{\leq \ell - 1}^{-1} \bL \bY_{\leq \ell-1}^\sT \bA^{-2} \bY_{\leq \ell-1} \bL  \bD_{\leq \ell - 1}^{-1} / n^2 \, ,
\]
 where we denoted $\bL := \big[ (n\bD_{\leq \ell - 1})^{-1} + \bY_{\leq \ell - 1 }^\sT \bA^{-1} \bY_{\leq \ell - 1}/ n \big]^{-1}$ and $\bA := \sum_{k \geq \ell} \mu_{d,k} \bQ_k + \lambda \id_n$. From Eq.~\eqref{eq:bound_LYAAYLop} in Proposition \ref{prop:E_bound_for_var}, we obtain
 \begin{equation}
 \begin{aligned}
    B_{23} = &~  O_{d,\P} (d^\delta n^{-1})\cdot \| \bD_{\leq \ell -1}^{-1} \|_\op^2 \cdot \|\bL \bY_{\leq \ell-1}^\sT \bA^{-2} \bY_{\leq \ell-1} \bL \|_\op / n^2 \\
    =&~ O_{d,\P} (d^\delta n^{-2}) \cdot O_d (d^{2(\ell -1)}) = o_{d,\P}(1)\, ,
    \end{aligned}
 \end{equation}
where we used that $\| \bD_{\leq \ell -1 } \|_\op= \max_{k<\ell} \xi_k^{-1} = O_d (d^{\ell - 1})$.

Let us now simplify $B_{22}$: applying Theorem 6 in \cite{mei21generalization}, we have 
\[
\sum_{k = \ell+1}^\infty F_{d,k}^2 \bQ_k = \left( \sum_{k > \ell} F_{d,k}^2 \right)\cdot \big[ \id_n + \Delta \big] = F_{>\ell}^2 \big[ \id_n + \Delta \big] \, ,
\]
and therefore 
\[
B_{22} = \xi_{\ell}^2 B(\cA_d , \ell) \cdot (1 + o_{d,\P}(1)) \cdot \Tr ( \bXi \bQ_\ell \bXi ) \, .
\]

Denote $\bR ( \zeta) = ( \bQ_\ell + \zeta \id_n)^{-1}$ and $\zeta_* = (\lambda + \mu_{>\ell})/\mu_{\ell}$. Notice that because 
\[
\| \bXi \|_{\op}, \| \bQ_\ell \bXi \|_{\op}, \| \bQ_\ell \bR(\zeta) \|_{\op}, \| \bR(\zeta) \|_{\op}  = O_{d,\P}(1)\, ,
\]
we can use Lemma \ref{lem:replace_resolvent} and simplify the expression of the different terms:
\[
\begin{aligned}
B_{21} =&~ F_{d,\ell}^2 \Big\{ 1 - 2 \frac{\xi_\ell}{\mu_{d,\ell}} \cdot \Tr \big[ \bR(\zeta_*) \bQ_{\ell} \big]  + \frac{\xi_{\ell}^2 B(\cA_d, \ell)}{\mu_{d,\ell}^2} \cdot \Tr \big[ \bR(\zeta_*) \bQ_{\ell} \bR(\zeta_*) \bQ_\ell \big]   + o_{d,\P}(1)\Big\}  \\
=&~ F_{\ell}^2 \Big\{ 1 - 2 \frac{\psi}{n} \Tr \big[ \bR(\zeta_*) \bQ_{\ell} \big]  + \frac{\psi}{n}\Tr \big[ \bR(\zeta_*) \bQ_{\ell} \bR(\zeta_*) \bQ_\ell \big]   \Big\} + o_{d,\P}(1) \, , \\
B_{22} =&~ F_{>\ell}^2 \frac{\psi}{n} \Tr \big[ \bR(\zeta_*) \bQ_{\ell} \bR(\zeta_*) \big] + o_{d,\P}(1) \, , \\
V_2 =&~ \sigma_\eps^2 \frac{\psi}{n} \Tr \big[ \bR(\zeta_*) \bQ_{\ell} \bR(\zeta_*) \big] + o_{d,\P}(1) \, .
\end{aligned}
\]
Finally, Lemma \ref{lem:calculus_resolvent} shows that 
\[
\begin{aligned}
\Big\{ 1 - 2 \frac{\psi}{n} \Tr \big[ \bR(\zeta_*) \bQ_{\ell} \big]  + \frac{\psi}{n}\Tr \big[ \bR(\zeta_*) \bQ_{\ell} \bR(\zeta_*) \bQ_\ell \big]   \Big\} =&~ \cB ( \psi , \zeta_* ) +o_{d,\P}(1)\, , \\
\frac{\psi}{n} \Tr \big[ \bR(\zeta_*) \bQ_{\ell} \bR(\zeta_*) \big] = &~ \cV ( \psi , \zeta_* ) +o_{d,\P}(1) \, ,
\end{aligned}
\]
which concludes the proof.

\subsubsection{Technical results: bounds in probability}
\label{sec:technical_KRR}

\begin{lemma}\label{lem:YXiYD} Follow the assumptions and notations in the proof of Theorem \ref{thm:KRR}. We have
\begin{equation}
\big\|  \bY_{\leq \ell-1}^\sT  \bXi \bY_{\leq \ell -1} \bD_{\leq \ell -1} - \id_{B_{\ell-1}}  \big\|_{\op} = o_{d,\P} (1) \, .
\end{equation}
\end{lemma}

\begin{proof}[Proof of Lemma \ref{lem:YXiYD}]
Recall that we can decompose
\[
\bH + \lambda \id_n = \bY_{\leq \ell - 1} \bD_{\leq \ell - 1} \bY_{\leq \ell - 1}^\sT + \xi_\ell\bY_{\ell} \bY_{\ell}^\sT  + ( \kappa_H + \lambda) (\id_n + \bDelta) \, ,
\]
where $\| \bDelta \|_{\op} = o_{d,\P}(1)$. By the Sherman-Morrison-Woodbury formula, we have
\[
 \bY_{\leq \ell-1}^\sT  \bXi \bY_{\leq \ell -1} \bD_{\leq \ell -1} =  \bY_{\leq \ell-1}^\sT \bE^{-1} \bY_{\leq \ell -1} \bR / n \, ,
\]
where $\bE = \xi_\ell/( \kappa_H + \lambda) \cdot  \bY_{\ell} \bY_{\ell}^\sT +  \id_n + \bDelta$ and 
\[
\bR = \big[ ( \kappa_H + \lambda) (n \bD_{\leq \ell -1} )^{-1} + \bY_{\leq \ell -1}^\sT \bE^{-1} \bY_{\leq \ell - 1} /n  \big]^{-1}\, .
\]
From the above formula, we deduce that
\[
 \bY_{\leq \ell-1}^\sT  \bXi \bY_{\leq \ell -1} \bD_{\leq \ell -1}  - \id = - ( \kappa_H + \lambda) (n \bD_{\leq \ell -1} )^{-1} \bR \, .
\]
Notice that $\sigma_{\min} ( \bE^{-1} ) \geq \Omega_{d,\P} (1 / \| \bQ_\ell \|_\op ) = \Omega_{d,\P} (1)$ and recall $\| ( \kappa_H + \lambda) (n \bD_{\leq \ell -1} )^{-1} \|_{\op} = o_{d,\P}(1)$. We have $\| \bY^\sT_{\leq \ell-1} \bY_{\leq \ell -1} / n - \id_{B-1} \|_{\op} = o_{d,\P}(1)$ (see Lemma \ref{lem:YY_concentration}). We deduce that $\| \bR \|_{\op} = O_{d,\P}(1)$ and
\[
\| \bY_{\leq \ell-1}^\sT  \bXi \bY_{\leq \ell -1} \bD_{\leq \ell -1}  - \id  \|_{\op} = o_{d,\P} (1)\, ,
\]
as desired.

% It will be useful (for the next lemma) to show the stronger result $\| \bY_{\leq \ell -1}^\sT \bE^{-1} \bY_{\leq \ell - 1} /n  - \id_{B_{\ell - 1}} \|_{\op} = o_{d,\P} (1) $. Using again Sherman-Morrison-Woodbury formula, we have
% \[
% \begin{aligned}
% \bY_{\leq \ell -1}^\sT \bE^{-1} \bY_{\leq \ell - 1} /n  =&~ \bY_{\leq \ell -1}^\sT \big[ \id_{B_{\ell - 1}} + \xi_\ell/( \kappa_H + \lambda) \cdot  \xi_\ell\bY_{\ell} \bY_{\ell}^\sT \big]^{-1} \bY_{\leq \ell - 1} /n  + \bDelta \\
% =&~ \bY_{\leq \ell -1}^\sT  \bY_{\leq \ell - 1} /n - \bY_{\leq \ell -1}^\sT \bY_{\ell} \big[ ( \kappa_H + \lambda) / \xi_{\ell} \cdot \id_{B_{\ell - 1}} + \bY_{\ell}^\sT \bY_{\ell} \big]^{-1} \bY_{\ell}^\sT \bY_{\leq \ell - 1} /n   + \bDelta \, ,
% \end{aligned}
% \]
% where we used that $\| b$

\end{proof}

\begin{lemma}\label{lem:XiMXi} Follow the assumptions and notations in the proof of Theorem \ref{thm:KRR}. We have
\begin{align}
    \big\|  \bY_{\leq \ell -1}^\sT \bXi \bY_{\leq \ell -1} \bD_{\leq \ell -1}^2 \bY_{\leq \ell -1}^\sT  \bXi \bY_{\leq \ell -1}  - \id_{B_{\ell - 1}} \big\|_{\op} =&~ o_{d,\P} (1)\label{eq:XiMXi_1} \, , \\
   n  \big\|    \bXi \bM_{>\ell}  \bXi \big\|_{\op} =&~ O_{d,\P} (d^{-1}) \label{eq:XiMXi_2} \, .
\end{align}
and for any symmetric matrix $\bA \in \R^{n \times n}$ such that $\| \bA \|_{\op} = O_{d,\P}(d^\delta)$ for some $\delta <1$, 
\begin{align}
\Tr \big[  \bA \bXi \bY_{\leq \ell -1} \bD_{\leq \ell -1}^2 \bY_{\leq \ell -1}^\sT  \bXi \big] =&~ o_{d,\P} (1) \, .\label{eq:XiMXi_3} 
\end{align}
\end{lemma}

\begin{proof}[Proof of Lemma \ref{lem:XiMXi}] The first bound \eqref{eq:XiMXi_1} follows simply from Lemma \ref{lem:YXiYD}. For the second bound \ref{eq:XiMXi_2}, we have
\[
n  \big\|    \bXi \bM_{>\ell}  \bXi \big\|_{\op}  = O_{d,\P} (1) \cdot \frac{n \kappa_M}{(\kappa_H + \lambda)^2}= O_{d,\P} (1)  \cdot \sup_{k \geq \ell +1} n \xi_k^2 = O_{d,\P} (d^{-1}) \, ,
\]
where we used that $\bM_{>\ell} = \kappa_M (\id + \Delta )$ and $\| \bXi \|_{\op} = (\kappa_H + \lambda)^{-1} \cdot O_{d,\P}(1)$.

For the third bound \eqref{eq:XiMXi_3}, we follow some of the notations introduced in the proof of Lemma \ref{lem:YXiYD}. In particular, by Sherman-Morrison-Woodbury formula, we have
\[
\begin{aligned}
\Tr \big[  \bA \bXi \bY_{\leq \ell -1} \bD_{\leq \ell -1}^2 \bY_{\leq \ell -1}^\sT  \bXi \big] =&~ \Tr \big[  \bA \bE \bY_{\leq \ell -1} \bR^2 \bY_{\leq \ell -1}^\sT \bE \big]/n^2 \\
=&~ O_{d,\P} (d^\delta) \cdot \Tr \big[   \bY_{\leq \ell -1}\bY_{\leq \ell -1}^\sT \big]/n^2\\
=&~ O_{d,\P} (d^\delta) \cdot \frac{n B_{\ell-1}}{n^2} = o_{d,\P} (1) \, ,
\end{aligned}
\]
where we used $\| \bA \|_{\op} , \| \bE \|_{\op} , \| \bR \|_{\op} = O_{d,\P} (1)$ in the second line, and Markov's inequality on the last line with $\E_{\bX} \big[ \Tr [   \bY_{\leq \ell -1}\bY_{\leq \ell -1}^\sT ] \big] = n B_{\ell-1}$.
\end{proof}

\begin{lemma}\label{lem:replace_resolvent}
Follow the assumptions and notations in the proof of Theorem \ref{thm:KRR}. Let $\bB \in \R^{n \times n}$ a matrix such that $\| \bB \|_{\op} = O_{d,\P}(1)$. Then
\begin{equation}
     \frac{1}{n}\Tr \big[ (\bH + \lambda \id_n)^{-1} \bQ_\ell \bB \big] = \frac{1}{\mu_{\ell} n} \Tr \big[ (\bQ_{\ell} +\zeta_* \id_n)^{-1} \bQ_\ell  \bB \big] + o_{d,\P} (1) \, ,
\end{equation}
where we recall that $\zeta_* = ( \lambda + \mu_{>\ell})/\mu_{\ell} $.
\end{lemma}

\begin{proof}[Proof of Lemma \ref{lem:replace_resolvent}]
Recall that, one can decompose $\bH + \lambda \id_n = \bY_{\leq \ell -1} \bD_{\leq \ell -1} \bY_{\leq \ell -1} + \bA$, where $\bA = \mu_{d,\ell} \bQ_\ell + (\kappa_H + \lambda) (\id_n + \Delta)$. By Sherman-Morrison-Woodbury formula, we have
\[
\begin{aligned}
 &~\frac{1}{n}\Tr \big[ \bXi \bQ_\ell \bB \big]\\
 =&~ \frac{1}{n}\Tr \big[ \bA^{-1} \bQ_\ell \bB \big] - \frac{1}{n^2}\Tr \big[ \bA^{-1} \bY_{\leq \ell - 1} [ (n \bD_{\leq \ell -1} )^{-1} + \bY_{\leq \ell -1}^\sT \bA^{-1} \bY_{\leq \ell -1}/n ]^{-1} \bY_{\leq \ell - 1}^\sT \bA^{-1} \bQ_\ell \bB \big]
\end{aligned}
\]
Using that $\| \bA^{-1} \bY_{\leq \ell - 1} [ (n \bD_{\leq \ell -1} )^{-1} + \bY_{\leq \ell -1}^\sT \bA^{-1} \bY_{\leq \ell -1}/n ]^{-1} \bY_{\leq \ell - 1}^\sT \|_{\op} /n \leq 1$ is of rank $B_{\leq \ell -1 }$ and $\| \bA^{-1} \bQ_{\ell} \|_{\op}, \| \bB \|_\op = O_{d,\P}(1)$, we deduce that 
\[
\frac{1}{n}\Tr \big[ \bXi \bQ_\ell \bB \big] = \frac{1}{n}\Tr \big[ \bA^{-1} \bQ_\ell \bB \big] + O_{d,\P } ( B_{\ell - 1} / n) = \frac{1}{n}\Tr \big[ \bA^{-1} \bQ_\ell \bB \big] +o_{d,\P}(1)\, .
\]
Note that $\bA = \mu_{d,\ell}  [\tbA + \bDelta]$ with $\tbA := \bQ_\ell + \xi_* \id_n $ and $\| \bDelta \|_\op = o_{d,\P}(1)$. We have
\[
\begin{aligned}
\frac{1}{n}\Tr \big[ \bA^{-1} \bQ_\ell \bB \big] - \frac{1}{n \mu_{d,\ell}} \Tr \big[ \tbA^{-1} \bQ_\ell \bB \big] =&~ \frac{1}{n \mu_{d,\ell}} \Tr \big[ (\tbA + \bDelta )^{-1} \bDelta \tbA^{-1} \bQ_\ell \bB \big] \\
=&~ O_{d} (1) \| \bDelta \|_{\op} \| (\tbA + \bDelta )^{-1} \|_{\op} \| \tbA^{-1} \bQ_\ell  \|_\op \| \bB \|_\op = o_{d,\P}(1)\, .
\end{aligned}
\]
Finally $\Tr \big[ \tbA^{-1} \bQ_\ell \bB \big] / n =O_{d,\P}(1)$, hence 
\[
\frac{1}{n \mu_{d,\ell}} \Tr \big[ \tbA^{-1} \bQ_\ell \bB \big] = \frac{1}{n \mu_{\ell}} \Tr \big[ \tbA^{-1} \bQ_\ell \bB \big] + o_{d,\P}(1) \, ,
\]
which concludes the proof.
\end{proof}

Introduce $\bR ( \zeta) = (\bQ_{\ell} + \zeta \id_n )^{-1}$ the resolvent of the Gegenbauer matrix. Denote $r_\psi ( \zeta)$ the Stieljes transform of the Marchenko-Pastur distribution (Eq.~\eqref{eq:MarchenkoPastur}). From Theorem \ref{thm:spectrum_Gegenbauer}, we have $\frac{1}{n} \Tr ( \bR ( \zeta)) = r_\psi (- \zeta) +o_{d,\P}(1)$. Furthermore, $\frac{1}{n} \Tr ( \bR ( \zeta)^2) =  r_\psi '  (-\zeta) +o_{d,\P}(1)$. We recall that $r_\psi ( \zeta)$ is the only positive solution (for $\zeta <0$) of 
\[
\zeta \psi r_\psi ( \zeta)^2 + (\zeta + \psi - 1) r_\psi ( \zeta) + 1 = 0\, .
\]
We get the formula:
\begin{align}
   r_\psi( \zeta)  =&~  \frac{ - \big[(\zeta - \psi - 1)^2 - 4\psi \big]^{1/2} +  1 - \zeta - \psi}{2\zeta \psi} \, , \\
   r_\psi' (\zeta) = &~ - \frac{r_\psi (\zeta)^2 + \psi r_\psi (\zeta)^3 }{\zeta \psi r_\psi (\zeta)^2 - 1}  \, .% - \frac{\xi +1 - \psi}{2 \psi \xi^2} + \frac{1}{2\psi\xi^2}\cdot \frac{\xi +1 + \psi}{1 +2\psi r_\psi( \xi)}\, .
   %\frac{1}{\xi} \cdot \frac{\psi r_\psi( \xi)^2}{1 +2\psi r_\psi( \xi)}\, .
\end{align}

\begin{lemma}\label{lem:calculus_resolvent}
Follow the assumptions of Theorem \ref{thm:KRR}. We have
\begin{align}
    \frac{1}{n} \Tr ( \bR(\zeta) \bQ_\ell ) = &~  1 - \zeta r_\psi (- \zeta) + o_{d,\P}(1) \, , \\
    \frac{1}{n} \Tr ( \bR(\zeta) \bQ_\ell  \bR(\zeta))=  &~ r_\psi (- \zeta) - \zeta r_\psi' (-\zeta) +o_{d,\P}(1)\, , \\
    \frac{1}{n} \Tr ( \bR(\zeta) \bQ_\ell  \bR(\zeta) \bQ_\ell )=  &~1 - 2 \zeta r_\psi(- \zeta) + \zeta^2 r_\psi'( - \zeta) +o_{d,\P}(1)\, .
\end{align}
\end{lemma}

\begin{proof}[Proof of Lemma \ref{lem:calculus_resolvent}]
This follows from $\bR(\zeta) \bQ_\ell = \id_n - \zeta \bR(\zeta)$ and simple algebra.
\end{proof}

From this lemma, we get
\[
\begin{aligned}
1 - 2 \frac{\psi}{n} \Tr \big[ \bR(\zeta_*) \bQ_{\ell} \big]  + \frac{\psi}{n}\Tr \big[ \bR(\zeta_*) \bQ_{\ell} \bR(\zeta_*) \bQ_\ell \big]  =&~ 1 - \psi + \psi \zeta^2 r' (- \zeta) = \cB ( \psi , \zeta_*)\, .
\end{aligned}
\]

\subsection{Convergence to expectation: proof of Proposition \ref{prop:risk_concentration}}\label{sec:proof_risk_concentration}

Proposition \ref{prop:risk_concentration} is a direct implication of the following proposition:

\begin{proposition}\label{prop:risk_bound_variance}
Under the assumptions of Theorem \ref{thm:KRR}, we have
\[
\E_{\bX} \big[ \Var_{f_*,\beps} ( R_{\text{test}} (f_* ; \bX , \beps , \lambda ) | \bX ) \big] = o_d (1)\, .
\]
\end{proposition}

This proposition is proved in Section \ref{sec:proof_variance_risk}, while some more technical bounds in expectation (instead of in probability, as in Section \ref{sec:technical_lemmas}) are deferred to Section \ref{sec:technical_bound_expectation}.

\subsubsection{Proof of Proposition \ref{prop:risk_bound_variance}}\label{sec:proof_variance_risk}

We decompose the risk into
\[
\begin{aligned}
R_{\text{test}} ( f_\star ; \bX , \beps, \lambda) =&~ \E_{\bx} \big[ \big( f_* (\bx) - \by^\sT \bXi \bh(\bx) \big)^2 \big] \\
=&~ T_1 - 2 T_2 -2 T_3 + T_4 + T_5 + T_6 \, , 
\end{aligned}
\]
where
\[
\begin{aligned}
&T_1 = \E_{\bx} [ f_* ( \bx)^2 ]\, , & &T_2 = \boldf^\sT \bXi \bV\, , & &T_3 = \beps^\sT \bXi \bV \, ,\\
&T_4 = \boldf^\sT \bXi \bM \bXi \boldf \, , & &T_5 = \beps^\sT \bXi \bM \bXi \boldf\, , & &T_6 = \beps^\sT \bXi \bM \bXi \beps \, ,
\end{aligned}
\]
where we denoted $\bV = \E_\bx [ \bh(\bx) f_* (\bx) ] $ and $\bM = \E_{\bx} [ h(\bx) h(\bx)^\sT ]$.

Recall that we assumed that $\beps = (\eps_1 , \ldots , \eps_n)$ are iid with $\E [ \eps_i ] = 0$, $\E[\eps_i^2] = \sigma_\eps^2$ and $\E[\eps_i^4 ] = \tau_4 <\infty$, and $f_*$ is a random function in the sense of Assumption \ref{ass:dist_f}, with random independent coefficients $\tbbeta := (\tbeta_{ks})_{k \geq \ell , s \in [B(\cA_d , k)]}$ with $\tbeta_{ks} \sim \normal( 0 , F_{d,k}^2 / B(\cA_d , k ))$. Each of the terms $T_1 , \ldots , T_6$ are quadratic forms in the vectors $(\beps , \tbbeta)$ and we will bound their variance individually:
\[
\E_{\bX} [ \Var_{\tbbeta, \beps} ( T_i | \bX )] = o_d (1) \, , \qquad i = 1 , \ldots , 6 .
\]
This directly imply the claim in Proposition \ref{prop:risk_concentration}.

\noindent
{\bf Step 1. Bounding $T_1$.} 

We have $T_1 = \| \bbeta_* \|_2^2 + \sum_{k \geq \ell} \sum_{s \in [B(\cA_d , k)]} \tbeta_{ks}^2$ and therefore
\begin{equation}\label{eq:T1}
\begin{aligned}
   \E_{\bX} [  \Var_{\tbbeta} ( T_1 | \bX) ] = \sum_{k \geq \ell} \sum_{s \in [B(\cA_d , k)]} \Var (\tbeta_{ks}^2) \leq&~ C \sum_{k \geq \ell} \sum_{s \in [B(\cA_d , k)]} \frac{F_{d,k}^4}{B(\cA_d,k)^2} \\
    =&~ C \sum_{k \geq \ell} F_{d,k}^2 \cdot \max_{k \geq \ell} \frac{F_{d,k}^2}{B(\cA_d,k)} = o_d (1) \, ,
\end{aligned}
\end{equation}
where we used that $\sum_{k \geq \ell} F_{d,k}^2   = O_d(1)$ and $\sup_{k \geq 1} F_{d,k}^2/B(\cA_d,k) = o_d(1)$ by assumption (see for example Lemma \ref{lem:decay_eigenvalues}).

\noindent
{\bf Step 2. Bounding $T_2$.} 

Recall the decomposition $\boldf = \boldf_{< \ell} + \boldf_{\geq \ell} $ with $\boldf_{< \ell} = \bY_{\leq \ell - 1} \bbeta_*$ and $\boldf_{\geq \ell} = \sum_{k\geq \ell} \bY_k \tbbeta_k $. Denote furthermore $\bV = \bV_{<\ell} + \bV_{\geq \ell} $ with $\bV_{<\ell} = \E[ \bh(\bx) \proj_{<\ell} f_* (\bx) ] = \bY_{\leq \ell-1} \bD_{\leq \ell -1} \bbeta_*$ and $\bV_{\geq \ell} = \sum_{k \geq \ell}\xi_k \bY_k \tbbeta_k$. We can decompose the variance of $T_2$ as:
\[
\Var_{\tbbeta} ( T_2 | \bX )\leq 3\Var_{\tbbeta} (  \boldf_{\geq \ell} \bXi \bV_{<\ell}  | \bX ) +  3\Var_{\tbbeta} (  \boldf_{< \ell} \bXi \bV_{\geq \ell} | \bX) + 3 \Var_{\tbbeta} (  \boldf_{\geq \ell} \bXi \bV_{\geq \ell} | \bX )\, .
\]
Let us bound each term separately. First,
\begin{equation}\label{eq:prep_T2_1}
\begin{aligned}
&~ \E_{\bX} \big[ \Var_{\tbbeta} (  \boldf_{\geq \ell} \bXi \bV_{<\ell} |\bX ) \big]\\
= &~ \E_{\bX} \Big[\sum_{k \geq \ell} \frac{F_{d,k}^2}{B(\cA_d, k)} \bbeta_*^\sT \bD_{\leq \ell - 1} \bY_{\leq \ell -1}^\sT \bXi \bY_k \bY_k^\sT  \bY_{\leq \ell -1} \bD_{\leq \ell - 1}\bbeta_* \Big] \\
= &~ \E_{\bX} \big[ \bbeta_*^\sT \bD_{\leq \ell - 1} \bY_{\leq \ell -1 }^\sT \bXi \bF_{\geq \ell} \bXi \bY_{\leq \ell -1 } \bD_{\leq \ell - 1} \bbeta_* \big] \\
=&~ \| \bbeta_* \|_2^2 \cdot \E_{\bX} \big[\| \bF_{\geq \ell} \|_{\op}^2 \big]^{1/2} \E_{\bX} \big[\| \bD_{\leq \ell - 1} \bY_{\leq \ell -1 }^\sT \bXi^2 \bY_{\leq \ell -1 } \bD_{\leq \ell - 1} \|_{\op}^2 \big]^{1/2} \, ,
\end{aligned}
\end{equation}
where we denoted $\bF_{\geq \ell} = \sum_{ k \geq \ell} F_{d,k}^2 \bQ_k$ in the second line, and we used Cauchy-Scwartz on the second line. Using the identity \eqref{eq:SMW_useful}, we have for any $\delta >0$,
\begin{equation}\label{eq:bound_DYXiXiYD}
\begin{aligned}
    \E_{\bX} \big[\| \bD_{\leq \ell - 1} \bY_{\leq \ell -1 }^\sT \bXi^2 \bY_{\leq \ell -1 } \bD_{\leq \ell - 1} \|_{\op}^2 \big]^{1/2} = &~ \frac{1}{n} \E_{\bX}  \big[\|\bL \bY_{\leq \ell -1 }^\sT \bA^{-2}  \bY_{\leq \ell -1 } \bL / n\|_{\op}^2 \big]^{1/2}\\
    =&~ O_d(n^{-1}) \, ,
    \end{aligned}
\end{equation}
where we denoted $\bL := \big[ (n\bD_{\leq \ell - 1})^{-1} + \bY_{\leq \ell - 1 }^\sT \bA^{-1} \bY_{\leq \ell - 1}/ n \big]^{-1}$ and $\bA := \sum_{k \geq \ell} \mu_{d,k} \bQ_k + \lambda \id_n$, and we used the bound \eqref{eq:bound_LYAAYLop} in Proposition \ref{prop:E_bound_for_var}. From Corollary \ref{cor:bound_exp_op} and Lemma \ref{lem:bound_Ql}, we have $\E_{\bX} \big[\| \bF_{\geq \ell} \|_{\op}^2 \big] = O_d (d^{\delta})$ for any $\delta>0$. Taking $\delta$ sufficiently small combined with Eq.~\eqref{eq:bound_DYXiXiYD} in Eq.~\eqref{eq:prep_T2_1} yields
\begin{equation}\label{eq:T2_1}
\begin{aligned}
\E_{\bX} \big[ \Var_{\tbbeta} (  \boldf_{\geq \ell} \bXi \bV_{<\ell} |\bX ) \big] = o_d(1) \, .
\end{aligned}
\end{equation}

Similarly,
\begin{equation}\label{eq:T2_2}
\begin{aligned}
&~\E_{\bX} [ \Var_{\tbbeta} (  \boldf_{< \ell} \bXi \bV_{\geq \ell} |\bX )] \\
= &~\sum_{k \geq \ell} \frac{F_k^2 \xi_k^2}{B(\cA_d, k)} \bbeta_*^\sT  \bY_{\leq \ell -1}^\sT \bXi \bY_k \bY_k^\sT \bXi \bY_{\leq \ell -1} \bbeta_* \\
= &~ \Big[  \sup_{k \geq \ell} \xi_k^2 \Big] \cdot \| \bbeta_* \|_2^2 \cdot \E_{\bX} \big[ \| \bY_{\leq \ell -1}^\sT \bXi \bF_{\geq \ell}  \bXi\bY_{\leq \ell -1} \|_{\op} \big] \\
=&~ O_{d} (n^{-2}) \cdot  \E_{\bX} \big[\| \bF_{\geq \ell} \|_{\op}^2 \big]^{1/2} \cdot \frac{1}{n } \E_{\bX} \big[\| \bD_{\leq \ell - 1}^{-1} \bL \bY_{\leq \ell -1 }^\sT \bA^{-2}  \bY_{\leq \ell -1 } \bL  \bD_{\leq \ell - 1}^{-1} / n\|_{\op}^2 \big]^{1/2} \\
=&~ O_{d} (n^{-2}) \cdot O_d (d^\delta) \cdot n^{-1} \cdot o_d ( n^2) = o_d (1)\, ,
\end{aligned}
\end{equation}
where we used $ \sup_{k \geq \ell} \xi_k^2 = O_d (d^{-\ell}) = O_d (n^{-2})$ (Lemma \ref{lem:decay_eigenvalues}) and $\| \bD_{\leq \ell -1}^{-1} \|_{\op} = \max_{k \leq \ell-1} \xi_k^{-1} = O_d (d^{\ell -1}) = o_d(n)$.

Finally,
\begin{equation}\label{eq:T2_3}
\begin{aligned}
\E_{\bX} [ \Var_{\tbbeta} (  \boldf_{\geq \ell} \bXi \bV_{\geq \ell} |\bX)] \leq &~\sum_{k \neq l \geq \ell} F_k^2 F_l^2 \xi_l^2 \Tr ( \bQ_k \bXi \bQ_l \bXi) + C\sum_{ k \geq \ell} \xi_k^2 F_k^4 \Tr ( \bQ_k \bXi \bQ_k \bXi) \\
\leq &~ C n \cdot  \Big[  \sup_{k \geq \ell} \xi_k^2 \Big] \cdot \| \bXi \|_\op^2  \sum_{k , l \geq \ell}  F_k^2 F_l^2 \E_{\bX} [ \| \bQ_k \|_{\op}^2 ]^{1/2}   \E_{\bX} [ \| \bQ_l \|_{\op}^2 ]^{1/2}  \\
=&~ n \cdot O_d (n^{-2} ) \cdot O_d (d^\delta) = o_d(1) \, ,
\end{aligned}
\end{equation}
where we used that $\Tr ( \bQ_k \bXi \bQ_l \bXi) \leq n \| \bQ_k \|_{\op} \| \bQ_l \|_{\op} \| \bXi \|_{\op}^2 $, $\| \bXi \|_{\op} \leq \lambda^{-1}$, and Corollary \ref{cor:bound_exp_op} and Lemma \ref{lem:bound_Ql}. 

Combining the bounds Eqs.~\eqref{eq:T2_1}, \eqref{eq:T2_2} and \eqref{eq:T2_3} yields
\begin{equation}\label{eq:T2}
\begin{aligned}
   \E_{\bX} [ \Var_{\tbbeta} ( T_2 | \bX) ] = o_d (1) \, .
\end{aligned}
\end{equation}

\noindent
{\bf Step 3. Bounding $T_3$.}

Again we decompose $T_3 = \beps^\sT \bXi \bV_{<\ell} + \beps^\sT \bXi \bV_{\geq \ell}$ into two parts. We have first 
\begin{equation*}
\begin{aligned}
    \E_{\bX} [ \Var (\beps^\sT \bXi \bV_{<\ell}| \bX) ] = &~ \sigma_\eps^2 \E_{\bX} [ \bbeta_*^\sT \bD_{\leq \ell -1} \bY_{\leq \ell - 1}^\sT \bXi^2 \bY_{\leq \ell - 1}\bD_{\leq \ell -1}\bbeta_* ] \\
    =&~ O_d(1) \cdot n^{-1} \E_{\bX} [ \| \bL \bY_{\leq \ell -1 }^\sT \bA^{-2}  \bY_{\leq \ell -1 } \bL  \|_{\op}] = o_d(1) \, .
\end{aligned}
\end{equation*}
Similarly, 
\begin{equation*}
\begin{aligned}
    \E_{\bX} [ \Var (\beps^\sT \bXi \bV_{\geq\ell}| \bX) ] = &~ \sigma_\eps^2 \sum_{k \geq \ell} \xi_k^2 F_{d,k}^2 \E_{\bX} \big[ \Tr( \bQ_k \bXi^2 ) \big] \\
    =&~ n \cdot  \Big[  \sup_{k \geq \ell} \xi_k^2 \Big] \cdot \| \bXi \|_\op^2  \cdot \E_{\bX} [ \| \bF_{\geq \ell} \|_{\op}] = o_d(1) \, .
\end{aligned}
\end{equation*}
Combining the two above bounds yields 
\begin{equation}\label{eq:T3}
\begin{aligned}
   \E_{\bX} [ \Var_{\tbbeta} ( T_3 | \bX) ] = o_d (1) \, .
\end{aligned}
\end{equation}

\noindent
{\bf Step 4. Bounding $T_4$.}

We decompose $T_4 = 2T_{41} + T_{42} +T_{43}$ with $T_{41} = \boldf_{<\ell}^\sT \bXi \bM \bXi \boldf_{\geq \ell}$ and $T_{42} = \boldf_{\geq \ell}^\sT \bXi \bM \bXi \boldf_{\geq \ell}$, and $T_{43}$ is independent of $\beps,\tbbeta$. We further decompose each of this term into two parts using $\bM = \bM_{\leq \ell -1} + \bM_{\geq \ell}$, where $\bM_{\leq \ell -1} = \E_{\bx} [ h_{<\ell} (\bx) h_{<\ell} (\bx)^\sT] = \bY_{\leq \ell-1} \bD_{\leq \ell - 1}^2 \bY_{\leq \ell - 1}$, and $\bM_{\geq \ell} = \sum_{k \geq \ell} \xi_k^2 B(\cA_d , k) \bQ_k$. We bound the variance of each of these terms separately.

Using the identity \eqref{eq:SMW_useful}, we have 
\[
\begin{aligned}
 \E_{\bX} [ \Var (T_{41}| \bX) ] \leq&~ C \sum_{k,l \geq \ell} F_{d,k}^2 F_{d,l}^2 \E_{\bX} \big[ \Tr ( \bQ_k \bXi \bM_{<\ell} \bXi \bQ_l \bXi \bM_{<\ell} \bXi ) \big] \\
 =&~ 2  \E_{\bX} \big[ \Tr (\bL^2 \bY_{\leq \ell -1 }^\sT \bA^{-1} \bF_{\geq \ell} \bA^{-1} \bY_{\leq \ell -1 } \bL^2  \bY_{\leq \ell -1 }^\sT \bA^{-1} \bF_{\geq \ell}  \bA^{-1} \bY_{\leq \ell -1 }  ) \big]/n^4 \, .
\end{aligned}
\]
We will bound differently whether $k,l <m$ and $k,l \geq m$ (where $m$ is given in Assumption \ref{ass:technical_ass}). We decompose $\bF_{\geq \ell}=\bF_{\ell:m}+\bF_{\geq m}$ with $\bF_{\ell:m} = \sum_{\ell \leq k<m} F_{d,k}^2 \bQ_k$ and $\bF_{\geq m} = \sum_{k \geq m} F_{d,k}^2 \bQ_k$. First, by Cauchy-Schwarz inequality and using Proposition \ref{prop:GMMM_prop3} and Lemma \ref{lem:bound_Ql}, 
\[
\begin{aligned}
&~ \E_{\bX} \big[ \Tr (\bL^2 \bY_{\leq \ell -1 }^\sT \bA^{-1} \bF_{\ell:m} \bA^{-1} \bY_{\leq \ell -1 } \bL^2  \bY_{\leq \ell -1 }^\sT \bA^{-1} \bF_{\ell:m}  \bA^{-1} \bY_{\leq \ell -1 }  ) \big]/n^4 \\
\leq &~ \E_{\bX} [ \| \bF_{\ell:m} \|_{\op}^4]^{1/2} \cdot  \E_{\bX} \big[ \Tr (\bL^2 \bY_{\leq \ell -1 }^\sT \bA^{-2} \bY_{\leq \ell -1 } \bL^2  \bY_{\leq \ell -1 }^\sT \bA^{-2}  \bY_{\leq \ell -1 }  )^2 \big]^{1/2}/n^4 \\
= &~ O_d(d^\delta) \cdot \E [ \| \bL \bY_{\leq \ell -1 }^\sT \bA^{-2}  \bY_{\leq \ell -1 } \bL /n \|_\op^4]^{1/2} /n = o_d (1) \, ,
\end{aligned}
\]
where we used the bound \eqref{eq:bound_LYAAYLop} in Proposition \ref{prop:E_bound_for_var}.
For the high degree part, notice that $\bA^{-1} \preceq ( \mu_{d,k} \bQ_k + \lambda \id)^{-1} $, and therefore, by Assumption \ref{ass:technical_ass}, there exists $\delta>0$ such that 
\[
 \| \bA^{-1/2} \bF_{\geq m} \bA^{-1/2} \|_{\op} = \sum_{k \geq m} \frac{F_{d,k}^2}{\mu_{d,k}} = O_d(n^{1-\delta}) \, .
\]
Applying this bound, we get
\[
\begin{aligned}
&~ \E_{\bX} \big[ \Tr (\bL^2 \bY_{\leq \ell -1 }^\sT \bA^{-1} \bF_{\geq m} \bA^{-1} \bY_{\leq \ell -1 } \bL^2  \bY_{\leq \ell -1 }^\sT \bA^{-1} \bF_{\geq m}  \bA^{-1} \bY_{\leq \ell -1 }  ) \big]/n^4 \\
= &~ O_d(1) \cdot \E_{\bX} \big[ \Tr (\bL^2 \bY_{\leq \ell -1 }^\sT \bA^{-1} \bY_{\leq \ell -1 } \bL^2  \bY_{\leq \ell -1 }^\sT \bA^{-1}  \bY_{\leq \ell -1 }  ) \big]/n^4 \\
= &~ O_d(1) \cdot \E [ \| \bL \|_\op^2] /n = O_d (d^{ \delta'} /n^{\delta}) = o_d(1) \, ,
\end{aligned}
\]
where we used the bound \eqref{eq:bound_Lop} in Proposition \ref{prop:E_bound_for_var}.
The cross-terms can be bounded in a similar manner. We conclude that 
\begin{equation}\label{eq:bound_T41}
     \E_{\bX} [ \Var (T_{41}| \bX) ] = o_d(1) \, .
\end{equation}

The second term $T_{42}$ can be bounded more easily using that $\bM_{\geq \ell} = \big[ \sup_{k \geq \ell} \xi_k \big] \cdot \sum_{k \geq \ell} \mu_{d,k} \bQ_k$ and therefore $\| \bXi^{1/2} \bM_{\geq \ell} \bXi^{1/2} \|_\op \leq 1$:
\begin{equation}\label{eq:bound_T42}
\begin{aligned}
 \E_{\bX} [ \Var (T_{42}| \bX) ] \leq&~ C \E_{\bX} \big[ \Tr ( \bF_{\geq \ell} \bXi \bM_{\geq \ell} \bXi \bF_{\geq \ell} \bXi \bM_{\geq\ell} \bXi ) \big] \\
 =&~ O_d(1) \cdot \big[ \sup_{k \geq \ell} \xi_k \big]^2 \cdot \| \bXi \|_\op^2 n \E_{\bX} \big[ \|  \bF_{\geq \ell} \|_\op^2 \big] = O_d(n^{-2}) \cdot O_d(d^\delta) = o_d(1) \, .
\end{aligned}
\end{equation}
Combining bounds \eqref{eq:bound_T41} and \eqref{eq:bound_T42} yields 
\begin{equation}\label{eq:T4}
\begin{aligned}
   \E_{\bX} [ \Var_{\tbbeta} ( T_4 | \bX) ] = o_d (1) \, .
\end{aligned}
\end{equation}

\noindent
{\bf Step 5. Bounding $T_6$.}

We decompose again $\bM = \bM_{<\ell} + \bM_{\geq \ell}$ and proceed similarly to $T_4$:
\[
\begin{aligned}
\E_{\bX} [ \Var ( \beps^\sT \bXi \bM_{<\ell} \bXi \beps |\bX ) ] \leq &~ C \E_{\bX}  [ \Tr ( \bXi^2 \bY_{\leq \ell -1} \bD_{\leq \ell-1}^2 \bY_{\leq \ell -1}^\sT \bXi^2 \bY_{\leq \ell -1} \bD_{\leq \ell-1}^2 \bY_{\leq \ell -1}^\sT) ] \\
 =&~ C \E_{\bX} [ \Tr ( \bL^2 \bY_{\leq \ell -1}^\sT \bA^{-2}  \bY_{\leq \ell -1}^\sT \bL^2\bA^{-2}  \bY_{\leq \ell -1}^\sT) ]/n^4 \\
 = &~ O_d(1) \cdot \E_{\bX} [ \| \bL \bY_{\leq \ell -1}^\sT \bA^{-2}  \bY_{\leq \ell -1}^\sT \bL/n \|_{\op}^2 ]/n = o_d(1) \, ,
\end{aligned}
\]
and 
\[
\begin{aligned}
\E_{\bX} [ \Var ( \beps^\sT \bXi \bM_{\geq \ell} \bXi \beps |\bX ) ] \leq &~ C\big[ \sup_{k \geq \ell} \xi_k \big]^2 \cdot \E_{\bX} [ \Tr ( \bXi^2 \bA \bXi^2 \bA ) ] \\
 =&~ O_d(n^{-2}) \cdot n \| \bXi \|_\op^2 \cdot \E_{\bX} [ \| \bA \|_{\op}^2 ] = o_d(1) \, . 
\end{aligned}
\]
Combining the two above bounds yields 
\begin{equation}\label{eq:T6}
\begin{aligned}
   \E_{\bX} [ \Var_{\tbbeta} ( T_6 | \bX) ] = o_d (1) \, .
\end{aligned}
\end{equation}

\noindent
{\bf Step 6. Concluding.}

The term $T_5$ can be bounded similarly as $T_4$ and $T_6$. The proposition follows by combining bounds \eqref{eq:T1}, \eqref{eq:T2}, \eqref{eq:T3}, \eqref{eq:T4} and \eqref{eq:T6}.

\subsubsection{Technical results: bounds in expectation}\label{sec:technical_bound_expectation}

In this section, we gather some $L^q$-bounds necessary for the proof of Proposition \ref{prop:risk_bound_variance}  (instead of bounds in probability, as in Section \ref{sec:technical_lemmas}). We first recall some concentration results on matrices of spherical harmonics (or Fourier basis) proved in \cite{ghorbani2021linearized} and \cite{mei21generalization}. 

\begin{lemma}\label{lem:YY_concentration}
% Fix an integer $\ell \in \naturals$. Let $n = \Omega_d ( d^\ell)$ and $(\bx_i)_{i \in [n]} \sim_{iid} \Unif (\cA_d) $. Denote $B_{\ell - 1} = \sum_{k = 0}^{\ell -1 } $
Follow the assumptions of Theorem \ref{thm:KRR}. Recall $n = \Theta_d ( d^\ell)$ and $(\bx_i)_{i \in [n]} \sim_{iid} \Unif (\cA_d) $. We denote $B_{\ell - 1} = \sum_{k = 0}^{\ell -1 } B(\cA_d , k)$ and $\bY_{\leq \ell - 1} = (\bY_0 , \ldots , \bY_{\leq \ell - 1} ) \in \R^{n \times B_{\ell -1}}$, where 
\[
\bY_k = ( Y_{ks} (\bx_i))_{i \in [n], s \in [B(\cA_d , k)]} \in \R^{n \times B(\cA_d, k)}\, .
\]
Then, there exists a constant $C>0$ such that for any $t>0$,
\[
\P \big( \| \bY_{\leq \ell - 1}^\sT \bY_{\leq \ell - 1} / n - \id_{B_{\leq \ell - 1}} \|_{\op} \geq t \big) \leq d  \exp \big\{  - C d t^2/(1+t) \big\} \, .
\]
In the case of the hypercube $\cA_d = \Cube^d$, the same result holds with $\bY_{\leq \ell - 1}$ replaced by $\bY_{\geq d - \ell +1} := [ \bY_{d}, \bY_{d-1} , \ldots , \bY_{d- \ell +1} ] \in \R^{n \times B_{\ell - 1}}$.
\end{lemma}

\begin{proof}[Proof of Lemma \ref{lem:YY_concentration}]
This follows from the proof of {\cite[Lemma 11]{ghorbani2021linearized}} by noting that $n/B_{\leq \ell - 1} = \Omega_d (d)$ (see also \cite[Proposition 3]{mei21generalization} for the general argument). In the case of the hypercube, note that for any $S \subseteq [d]$ with $|S| = d - k$, we have $Y_{S} ( \bx) = \prod_{i \in S} x_i = \big( \prod_{i \in [d]} x_i \big) Y_{S^c} (\bx)$, hence we can rewrite $\bY_{\geq d-\ell +1} = [ \bY_{d}, \bY_{d-1} , \ldots , \bY_{d- \ell +1} ]$ as 
\begin{equation}\label{eq:equiv_low_high_cube}
\bY_{\geq d-\ell +1} = \bS \bY_{\leq \ell -1}\, ,\qquad \bS = \diag (( Y_{[d]} (\bx_i ))_{i \in [n]})\, ,
\end{equation}
where $Y_{[d]} (\bx) = \prod_{i \in [d]} x_i \in \{\pm 1\}$ (hence $\bS^2 = \id_n$). Therefore $\bY_{\geq d-\ell +1}^\sT \bY_{\geq d-\ell +1}  = \bY_{\leq \ell -1}^\sT  \bY_{\leq \ell -1} $.
\end{proof}

The following is a reformulation of Proposition 3 proved in \cite{ghorbani2021linearized} (see also Proposition 4 in \cite{mei21generalization} for a more general proof).

\begin{proposition}[{\cite[Proposition 3]{ghorbani2021linearized}}]\label{prop:GMMM_prop3} 
Follow the assumptions of Theorem \ref{thm:KRR}. Recall $n = \Theta_d ( d^\ell)$ and $(\bx_i)_{i \in [n]} \sim_{iid} \Unif (\cA_d) $. Denote $\bQ_k = ( Q_k^{(d)} ( \< \bx_i , \bx_j \>) )_{ij \in [n]} \in \R^{n \times n}$ the $k$-th Gegenbauer empirical matrix. Fix $k > \ell$ and an integer $q \geq 1$, then 
\[
\E_{\bX} \Big[ \| \bQ_k - \id_n \|_{\op}^q \Big] = o_d (1) \, .
\]
In the case  of the hypercube $\cA_d = \Cube^d$, the same result holds for $\bQ_{d-k}$.

Furthermore, we have
\[
\begin{aligned}
&\text{for } \cA_{d} = \S^{d-1} (\sqrt{d}): & \qquad&\E_{\bX} \Big[ \sup_{k \geq \ell+1} \| \bQ_k - \id_n \|_{\op}^2 \Big] = o_d (1)\, ,\\
&\text{for } \cA_{d} = \Cube^d : & \qquad&\E_{\bX} \Big[ \sup_{d-\ell-1 \geq k \geq \ell+1} \| \bQ_k - \id_n \|_{\op}^2 \Big] = o_d (1)\, .
\end{aligned}
\]
\end{proposition}

Note that for $\cA_d = \Cube^d$, using Eq.~\eqref{eq:equiv_low_high_cube}, we have $\bQ_{d - k} = \bS \bQ_k  \bS $ and therefore the equality $\| \bQ_{d-k} - \id_n \|_{\op} = \| \bQ_{k} - \id_n \|_{\op}$. 

% The following proposition provides a uniform bound on the operator norm of the Gegenbauer matrices:

% \begin{proposition}\label{prop:unif_Gegenbauer}
% Following the same setting as Proposition \ref{prop:GMMM_prop3}. We have
% \[
% \begin{aligned}
% &\text{for } \cA_{d} = \S^{d-1} (\sqrt{d}): & \qquad&\E_{\bX} \Big[ \sup_{k \geq \ell+1} \| \bQ_k - \id_n \|_{\op}^2 \Big] = o_d (1)\, ,\\
% &\text{for } \cA_{d} = \Cube^d : & \qquad&\E_{\bX} \Big[ \sup_{d-\ell-1 \geq k \geq \ell+1} \| \bQ_k - \id_n \|_{\op}^2 \Big] = o_d (1)\, .
% \end{aligned}
% \]
% \end{proposition}

% \begin{proof}[Proof of Proposition \ref{prop:unif_Gegenbauer}]
% \notate{Do}
% \end{proof}
The above proposition consider the Gegenbauer matrices $\bQ_k$ for $k \geq \ell+1$. For $\bQ_{\ell}$, recall that in Theorem \ref{thm:spectrum_Gegenbauer}, we proved that its empirical spectral law converges in distribution to a Marchenko-Pastur law. However we do not bound the tail of its largest eigenvalue. The following lemma provides an easy (but loose) upper bound on the expected moment of the operator norm of $\bQ_{\ell}$. This will be enough for the purpose of our proof.

\begin{lemma}\label{lem:bound_Ql}
Follow the same setting as Proposition \ref{prop:GMMM_prop3}. We have
for any fixed $q \in [d]$ and any constant $\delta >0$,
\[
\E_{\bX} \big[ \| \bQ_{\ell} \|_{\op}^q \big] = O_d (d^\delta) \, .
\]
The same result holds for $\bQ_{d - \ell}$ in the case of the hypercube $\cA_d = \Cube^d$.
\end{lemma}

\begin{proof}[Proof of Lemma \ref{lem:bound_Ql}]
Recall that we can write $\bQ_\ell = \bY_{\ell}\bY_{\ell}^\sT / B(\cA_d, \ell)$. We can therefore apply Proposition \ref{prop:Vershynin_matrix} with $\bY = \bY_\ell$ and $B = B(\cA_d, \ell)$. We get
\[
\E_{\bX} \big[ \| \bQ_\ell \|_\op^q \big]^{1/2} = \frac{1}{\sqrt{B}}\E_{\bX} \big[ \| \bY \|_\op^{2q} \big]^{1/(2q)} \leq C \Big[ \sqrt{ \frac{n}{B}} + \sqrt{\frac{\Gamma \log(\min(n,B))}{B}} \Big]\,.
\]
Let us bound $\Gamma = \E_{\bY} [ \max_{i \leq n} \| \by_i \|_2^{q_*} ]^{2/q_*}$, where $\by_i = \bY_\ell ( \bx_i)$. For any $p \geq 1$, we have
\[
\begin{aligned}
\E_{\bY} \Big[ \max_{i \leq n} \| \by_i \|_2^{q_*} \Big]^{2/q_*} \leq \E_{\bY} \Big[ \max_{i \leq n} \| \by_i \|_2^{pq_*} \Big]^{2/(pq_*)}
\leq&~ \E_{\bY} \Big[ \sum_{i \in [n]} \| \by_i \|_2^{pq_*} \Big]^{2/(pq_*)} \\
=&~ n^{2/(pq_*)}\E_{\by_i} \big[ \| \by_i \|_2^{pq_*} \big]^{2/(pq_*)}\, .
\end{aligned}
\]
Denote $r = pq_*/2 \geq 1$. By Jensen's inequality, 
\[
\begin{aligned}
\E_{\by_i} \big[ \| \by_i \|_2^{2r} \big] =&~ \E_{\bx} \Big[ \Big( \sum_{ s \in [B]} Y_{\ell s} (\bx)^2 \Big)^r \Big] \\
\leq&~ B^{r-1} \E_{\bx} \Big[ \sum_{ s \in [B]} Y_{\ell s} (\bx)^{2r}  \Big] \\
\leq &~ C B^{r-1}\sum_{ s \in [B]} \E_{\bx} \big[ Y_{\ell s} (\bx)^2 \big]^{r} = C B^r\, ,
\end{aligned}
\]
where we used in the last line that $Y_{\ell s}$ are degree-$\ell$ polynomials, and the hypercontractivity property (see Section \ref{app:hypercontractivity}). We deduce that $\Gamma \leq C n^{1/r} B$. We deduce that
\[
\E_{\bX} \big[ \| \bQ_\ell \|_\op^q \big]^{1/2} \leq C + C n^{1/(2r)} \log (n)^{1/2} \, ,
\]
and conclude by taking $r$ sufficiently large (i.e., $p$ sufficiently large).
\end{proof}

It will be useful to state the following bound, which is a direct consequence of the above results.

\begin{corollary}\label{cor:bound_exp_op}
Follow the same setting as Proposition \ref{prop:GMMM_prop3}. Suppose there exists a sequence $(\mu_{d,k})_{d \geq 1, k \geq \ell +1}$ such that $\sum_{k \geq \ell +1} \mu_{d,k} = O_d(1)$. If $\cA_d = \Cube^d$, further assume that there exists $\delta >0$ such that $\sup_{0 \leq k \leq \ell} \mu_{d,d-k}/ B(\Cube^d , d-k) = O_d (d^{-\ell - \delta})$. Then, if we denote $\bH_{\bmu} = \sum_{k \geq \ell +1} \mu_{d,k} \bQ_k $, we have for any fixed integer $q \in [d]$, 
\[
\E_{\bX} \big[ \| \bH_{\bmu} \|_{\op}^2 \big] = O_d(1) \, .
\]
\end{corollary}

\begin{proof}[Proof of Corollary \ref{cor:bound_exp_op}]
From Proposition \ref{prop:GMMM_prop3} and its proof, we know that (replacing $\sum_{k \geq \ell+1}$ by $\sum_{d- \ell - 1 \geq k \geq \ell +1}$ in the case of the hypercube)
\[
\begin{aligned}
\E_{\bX} \Big[ \Big\| \sum_{k \geq \ell +1} \mu_{d,k} \bQ_k  \Big\|_{\op}^2 \Big]^{1/2}  \leq \E_{\bX} \Big[ \Big\| \sum_{k \geq \ell +1} \mu_{d,k} (\bQ_k - \id_n)  \Big\|_{\op}^2 \Big]^{1/2} + \sum_{k \geq \ell +1} \mu_{d,k} = O_d(1) \, .
\end{aligned}
\]
The only terms that remain to be bounded are the $\bQ_{d-k}$ with $k \in \{0, \ldots, \ell\}$ for the hypercube $\cA_d = \Cube^d$. From Lemma \ref{lem:bound_Ql} and assumption on $\mu_{d,d-\ell}$, we have for $\delta '>0$ chosen sufficiently small,
\[
\mu_{d,d-\ell} \E_{\bX} [ \| \bQ_{\ell} \|_\op^2 ]^{1/2} = O_d (d^{-\delta}) \cdot O_d (d^{\delta '}) = o_d(1)\, .
\]
For $k < \ell$,
\begin{equation}\label{eq:Qlarge}
\begin{aligned}
&~\E_{\bX} \Big[ \Big\| \sum_{k \leq \ell -1 } \mu_{d,d-k} \bQ_{d-k}  \Big\|_{\op}^2 \Big]^{1/2} \\
=&~ \Big[ \sup_{k < \ell} \mu_{d,k}/B(\Cube^d,d-k) \Big] \cdot \ell \cdot \E_{\bX}  \big[ \| \bY_{\geq d-\ell +1 } \|_{\op}^4 \big]^{1/2} \\
=&~ O_d (d^{-\ell - \delta}) \cdot n \cdot \Big[ 1 + \E_{\bX}  \big[ \| \bY_{\geq d-\ell +1}^\sT \bY_{\geq d-\ell +1}/ n - \id_{B_{\ell - 1}} \|_{\op}^2 \big]^{1/2} \Big]\, .
\end{aligned}
\end{equation}
Denote the event $\cE = \{ \| \bY_{\geq d-\ell +1}^\sT \bY_{\geq d-\ell +1}/ n - \id_{B_{\ell - 1}} \|_{\op} \geq 1/2\}$. Note that 
\[
\| \bY_{\geq d-\ell +1} \|_{\op} \leq \Tr ( \bY_{\geq d-\ell +1}^\sT \bY_{\geq d-\ell +1})^{1/2} = \sqrt{B_{\leq \ell -1} n }\, .
\]
Hence $\| \bY_{\geq d-\ell +1}^\sT \bY_{\geq d-\ell +1}/ n - \id_{B_{\ell - 1}} \|_{\op} \leq 1 + B_{\leq \ell -1}$. From Lemma \ref{lem:YY_concentration}, we get
\[
\E_{\bX}  \big[ \| \bY_{\geq d-\ell +1}^\sT \bY_{\geq d-\ell +1}/ n - \id_{B_{\ell - 1}} \|_{\op}^2 \big] \leq \frac{1}{4 } +  (1 +  B_{\leq \ell -1})^2 \P ( \cE ) \leq C + C B^2 d \exp (-c d) = O_d(1) \, ,
\]
which combined with Eq.~\eqref{eq:Qlarge} concludes the proof.
\end{proof}

In the proof of Proposition \ref{prop:risk_bound_variance}, we will use repeatedly the following identity. Recall that $\bXi = ( \bY_{\leq \ell - 1} \bD_{\leq \ell -1} \bY_{\leq \ell -1}^\sT + \bA )^{-1}$ where $\bA = \sum_{k \geq \ell} \mu_{d,k} \bQ_k + \lambda \id_n$. By Sherman-Morrison-Woodbury formula, we have 
\begin{equation}\label{eq:SMW_useful}
    \bD_{\leq \ell -1 } \bY_{\leq \ell - 1}^\sT \bXi = \big[ (n\bD_{\leq \ell - 1})^{-1} + \bY_{\leq \ell - 1 }^\sT \bA^{-1} \bY_{\leq \ell - 1}/ n \big]^{-1} \bY_{\leq \ell -1}^\sT \bA^{-1} /n \, .
\end{equation}
We will denote 
\[
\bL = \big[ (n\bD_{\leq \ell - 1})^{-1} + \bY_{\leq \ell - 1 }^\sT \bA^{-1} \bY_{\leq \ell - 1}/ n \big]^{-1}\, . 
\]
In the proof of Proposition \ref{prop:risk_bound_variance}, we will use the following bounds on matrix $\bL$:

\begin{proposition}\label{prop:E_bound_for_var}
Assume the same setting as Proposition \ref{prop:risk_bound_variance}. For any fixed $\delta>0$, we have 
\begin{align}\label{eq:bound_Lop}
    \E_{\bX} \big[ \|\bL \|_{\op}^2 \big] = &~ O_d(d^\delta) \, .
\end{align}
Furthermore, for any fixed integer $q \geq 1$, we have
\begin{align}\label{eq:bound_LYAAYLop}
    \E_{\bX} \big[ \|\bL \bY_{\leq \ell - 1}^\sT \bA^{-2} \bY_{\leq \ell - 1} \bL /n \|_{\op}^q \big] =&~ O_d (1) \, .
\end{align}
\end{proposition}

\begin{proof}[Proof of Proposition \ref{prop:E_bound_for_var}]
Denote the event $\cE = \{ \| \bY_{\geq d-\ell +1}^\sT \bY_{\geq d-\ell +1}/ n - \id_{B_{\ell - 1}} \|_{\op} \geq 1/2\}$. On $\cE^c$, we can use the matrix identity in Lemma \ref{lem:matrix_identity} and obtain:
\[
\begin{aligned}
\| \bL \|_{\op} \leq &~  \big\| (\bY_{\leq \ell - 1 }^\sT \bA^{-1} \bY_{\leq \ell - 1}/ n)^{-1} \big\|_\op \\
= &~n  \big\| \bY_{\leq \ell - 1 }^\dagger \bA^{1/2} \bPi \bA^{1/2} ( \bY_{\leq \ell - 1 }^\dagger )^\sT \big\|_{\op} \\ 
\leq &~ 2  \| \bA \|_{\op}\,. 
\end{aligned}
\]
On the event $\cE$, we use that $\| \bL \|_{\op} \leq n \| \bD_{\leq \ell -1 } \|_{\op} = O_d (d^\ell)$. Hence, using Lemma \ref{lem:YY_concentration}, we obtain
\[
\E [ \| \bL \|_{\op}^2 ] \leq C \E [ \| \bA \|_{\op}^2 ] + C d^\ell \P(\cE) \leq O_d (d^\delta) +C d^{\ell+1} \exp(-cd) = O_d (d^\delta) \, , 
\]
where we used Corollary \ref{cor:bound_exp_op} and Lemma \ref{lem:bound_Ql}.

For the second bound, we use again the matrix identity of Lemma \ref{lem:matrix_identity} on $\cE^c$:
\[
\begin{aligned}
&~ \|\bL \bY_{\leq \ell - 1}^\sT \bA^{-2} \bY_{\leq \ell - 1} \bL /n \|_{\op} \\
=&~ n \|\bY_{\leq \ell - 1 }^\dagger \bA^{1/2} \bPi \bA^{1/2} ( \bY_{\leq \ell - 1 }^\dagger )^\sT \bY_{\leq \ell - 1}^\sT \bA^{-2} \bY_{\leq \ell - 1} \bY_{\leq \ell - 1 }^\dagger \bA^{1/2} \bPi \bA^{1/2} ( \bY_{\leq \ell - 1 }^\dagger )^\sT  \|_{\op} \\
\leq&~ 2 \|\bA^{1/2} \bPi \bA^{1/2} \bA^{-2} \bA^{1/2} \bPi \bA^{1/2}  \|_{\op}
\leq 2 \, .
\end{aligned}
\]
On the event $\cE$, we use that 
\[
\begin{aligned}
 \|\bL \bY_{\leq \ell - 1}^\sT \bA^{-2} \bY_{\leq \ell - 1} \bL /n \|_{\op}
=&~ \| \bA^{-1}  \|_{\op} \|\bL \bY_{\leq \ell - 1}^\sT \bA^{-1} \bY_{\leq \ell - 1} \bL/n  \|_{\op} \\
\leq&~ \lambda^{-1}  \|\bL   \|_{\op} \leq \lambda^{-1} n \| \bD_{\leq \ell -1} \|_\op = O_d (d^{\ell}) \, ,
\end{aligned}
\]
where we used $\| \bA^{-1} \|_{\op} \leq \lambda$ and $\| \bL^{1/2} \bY_{\leq \ell - 1 }^\sT \bA^{-1} \bY_{\leq \ell - 1} \bL^{1/2}/ n  \|_{\op} \leq 1$.
Hence, using Lemma \ref{lem:YY_concentration}, we obtain
\[
\E [ \| \bL \bY_{\leq \ell - 1}^\sT \bA^{-2} \bY_{\leq \ell - 1} \bL /n \|_{\op}^q ] \leq C + C d^{q\ell+1} \exp(-cd) = O_d (1) \, , 
\]
which concludes the proof of this proposition.
\end{proof}

\subsection{Proof of the asymptotic formula of the training error and RKHS norm}\label{sec:training_error}

The proof is very similar to the proof for the prediction error and we will simply outline the main steps. First recall that the training error is given by
\[
R_{\text{train}} ( f_* ; \bX , \beps , \lambda ) = \frac{1}{n} \| \by - \hat \by ( \hba_\lambda ) \|_2^2 = \frac{1}{n} \| \by - \bH (\bH + \lambda \id_n )^{-1} \by \|_2^2 = \frac{\lambda^2}{n} \| (\bH + \lambda \id_n )^{-1} \by \|_2^2 \, .
\]
The RKHS norm of the KRR solution $\hf ( \cdot ; \hba_\lambda )$ is given by 
\[
\| \hf ( \cdot ; \hba_\lambda ) \|_{\cH}^2 = \hba_\lambda^\sT \bH \hba_\lambda = \by^\sT (\bH + \lambda \id_n )^{-1} \bH (\bH + \lambda \id_n )^{-1} \by \, .
\]

\noindent
{\bf Step 1. Computing the asymptotics formula of $\E_{\beps, f_*} [R_{\text{train}}] $ and $\E_{\beps, f_*} [\| \hf ( \cdot ; \hba_\lambda ) \|_{\cH}^2 ] $.}

We decompose
\[
\E_{\beps, f_*} [R_{\text{train}} ( f_* ; \bX , \beps , \lambda ) ] = B_{11} + B_{12} + V_1 \, , 
\]
where
\[
\begin{aligned}
B_{11} = &~ \frac{\lambda^2}{n} \bbeta_*^\sT \bY_{\leq \ell - 1}^\sT \bXi^2 \bY_{\leq \ell - 1}\bbeta_* \, ,\\
B_{12} = &~ \frac{\lambda^2}{n}  \sum_{k \geq \ell } F_{d,k}^2 \Tr ( \bXi^2 \bQ_k ) \, , \\
V_1 = &~ \frac{\lambda^2 \sigma_\eps^2}{n}  \Tr ( \bXi^2 ) \, .
\end{aligned}
\]
From Eq.~\eqref{eq:SMW_useful}, we have
\[
\begin{aligned}
B_{11} \leq&~  \frac{\lambda^2}{n} \| \bbeta_* \|_2^2 \|  \bY_{\leq \ell - 1}^\sT \bXi^2 \bY_{\leq \ell - 1}\|_{\op} \\
=&~ O_{d,\P}(n^{-2}) \cdot \| \bD_{\leq \ell -1}^{-1} \bL \bY_{\leq \ell - 1}^\sT\bA^{-2} \bY_{\leq \ell - 1} \bL \bD_{\leq \ell -1}^{-1} \|_{\op} = o_{d,\P}(1) \, .
\end{aligned}
\]
For the two last terms:
\[
\begin{aligned}
B_{12} =&~ \lambda^2 F_{\ell}^2 \cdot \frac{1}{n} \Tr ( \bXi^2 \bQ_\ell ) + \lambda^2 F_{>\ell}^2 \cdot \frac{1}{n} \Tr ( \bXi^2 ) +o_{d,\P}(1) \\
=&~ \lambda^2 \frac{F_{\ell}^2}{\mu_\ell^2} \big[  r_\psi ( - \zeta_* ) - \zeta_* r_\psi ' ( - \zeta_* ) \big] + \lambda^2 \frac{F_{>\ell}^2}{\mu_\ell^2} r_\psi ' ( - \zeta_* ) +o_{d,\P}(1)  \, , \\
V_1 = &~ \lambda^2 \frac{\sigma_\eps^2}{\mu_\ell^2} r_\psi ' ( - \zeta_* ) +o_{d,\P}(1) \, .
\end{aligned}
\]

Similarly, for the RKHS norm, we have the decomposition:
\[
\begin{aligned}
\frac{1}{n}\E_{\beps, f_*} [\| \hf ( \cdot ; \hba_\lambda ) \|_{\cH}^2 ] = B_{21} + B_{22} + V_2 \, , 
\end{aligned}
\]
where
\[
\begin{aligned}
B_{21} = &~ \frac{1}{n} \bbeta_*^\sT \bY_{\leq \ell - 1}^\sT \bXi \bH \bXi \bY_{\leq \ell - 1}\bbeta_* \, ,\\
B_{22} = &~ \frac{1}{n}  \sum_{k \geq \ell } F_{d,k}^2 \Tr ( \bXi \bH \bXi \bQ_k ) \, , \\
V_2 = &~ \frac{ \sigma_\eps^2}{n}  \Tr ( \bXi^2 \bH ) \, .
\end{aligned}
\]
For the first term, we have by Sherman-Morrison-Woodbury formula,
\[
\begin{aligned}
B_{21} = &~ \frac{1}{n} \bbeta_*^\sT \bY_{\leq \ell - 1}^\sT \bXi \bY_{\leq \ell - 1}\bbeta_* - B_{11}/\lambda \\
=&~ \bbeta_*^\sT (n \bD_{\leq \ell -1} )^{-1}  \big[ (n \bD_{\leq \ell -1} )^{-1} +\bY_{\leq \ell - 1}^\sT \bA^{-1} \bY_{\leq \ell - 1} /n \big]^{-1}  ( \bY_{\leq \ell - 1}^\sT \bA^{-1} \bY_{\leq \ell - 1} /n ) \bbeta_* +o_{d,\P}(1)\\
=&~ \bbeta_*^\sT (n \bD_{\leq \ell -1} )^{-1} \bbeta_* + o_{d,\P} (1) = o_{d,\P}(1) \, ,
\end{aligned}
\]
where we used that $\| (n \bD_{\leq \ell -1} )^{-1} \|_{\op} = o_{d} (1)$ and $\lambda_{\min} ( \bY_{\leq \ell - 1}^\sT \bA^{-1} \bY_{\leq \ell - 1} /n  ) = \Omega_{d,\P} (1) $. Similarly to the training error,
\[
\begin{aligned}
B_{22} = &~  F_{\ell}^2 \cdot \frac{1}{n} \Tr ( \bXi \bH \bXi \bQ_\ell ) + F_{>\ell}^2 \cdot \frac{1}{n} \Tr ( \bXi^2 \bH ) +o_{d,\P}(1) \\
=&~  F_{\ell}^2 \cdot \frac{1}{n}  \big[  \Tr ( \bXi  \bQ_\ell )  - \lambda  \Tr (  \bXi^2 \bQ_\ell )\big] +   F_{>\ell}^2 \cdot \frac{1}{n} \big[ \Tr ( \bXi ) - \lambda \Tr ( \bXi^2  ) \big] + o_{d,\P}(1) \\
=&~  \frac{F_{\ell}^2}{\mu_\ell} \Big[ 1 - (\zeta_* + \lambda/\mu_\ell)  r_\psi ( - \zeta_* )  + \frac{\lambda }{\mu_\ell} \zeta_* r_\psi ' (- \zeta_*) \Big] +  \frac{F_{>\ell}^2}{\mu_\ell} \Big[ r_\psi (- \zeta_*) - \frac{\lambda}{\mu_\ell} r_\psi ' ( - \zeta_* ) \Big] +o_{d,\P}(1)  \, ,
\end{aligned}
\]
and
\[
\begin{aligned}
V_1 = &~  \frac{\sigma_\eps^2}{\mu_\ell} \Big[ r_\psi (- \zeta_*) - \frac{\lambda}{\mu_\ell} r_\psi ' ( - \zeta_* ) \Big] +o_{d,\P}(1)\, .
\end{aligned}
\]

\noindent
{\bf Step 2. Bounding $\E_{\bX} [ \Var_{\beps, f_*}( R_{\text{train}} | \bX)] $ and $\E_{\bX} [ \Var_{\beps, f_*}(\| \hf ( \cdot ; \hba_\lambda ) \|_{\cH}^2 |\bX ) ] $.}

We see that it is sufficient to show that
\[
\begin{aligned}
\frac{1}{n^2} \E_{\bX} [ \Var_{\beps, f_*}( \by^\sT \bXi \by | \bX)] = o_d(1) \, .
\end{aligned}
\]
Similarly, we decompose $\by^\sT \bXi \by$ as in the proof of Proposition \ref{prop:risk_bound_variance} and bound each term separately:
\[
\begin{aligned}
\frac{1}{n^2} \E_{\bX} [ \Var_{\beps, f_*}( \boldf_{\leq \ell -1}^\sT \bXi \beps | \bX)] =&~ O_d(1) \cdot \frac{1}{n^2}\| \bbeta_* \|_2^2 \E_{\bX} \big[ \| \bY_{\leq \ell - 1}^\sT \bXi^2  \bY_{\leq \ell - 1} \|_\op \big] = o_{d} (1) \, ,\\
\frac{1}{n^2} \E_{\bX} [ \Var_{\beps, f_*}( \boldf_{\leq \ell -1}^\sT \bXi \boldf_{>\ell} | \bX)] =&~ O_d(1) \cdot  \frac{1}{n^2} \| \bbeta_* \|_2^2 \E_{\bX} \big[ \| \bF_{\geq \ell } \|^2_\op \big]^{1/2} \E_{\bX} \big[ \| \| \bY_{\leq \ell - 1}^\sT \bXi^2  \bY_{\leq \ell - 1}  \|^2_\op \big]^{1/2}  \\
=&~ O_d(d^\delta) \cdot \| (n \bD_{\leq \ell - 1})^{-1} \|_\op^2 = o_{d}(1)\, , \\
\frac{1}{n^2} \E_{\bX} [ \Var_{\beps, f_*}( \beps^\sT \bXi \beps | \bX)] =&~ O_d(1) \cdot \frac{1}{n^2} \Tr (\bXi^2 ) = o_d(1) \, , \\
\frac{1}{n^2} \E_{\bX} [ \Var_{\beps, f_*}( \boldf_{> \ell }^\sT \bXi \boldf_{> \ell } | \bX)] =&~ O_d(1) \cdot \frac{1}{n} \E_{\bX} [ \| \bF_{\geq \ell } \|_{\op}^2 ] = o_d(1) \, .
\end{aligned}
\]

Similarly to the proof in Section \ref{sec:outline}, combining the convergence in probability of the expectation in step 1, and the bound on the variance in step 2 yields the results for the training error and RKHS norm of Theorem \ref{thm:KRR}.

\subsection{Auxiliary lemmas}

\begin{lemma}\label{lem:decay_eigenvalues}
Let $h_d: \R \to \R$ be a sequence of inner-product kernel such that $h_d(1) =O_d(1)$. Let $(\xi_{d,k}(h_d))_{k\geq 0}$ be its Gegenbauer coefficients in $\cA_d$. Recall that $\xi_{d,k} (h) \geq 0$ are given by
\[
\xi_{d,k} (h_d) = \E_{\bx \sim \Unif (\cA_d)} [ h_d (\< \ones , \bx \>/d) Q_k^{(d)} ( \< \ones , \bx \> )  ] \, .
\]
If $\cA_d= \Cube^d$, assume that $\max_{0\leq k \leq \ell} \xi_{d,d-k} (h_d) = O_d (d^{-\ell - 1})$. Then we have
\begin{align}
    \sup_{k \geq \ell +1} \xi_{d,k} (h_d) = O_d ( d^{-\ell -1 }) \\, .
\end{align}
\end{lemma}

\begin{proof}[Proof of Lemma \ref{lem:decay_eigenvalues}]
From the assumption, there exists a constant $C>0$ such that
\[
h_d(1) = \sum_{k \geq 0} \xi_{d,k} (h_d) B (\cA_d , k) \leq C\, , 
\]
and therefore $\xi_{d,k} (h_d) \leq C/ B(\cA_d, k)$. The lemma follows from $\sup_{\ell+1 \leq k \leq d- \ell - 1} B (\Cube^d , k)^{-1} = {{d}\choose{\ell +1}}^{-1} = O_d (d^{-\ell -1})$ and $\sup_{k \geq \ell+1 } B (\S^d , k)^{-1} = B (\S^d , \ell +1)^{-1} = O_d (d^{-\ell -1})$, where we used that $B (\S^d , k)$ is non-decreasing (see for example Lemma 1 in \cite{ghorbani2021linearized}).
\end{proof}

\vspace{+10pt}

The following proposition is a simple modification of the proof of Theorem 5.48 in \cite{vershynin2010introduction}:

\begin{proposition}\label{prop:Vershynin_matrix}
Let $\bY \in \R^{n \times B}$ be a matrix whose rows $\by_i$ are independent random vectors in $\R^B$ with common second moment matrix $\bSigma = \E[\by_i \by_i^\sT]$. Fix an integer $q \geq 1$ and let $q_* = 2q/(2q - 1)$ and $\Gamma := \E_{\bY} [ \max_{i \leq n} \| \by_i \|_2^{q_*} ]^{2/q_*} $. Then, there exists a constant $C_q >0$ such that 
\begin{equation}\label{eq:Lq_bound_op}
\E \Big[ \Big\| \frac{1}{n} \bY^\sT \bY - \bSigma \Big\|_{\op}^q \Big]^{1/q} \leq \max ( \| \bSigma\|^{1/2}_\op \delta, \delta^2 )\quad \text{ where }\quad \delta = C_q \sqrt{\frac{\Gamma \log ( \min (n,B))}{n}}\, ,
\end{equation}
and
\begin{equation}\label{eq:Lq_bound_op_2}
\E[ \| \bY \|_{\op}^{2q} ]^{1/(2q)} \leq C_q \big[ \| \bSigma \|_\op^{1/2} \sqrt{n} + \sqrt{\Gamma \log ( \min (n,B))} \big] \, .
\end{equation}
\end{proposition}

\begin{proof}[Proof of Proposition \ref{prop:Vershynin_matrix}]
The proof follows from the same argument as in the proof of Theorem 5.45 in \cite{vershynin2010introduction}. We will denote $C>0$ a generic constant that only depends on $q$. By symmetrization, we have
\[
\E_{\bY} \Big[ \Big\| \frac{1}{n} \bY^\sT \bY - \bSigma \Big\|_{\op}^q \Big]^{1/q} \leq C \E_{\bY,\beps} \Big[ \Big\| \frac{1}{n} \sum_{i = 1} \eps_i \by_i \by_i^\sT   \Big\|_{\op}^q \Big]^{1/q} \, ,
\]
where the $\eps_i$ are $n$ independent Rademacher random variables. 

By noncommutative Khinchine's inequality, we have
\[
\begin{aligned}
\E_{\bY,\beps} \Big[ \Big\| \frac{1}{n} \sum_{i = 1} \eps_i \by_i \by_i^\sT   \Big\|_{\op}^q \Big]^{1/q} \leq&~ C \frac{\sqrt{\log(\min (n,B))}}{n} \E_{\bY} \Big[ \Big\| \sum_{i \in [n]} \| \by_i \|_2^2 \by_i \by_i^\sT \Big\|_{\op}^{1/2} \Big] \\
\leq &~ C \frac{\sqrt{\log(\min (n,B))}}{n} \E_{\bY}  \Big[ \max_{i \in [n} \| \by_i \|_2 \cdot \Big\| \sum_{i\in [n]} \by_i \by_i \Big\|^{1/2} \Big] \\
\leq &~ C \frac{\sqrt{\log(\min (n,B))}}{n} \E_{\bY}  \Big[ \max_{i \in [n} \| \by_i \|_2 \cdot \Big\| \sum_{i\in [n]} \by_i \by_i \Big\|^{1/2} \Big] \\
\leq&~ C \frac{\sqrt{\log(\min (n,B))}}{n} \E_{\bY}  \Big[ \max_{i \in [n} \| \by_i \|_2^{q_*} \Big]^{1/q_*} \cdot  \E_{\bY}  \Big[ \Big\| \sum_{i\in [n]} \by_i \by_i \Big\|^{q} \Big]^{1/(2q)} \, . 
\end{aligned}
\]
Denoting $V = \E_{\bY,\beps} \Big[ \Big\| \frac{1}{n} \bY^\sT \bY - \bSigma   \Big\|_{\op}^q \Big]^{1/q}$ and using $\E_{\bY}  \Big[ \Big\| \frac{1}{n}\sum_{i\in [n]} \by_i \by_i \Big\|^{q} \Big]^{1/q} \leq C (V+ \| \bSigma \|_{\op} )$, we get
\[
V \leq C \sqrt{\frac{\Gamma \log(\min (n,B))}{n}} (V +  \| \bSigma \|_{\op})^{1/2}\, .
\]
Denoting $\delta = 4 C \sqrt{\frac{\Gamma \log(\min (n,B))}{n}}$, this implies Eq.~\eqref{eq:Lq_bound_op}. Equation~\eqref{eq:Lq_bound_op_2} is a simple consequence of bound \eqref{eq:Lq_bound_op}.
\end{proof}

Finally, the following lemma provides a useful matrix algebra identity:

\begin{lemma}\label{lem:matrix_identity}
Let $\bU \in \R^{n \times B}$ and $\bV \in \R^{n \times n}$ be two matrices such that $B <n$, $\bU$ is full rank and $\bV$ is symmetric positive definite. We have the following identity:
\[
( \bU^\sT \bV \bU )^{-1} = \bU^\dagger \bV^{-1/2} \bPi \bV^{-1/2} (\bU^\dagger)^\sT \, , 
\]
where $\bPi = \id_n - [\bV^{-1/2} \bP]  [\bV^{-1/2} \bP]^\dagger$, $\bP =  \id_n - (\bU^\dagger)^\sT \bU^\sT  $ and $\bU^\dagger $ is the Moore-Penrose inverse of $\bU$ (here by assumption, $\bU^\dagger = ( \bU^\sT \bU)^{-1} \bU^\sT $).
\end{lemma}

\begin{proof}[Proof of Lemma \ref{lem:matrix_identity}]
Denote for convenience $\bS = \bV^{-1/2}$. This identity comes from the observation that 
\[
( \bU^\sT \bV \bU )^{-1}= \lim_{\lambda \to 0^+} \big[ \bU^\sT (\bS \bS + \lambda \id_n)^{-1} \bU \big]^{-1} \, ,
\]
and by repeatedly applying Sherman-Morrison-Woodbury identity. First,
\[
\begin{aligned}
 \bU^\sT (\bS \bS + \lambda \id_n)^{-1} \bU =&~ \lambda^{-1} \bU^\sT \bU - \lambda^{-2} \bU^\sT \bS ( \id_n + \lambda^{-1} \bS\bS)^{-1} \bS \bU \, .
\end{aligned}
\]
We can then use the identity again on its inverse:
\[
\begin{aligned}
&~ \big[ \bU^\sT (\bS \bS + \lambda \id_n)^{-1} \bU \big]^{-1} \\
=&~ \lambda (\bU^\sT \bU)^{-1}  - (\bU^\sT \bU)^{-1}  \bU^\sT \bS \big[ - ( \id_n + \lambda^{-1} \bS \bS ) + \lambda^{-1} \bS \bU (\bU^\sT \bU)^{-1}  \bU^\sT \bS \big]^{-1} \bS \bU (\bU^\sT \bU)^{-1} \\
=&~ \lambda (\bU^\sT \bU)^{-1}  + (\bU^\sT \bU)^{-1}  \bU^\sT \bS \big[  \id_n + \lambda^{-1} \bS (\id -  \bU (\bU^\sT \bU)^{-1}  \bU^\sT) \bS \big]^{-1} \bS \bU (\bU^\sT \bU)^{-1} \, . 
\end{aligned}
\]
By the definition of the pseudo-inverse, $\bP = \id -  \bU (\bU^\sT \bU)^{-1}  \bU^\sT = \bP^2$. We apply a third time the SMW formula:
\[
\begin{aligned}
&~ \big[ \bU^\sT (\bS \bS + \lambda \id_n)^{-1} \bU \big]^{-1} \\
=&~ \lambda (\bU^\sT \bU)^{-1}  + \bU^\dagger \bS \big[  \id_n + \lambda^{-1} \bS \bP \bP  \bS \big]^{-1} \bS (\bU^\dagger)^\sT \\
=&~ \lambda (\bU^\sT \bU)^{-1}  + \bU^\dagger \bS \big[ \id_n - \bS \bP ( \lambda \id_n +\bP \bS  \bS \bP)^{-1}  \bP \bS \big] \bS (\bU^\dagger)^\sT \\
\xrightarrow[\lambda \to 0^+]{} &~\bU^\dagger \bS \bPi \bS (\bU^\dagger)^\sT\, ,
\end{aligned}
\]
which concludes the proof.
\end{proof}

\clearpage

\section{Technical background}\label{sec:technical_background}

\subsection{Functions on the sphere}\label{sec:functions_sphere}

In this section we introduce some notation and technical background on functional spaces on the sphere.
In particular, we review the decompositions in (hyper-)spherical harmonics on the  $\S^{d-1}(\sqrt{d})$ and in orthogonal polynomials
on the real line. 
We refer the reader to \cite{costas2014spherical,szego1939orthogonal,chihara2011introduction,ghorbani2021linearized}
for further information on these topics.

\subsubsection{Functional spaces over the sphere}

For $d \ge 3$, we let $\S^{d-1}(r) = \{\bx \in \R^{d}: \| \bx \|_2 = r\}$ denote the sphere with radius $r$ in $\reals^d$.
We will mostly work with the sphere of radius $\sqrt d$, $\S^{d-1}(\sqrt{d})$ and will denote by $\tau_{d}$  the uniform probability measure on $\S^{d-1}(\sqrt d)$. 
All functions in this section are assumed to be elements of $ L^2(\S^{d-1}(\sqrt d) ,\tau_d)$, with scalar product and norm denoted as $\<\,\cdot\,,\,\cdot\,\>_{L^2}$
and $\|\,\cdot\,\|_{L^2}$:
\begin{align}
\<f,g\>_{L^2} \equiv \int_{\S^{d-1}(\sqrt d)} f(\bx) \, g(\bx)\, \tau_d(\de \bx)\,.
\end{align}

For $\ell\in\integers_{\ge 0}$, let $\tilde{V}_{d,\ell}$ be the space of homogeneous harmonic polynomials of degree $\ell$ on $\reals^d$ (i.e. homogeneous
polynomials $q(\bx)$ satisfying $\Delta q(\bx) = 0$), and denote by $V_{d,\ell}$ the linear space of functions obtained by restricting the polynomials in $\tilde{V}_{d,\ell}$
to $\S^{d-1}(\sqrt d)$. With these definitions, we have the following orthogonal decomposition
\begin{align}
L^2(\S^{d-1}(\sqrt d) ,\tau_d) = \bigoplus_{\ell=0}^{\infty} V_{d,\ell}\, . \label{eq:SpinDecomposition}
\end{align}
The dimension of each subspace is given by
\begin{align}
\dim(V_{d,\ell}) = B(\S^{d-1}; \ell) = \frac{2 \ell + d - 2}{d - 2} { \ell + d - 3 \choose \ell} \, .
\end{align}
For each $\ell\in \integers_{\ge 0}$, the spherical harmonics $\{ Y_{\ell, j}^{(d)}\}_{1\le j \le B(\S^{d-1}; \ell)}$ form an orthonormal basis of $V_{d,\ell}$:
\[
\<Y^{(d)}_{ki}, Y^{(d)}_{sj}\>_{L^2} = \delta_{ij} \delta_{ks}.
\]
Note that our convention is different from the more standard one, that defines the spherical harmonics as functions on $\S^{d-1}(1)$.
It is immediate to pass from one convention to the other by a simple scaling. We will drop the superscript $d$ and write $Y_{\ell, j} = Y_{\ell, j}^{(d)}$ whenever clear from the context.

We denote by $\oproj_k$  the orthogonal projections to $V_{d,k}$ in $L^2(\S^{d-1}(\sqrt d),\tau_d)$. This can be written in terms of spherical harmonics as
\begin{align}
\oproj_k f(\bx) \equiv& \sum_{l=1}^{B(\S^{d-1}; k)} \< f, Y_{kl}\>_{L^2} Y_{kl}(\bx). 
\end{align}
We also define
$\oproj_{\le \ell}\equiv \sum_{k =0}^\ell \oproj_k$, $\oproj_{>\ell} \equiv \id -\oproj_{\le \ell} = \sum_{k =\ell+1}^\infty \oproj_k$,
and $\oproj_{<\ell}\equiv \oproj_{\le \ell-1}$, $\oproj_{\ge \ell}\equiv \oproj_{>\ell-1}$.

\subsubsection{Gegenbauer polynomials}
\label{sec:Gegenbauer}

The $\ell$-th Gegenbauer polynomial $Q_\ell^{(d)}$ is a polynomial of degree $\ell$. Consistently
with our convention for spherical harmonics, we view $Q_\ell^{(d)}$ as a function $Q_{\ell}^{(d)}: [-d,d]\to \reals$. The set $\{ Q_\ell^{(d)}\}_{\ell\ge 0}$
forms an orthogonal basis on $L^2([-d,d],\tilde\tau^1_{d})$, where $\tilde\tau^1_{d}$ is the distribution of $\sqrt{d}\<\bx,\be_1\>$ when $\bx\sim \tau_d$,
satisfying the normalization condition:
\begin{align}
\< Q^{(d)}_k(\sqrt{d}\< \be_1, \cdot\>), Q^{(d)}_j(\sqrt{d}\< \be_1, \cdot\>) \>_{L^2(\S^{d-1}(\sqrt d))} = \frac{1}{B(\S^{d-1};k)}\, \delta_{jk} \, .  \label{eq:GegenbauerNormalization}
\end{align}
In particular, these polynomials are normalized so that  $Q_\ell^{(d)}(d) = 1$. 
As above, we will omit the superscript $(d)$ in $Q_\ell^{(d)}$ when clear from the context.

Gegenbauer polynomials are directly related to spherical harmonics as follows. Fix $\bv\in\S^{d-1}(\sqrt{d})$ and 
consider the subspace of  $V_{\ell}$ formed by all functions that are invariant under rotations in $\reals^d$ that keep $\bv$ unchanged.
It is not hard to see that this subspace has dimension one, and coincides with the span of the function $Q_{\ell}^{(d)}(\<\bv,\,\cdot\,\>)$.

We will use the following properties of Gegenbauer polynomials
\begin{enumerate}
\item For $\bx, \by \in \S^{d-1}(\sqrt d)$
\begin{align}
\< Q_j^{(d)}(\< \bx, \cdot\>), Q_k^{(d)}(\< \by, \cdot\>) \>_{L^2} = \frac{1}{B(\S^{d-1}; k)}\delta_{jk}  Q_k^{(d)}(\< \bx, \by\>).  \label{eq:ProductGegenbauer}
\end{align}
\item For $\bx, \by \in \S^{d-1}(\sqrt d)$
\begin{align}
Q_k^{(d)}(\< \bx, \by\> ) = \frac{1}{B(\S^{d-1}; k)} \sum_{i =1}^{ B(\S^{d-1}; k)} Y_{ki}^{(d)}(\bx) Y_{ki}^{(d)}(\by). \label{eq:GegenbauerHarmonics}
\end{align}
\end{enumerate}
These properties imply that ---up to a constant--- $Q_k^{(d)}(\< \bx, \by\> )$ is a representation of the projector onto 
the subspace of degree -$k$ spherical harmonics
\begin{align}
(\oproj_k f)(\bx) = B(\S^{d-1}; k) \int_{\S^{d-1}(\sqrt{d})} \, Q_k^{(d)}(\< \bx, \by\> )\,  f(\by)\, \tau_d(\de\by)\, .\label{eq:ProjectorGegenbauer}
\end{align}
For a function $\barsigma \in L^2([-\sqrt d, \sqrt d], \tau^1_{d})$ (where $\tau^1_{d}$ is the distribution of $\< \be_!, \bx \> $ when $\bx \sim_{iid} \Unif(\S^{d-1}(\sqrt d))$), denoting its spherical harmonics coefficients $\xi_{d, k}(\barsigma)$ to be 
\begin{align}\label{eqn:technical_lambda_sigma}
\xi_{d, k}(\barsigma) = \int_{[-\sqrt d , \sqrt d]} \barsigma(x) Q_k^{(d)}(\sqrt d x) \tau^1_{d}(\de x),
\end{align}
then we have the following equation holds in $L^2([-\sqrt d, \sqrt d],\tau^1_{d-1})$ sense
\[
\barsigma(x) = \sum_{k = 0}^\infty \xi_{d, k}(\barsigma) B(\S^{d-1}; k) Q_k^{(d)}(\sqrt d x). 
\]

To  any rotationally invariant kernel $H_d(\bx_1, \bx_2) = h_d(\< \bx_1, \bx_2\> / d)$,
with $h_d(\sqrt{d}\, \cdot \, ) \in L^2([-\sqrt{d},\sqrt{d}],\tau^1_{d})$,
we can associate a self-adjoint operator $\cuH_d:L^2(\S^{d-1}(\sqrt{d}))\to L^2(\S^{d-1}(\sqrt{d}))$
via
\begin{align}
\cuH_df(\bx) \equiv \int_{\S^{d-1}(\sqrt{d})} h_d(\<\bx,\bx_1\>/d)\, f(\bx_1) \, \tau_d(\de \bx_1)\, .
\end{align}
By rotational invariance,   the space $V_{k}$ of homogeneous polynomials of degree $k$ is an eigenspace of
$\cuH_d$, and we will denote the corresponding eigenvalue by $\xi_{d,k}(h_d)$. In other words
$\cuH_df(\bx) \equiv \sum_{k=0}^{\infty} \xi_{d,k}(h_d) \oproj_{k}f$.   The eigenvalues can be computed via
\begin{align}
  \xi_{d, k}(h_d) = \int_{[-\sqrt d , \sqrt d]} h_d\big(x/\sqrt{d}\big) Q_k^{(d)}(\sqrt d x) \tau^1_{d-1}(\de x)\, .
\end{align}

\subsubsection{Hermite polynomials}
\label{sec:Hermite}

The Hermite polynomials $\{\bbHe_k\}_{k\ge 0}$ form an orthogonal basis of $L^2(\reals,\gamma)$, where $\gamma(\de x) = e^{-x^2/2}\de x/\sqrt{2\pi}$ 
is the standard Gaussian measure, and $\bbHe_k$ has degree $k$. We will follow the classical normalization (here and below, expectation is with respect to
$G\sim\normal(0,1)$):
\begin{align}
\E\big\{\bbHe_j(G) \,\bbHe_k(G)\big\} = k!\, \delta_{jk}\, .
\end{align}
As a consequence, for any function $g\in L^2(\reals,\gamma)$, we have the decomposition
\begin{align}\label{eqn:sigma_He_decomposition}
g(x) = \sum_{k=0}^{\infty}\frac{\mu_k(g)}{k!}\, \bbHe_k(x)\, ,\;\;\;\;\;\; \mu_k(g) \equiv \E\big\{g(G)\, \bbHe_k(G)\}\, .
\end{align}

The Hermite polynomials can be obtained as high-dimensional limits of the Gegenbauer polynomials introduced in the previous section. Indeed, the Gegenbauer polynomials (up to a $\sqrt d$ scaling in domain) are constructed by Gram-Schmidt orthogonalization of the monomials $\{x^k\}_{k\ge 0}$ with respect to the measure 
$\tilde\tau^1_{d}$, while Hermite polynomial are obtained by Gram-Schmidt orthogonalization with respect to $\gamma$. Since $\tilde\tau^1_{d}\Rightarrow \gamma$
(here $\Rightarrow$ denotes weak convergence),
it is immediate to show that, for any fixed integer $k$, 
\begin{align}
\lim_{d \to \infty} \Coeff\{ Q_k^{(d)}( \sqrt d x) \, B(\S^{d-1}; k)^{1/2} \} = \Coeff\left\{ \frac{1}{(k!)^{1/2}}\,\bbHe_k(x) \right\}\, .\label{eq:Gegen-to-Hermite}
\end{align}
Here and below, for $P$ a polynomial, $\Coeff\{ P(x) \}$ is  the vector of the coefficients of $P$. As a consequence,
for any fixed integer $k$, we have
\begin{align}\label{eqn:mu_lambda_relationship}
\mu_k(\barsigma) = \lim_{d \to \infty} \xi_{d,k}(\barsigma) (B(\S^{d-1}; k)k!)^{1/2}, 
\end{align}
where $\mu_k(\barsigma)$ and $\xi_{d,k}(\barsigma)$ are given in Eq. (\ref{eqn:sigma_He_decomposition}) and (\ref{eqn:technical_lambda_sigma}). 

\subsection{Functions on the hypercube}\label{sec:functions_hypercube}

%Consider $\btheta \sim \Unif ( \{ -1, +1 \} )^{\otimes d}$ and $\bx \sim \Unif ( \{ -1, +1 \} )^{\otimes d}$. We will mirror the notations used for the functional space on the sphere. 
Fourier analysis on the hypercube is a well studied subject \cite{o2014analysis}. The purpose of this section is to introduce
some notations that  make the correspondence with proofs on the sphere straightforward.
For convenience, we will adopt the same notations as for their spherical case. 

\subsubsection{Fourier basis}

Denote $\Cube = \{ -1 , +1 \}^d $ the hypercube in $d$ dimension. Let us denote $\tau_{d}$ to be the uniform probability measure on $\Cube$. All the functions will be assumed to be elements of $L^2 (\Cube, \tau_d)$ (which contains all the bounded functions $f : \Cube \to \R$), with scalar product and norm denoted as $\< \cdot , \cdot \>_{L^2}$ and $\| \cdot \|_{L^2}$:
\[
\< f , g \>_{L^2} \equiv \int_{\Cube} f(\bx) g(\bx) \tau_d ( \de \bx) = \frac{1}{2^n} \sum_{\bx \in \Cube} f(\bx) g(\bx).
\]
Notice that $L^2 ( \Cube, \tau_d)$ is a $2^n$ dimensional linear space. By analogy with the spherical case we decompose $L^2 ( \Cube, \tau_d)$ as a direct sum of $d+1$ linear spaces obtained from polynomials of degree $\ell = 0, \ldots , d$
\[
L^2 ( \Cube, \tau_d) = \bigoplus_{\ell = 0}^d V_{d,\ell}.
\]

For each $\ell \in \{ 0 , \ldots , d \}$, consider the Fourier basis $\{ Y_{\ell, S}^{(d)} \}_{S \subseteq [d], |S | =\ell}$ of degree $\ell$, where for a set $S \subseteq [d]$, the basis is given by
\[
Y_{\ell, S}^{(d)} (\bx) \equiv x^S \equiv \prod_{i \in S} x_i.
\]
It is easy to verify that (notice that $x_i^k = x_i$ if $k$ is odd and  $x_i^k = 1$ if $k$ is even)
\[
\< Y_{\ell, S}^{(d)} , Y_{k, S'}^{(d)} \>_{L^2} = \E[ x^{S} \times x^{S'} ] = \delta_{\ell,k} \delta_{S,S'}. 
\]
Hence $\{ Y_{\ell, S}^{(d)} \}_{S \subseteq [d], |S | =\ell}$ form an orthonormal  basis of $V_{d,\ell}$ and 
\[
\dim (V_{d,\ell} ) = B(\Cube;\ell) = {{d}\choose{\ell}}.
\]
As above, we will omit the superscript $(d)$ in $Y_{\ell, S}^{(d)}$ when clear from the context.

\subsubsection{Hypercubic Gegenbauer}

We consider the following family of polynomials $\{ Q^{(d)}_\ell \}_{\ell = 0 , \ldots, d}$ that we will call hypercubic Gegenbauer, defined as
\[
Q^{(d)}_\ell (  \< \bx , \by \> ) = \frac{1}{B(\Cube; \ell)} \sum_{S \subseteq [d], |S| = \ell} Y_{\ell,S}^{(d)} ( \bx ) Y_{\ell,S}^{(d)} ( \by ). 
\]
Notice that the right hand side only depends on $ \< \bx , \by \>$ and therefore these polynomials are uniquely defined. In particular,
\[
\< Q_\ell^{(d)} ( \< \ones , \cdot \> ) , Q_k^{(d)} ( \< \ones , \cdot \> ) \>_{L^2} = \frac{1}{B(\Cube;k)} \delta_{\ell k}.
\]
Hence $\{ Q^{(d)}_\ell \}_{\ell = 0 , \ldots, d}$ form an orthogonal basis of $L^2 ( \{ -d , -d+2 , \ldots , d-2 ,d\}, \Tilde \tau_d^1 )$ where $\Tilde \tau_d^1$ is the distribution of $\< \ones , \bx \>$ when $\bx \sim \tau_d$, i.e., $\Tilde \tau_d^1 \sim 2 \text{Bin}(d, 1/2) - d/2$.

We have
\[
\< Q_\ell^{(d)} ( \< \bx , \cdot \> ) , Q_k^{(d)} ( \< \by , \cdot \> ) \>_{L^2} =   \frac{1}{B(\Cube;k)} Q_{k} ( \< \bx , \by \> )\delta_{\ell k} .
\]
For a function $\barsigma ( \cdot / \sqrt{d} ) \in L^2 ( \{ -d , -d+2 , \ldots , d-2 ,d\}, \Tilde \tau_d^1 )$, denote its hypercubic Gegenbauer coefficients $\xi_{d,k} ( \barsigma)$ to be
\[
\xi_{d,k} (\barsigma ) = \int_{\{-d, -d+2 , \ldots , d-2 , d\} } \barsigma(x / \sqrt{d} ) Q_k^{(d)} ( x) \Tilde \tau_d^1 (\de x).
\]

Notice that by weak convergence of $\< \ones, \bx \> / \sqrt{d}$ to the normal distribution, we have also convergence of the (rescaled) hypercubic Gegenbauer polynomials to the Hermite polynomials, i.e., for any fixed $k$, we have
\begin{align}
\lim_{d \to \infty} \Coeff\{ Q_k^{(d)}( \sqrt d x) \, B(\Cube; k)^{1/2} \} = \Coeff\left\{ \frac{1}{(k!)^{1/2}}\,\bbHe_k(x) \right\}\, .\label{eq:Hyper-Gegen-to-Hermite}
\end{align}

\subsection{Hypercontractivity of the uniform distribution on the sphere and the hypercube}
\label{app:hypercontractivity}

By Holder's inequality, we have $\| f \|_{L^p} \le \| f \|_{L^q}$ for any $f$ and any $p \le q$. The reverse inequality does not hold in general, even up to a constant. However, for some measures, the reverse inequality will hold for some sufficiently nice functions. These measures satisfy the celebrated hypercontractivity properties \cite{gross1975logarithmic, bonami1970etude, beckner1975inequalities, beckner1992sobolev}. 

\begin{lemma}[Hypercube hypercontractivity \cite{beckner1975inequalities}]\label{lem:hyper_cube} For any $\ell = \{ 0 , \ldots , d \}$ and $f_* \in L^2 ( \Cube)$ to be a degree $\ell$ polynomial, then for any integer $q\ge 2$, we have
\[
\| f_* \|_{L^q ( \Cube)}^2 \leq (q-1)^\ell \cdot \| f_* \|^2_{L^2 (\Cube)}.
\]
\end{lemma}

Besides this classical result, we will also use the following simple observation:

\begin{lemma}[Hypercube hypercontractivity for high-degree polynomials ]\label{lem:hyper_high_cube} For any $\ell = \{ 0 , \ldots, d\}$ and $f_* \in \text{span} \{ Y_S : |S| \geq d- \ell \}$ (function orthogonal to degree-$(d - \ell - 1)$ polynomials), then for any integer $q \geq 2$, we have
\[
\| f_* \|_{L^q ( \Cube)}^2 \leq (q-1)^\ell \cdot \| f_* \|^2_{L^2 (\Cube)}.
\]
\end{lemma}

\begin{proof}[Proof of Lemma \ref{lem:hyper_high_cube}]
Note that for any $S \subseteq [d]$, we have $Y_S (\bx) = \big( \prod_{i \in [d]} x_i \big) \cdot Y_{S^c} (\bx)$. Hence, we have 
\[
f_* (\bx) = \sum_{S \subseteq [d], |S| \geq d- \ell} c_S Y_S (\bx)= \big( \prod_{i \in [d]} x_i \big)\sum_{S \subseteq [d], |S| \leq  \ell} c_{[d]\setminus S]} Y_S (\bx) =: \big( \prod_{i \in [d]} x_i \big) \cdot \Tilde f_* (\bx) \, .
\]
We deduce that $\| f_* \|_{L^q (\Cube^d)} = \| \Tilde f_* \|_{L^q (\Cube^d)}$, with $\Tilde f_*$ a degree-$\ell$ polynomial. We can therefore directly apply Lemma \ref{lem:hyper_cube}.
\end{proof}

Finally, we have the following similar hypercontractivity property on the sphere:

\begin{lemma}[Spherical hypercontractivity \cite{beckner1992sobolev}]\label{lem:hypercontractivity_sphere}
For any $\ell \in \N$ and $f_* \in L^2(\S^{d-1})$ to be a degree $\ell$ polynomial, for any $q \ge 2$, we have 
\[
\| f_* \|_{L^q(\S^{d-1})}^2 \le (q - 1)^\ell \cdot \| f_* \|_{L^2(\S^{d-1})}^2. 
\]
\end{lemma}

\end{document}